\documentclass[english,11pt,a4paper,leqno]{amsart}
\usepackage{amsmath,amstext, amsthm, amssymb}
\usepackage{appendix}
\usepackage{color}
\definecolor{liens}{rgb}{1,0,0}
\usepackage[colorlinks=true, linkcolor=blue, 
hyperfootnotes=true,citecolor=blue,urlcolor=black]{hyperref}

\usepackage{latexsym, amsthm, amsfonts, amsmath,amssymb,graphicx,xcolor}
\input xy
\xyoption{all}
\usepackage[latin1]{inputenc} %%% Pour les accents
\usepackage[T1]{fontenc} %%% pour la césure des mots accentués
\usepackage[english]{babel} %%% toujours en dernier
\newtheorem*{propointro}{Proposition}
\newtheorem{theo}{Theorem}[section]
\newtheorem*{theointro}{Theorem}
\newtheorem{lem}[theo]{Lemma}
\newtheorem{propo}[theo]{Proposition}
\newtheorem{coro}[theo]{Corollary}

\theoremstyle{definition}
\newtheorem{defi}[theo]{Definition}

\theoremstyle{remark}
\newtheorem{rem}[theo]{Remark}
\newtheorem{ex}[theo]{Example}
\makeatletter

\@addtoreset{equation}{section}
\makeatother
\def\pd{(\sq,\delta)}

\def\de{\delta}
\def\Kqpform{\mathbb{C}((z^*))}

\def\R{\mathbb{R}}
\def\Z{\mathbb{Z}}
\def\C{\mathbb{C}}
\def\N{\mathbb{N}}
\def\Q{\mathbb{Q}}
\def\cL{\mathcal{L}}

\def\n{\eta}

\def\a{\alpha}

\def\d{\delta}

\def\b{\beta}

\def\n'{\nu}
\def\d{\delta}
\def\l{\lambda}

\def\k{\kappa}

\def\qpr{\mathbf{q}}

\def\L{\Lambda}

\def\Kqp{{ \mathbb{C}(z^*)}}

\def\sq {{\sigma_{q}}}

\def\Aut{\mathrm{Aut}}
\def\Gal {\mathrm{Gal}}
\def\GL {\mathrm{GL}}
\def\SL {\mathrm{SL}}

\def\cQ{\mathcal{Q}}
\def\Galdelta {\mathrm{Gal}^\delta}

\newcommand{\circKol}{\circ_{\delta}}
\newcommand{\derKol}{der_{\delta}}
\def\tq{\qpr}

\def\K{\mathbf{K}}
\def\Kq{\C(z)}
\def\Kqbar{\overline{\C(z)}}
\def\Kqform{\mathbb{C}((z^*))}

\def\k{\mathbf{k}}
\def\Const{\mathbf{C}}

\def\L{\mathbf{L}}
\def\Lqbar{\mathbf{L}}
 \def\sp{{\sigma_{\qpr}}}
\def\ssp{(\sigma_q,\sp)}
\def\cM{\mathcal{M}}
\def\V{\bold{V}}
\def\Alg{\mathrm{Alg}}
\def\Sets{\mathrm{Sets}}
\def\spGal{ {\mathrm{Gal}^{\sp}}}
\def\sptrdeg{\sp$-$\mathrm{trdeg}}
\def\Groups{\mathrm{Groups}}
\def\spdim{\sp$-$\mathrm{dim}}
\def\Hom{\mathrm{Hom}}
\def\M{\mathbf{M}}
\def\SO{\mathrm{SO}}
\def\Sp{\mathrm{Sp}}
\newtheorem{hypotheses}[theo]{Assumptions}

\newcommand{\crochets}[1]{\left[#1\right]}
\newcommand{\divi}{\operatorname{div}}
\newcommand{\qdivi}{\operatorname{div}_{q}}
\usepackage{marginnote}
\usepackage{pifont}

\begin{document}

\sloppy

\title{Functional relations for solutions of $q$-difference equations}
\author{Thomas Dreyfus}
\address{Institut de Recherche Math\'ematique Avanc\'ee, U.M.R. 7501 Universit\'e de Strasbourg et C.N.R.S. 7, rue Ren\'e Descartes 67084 Strasbourg, FRANCE}
\email{dreyfus@math.unistra.fr}
\author{Charlotte Hardouin}
\address{Universit\'e Paul Sabatier - Institut de Math\'ematiques de Toulouse, 118 route de Narbonne, 31062 Toulouse.}
\email{hardouin@math.univ-toulouse.fr}
\author{Julien Roques}
\address{Univ Lyon, Universit\'e Claude Bernard Lyon 1, CNRS UMR 5208, Institut Camille Jordan, 43 blvd. du 11 novembre 1918, F-69622 Villeurbanne cedex, France}
\email{roques@math.univ-lyon1.fr}


\thanks{The second author would like to thank the ANR-11-LABX-0040-CIMI within
the program ANR-11-IDEX-0002-0 for its partial support. The second authors work is also supported by ANR Iso-Galois. This project has received funding from the European Research Council (ERC) under the European Union's Horizon 2020 research and innovation programme under the Grant Agreement No 648132.}
\date{\today}

\begin{abstract} 
In this paper, we study the algebraic relations satisfied by the solutions of $q$-difference equations and their transforms
with respect to an auxiliary operator. Our main tools are the parametrized Galois theories developed in \cite{HS}
and \cite{OvWib}. The first part of this paper is concerned with the case where the auxiliary operator  is a derivation, whereas
the second part deals with a $\qpr$-difference operator.  In both cases, we give criteria  to guarantee the  algebraic independence of  a series,  solution of a $q$-difference equation, with either its successive derivatives or its $\qpr$-transforms.  We apply our results to  $q$-hypergeometric series.
\end{abstract} 
\maketitle
\setcounter{tocdepth}{1}
\tableofcontents
%\dedicatory{u}

\section*{Introduction}

The study of the differential transcendence of special functions is an old and difficult problem. Only very recently, systematic methods to tackle this kind of question were discovered. Indeed, after the seminal work of Cassidy and Singer in \cite{CS}, several authors developed Galoisian approaches in order to study the differential or difference relations between solutions of linear differential or difference equations; see {\it e.g.} Hardouin and Singer \cite{HS}, Di Vizio, Hardouin and Wibmer \cite{DVHaWib1,DVHaWib2} and  Ovchinnikov and Wibmer \cite{OvWib}. For instance, this led to a short and comprehensive proof of H\"older's theorem asserting the differential transcendence of Euler's Gamma function; see \cite{HS}. Also, this enabled the authors of the present paper to study the differential transcendence of generating series issued from the theory of automatic sequences, such as the Baum-Sweet or the Rudin-Shapiro generating series, which turn out to satisfy linear Mahler equations; see \cite{DHR}. In the present paper, we take a close look at the differential algebraic relations satisfied by solutions of linear $q$-difference equations.Very little was known about the differential or difference algebraic relations between these solutions. The first results in this direction, due to B\'ezivin (\cite{BeBo}) and Ramis (\cite{R92}), assert that a 
non rational solution of a linear $q$-difference equation does not satisfy a linear dependence relation  with its successive transforms with respect to a derivation or  a $\qpr$-difference operator provided that $\qpr$ is multiplicatively independent of $q$, i.e., $\log(q/ \qpr)\notin \Q$. Later, the parametrized   Galois theories developed by Hardouin and Singer in  \cite{HS} and Ovchinnikov and Wibmer in \cite{OvWib} allowed their authors to give  complete criteria for the differential or difference transcendence  for the solutions of  $q$-difference equations of order one or of systems of such equations.  For irreducible $q$-difference equations, the results of \cite{HS} allowed to characterize the dependencies of the solutions   via the existence of a linear compatible equation in the auxiliary operator.   Our  paper is mainly  concerned with  $q$-difference equations of order  greater than two and combines the results of B\'{e}zivin and Ramis with the parametrized Galois theories mentioned above. This paper is divided in two parts. \vskip 5 pt
$$\ast\ast\ast$$
In the first part, we study the algebraic  relations between the successive derivatives of the solutions of linear $q$-difference equations. These relations are encoded by the parametrized difference Galois groups introduced by Hardouin and Singer in \cite{HS}. The basic (and, at first sight, quite optimistic) question is: if we know the algebraic relations between the solutions, what can be said about the differential algebraic relations? In Galoisian terms, an equivalent question is:  if we know what the non parametrized  difference Galois group is, what can be said about the parameterized difference Galois group?
 Our answer reads as follows. Consider a linear $q$-difference equation
\begin{equation}\label{equa q diff intro}\tag{1}
a_{n}(z)y(q^{n}z)+a_{n-1}(z)y(q^{n-1}z)+\cdots+a_{0}(z)y(z)=0
\end{equation}
where $a_{0}(z),\ldots,a_{n-1}(z),a_{n}(z) \in \C(z)$, $a_{0}(z)a_{n}(z)\neq 0$, and where $q$ is a non zero complex number with $|q|\neq 1$. Let $G$ be the difference Galois group of this equation. This is an algebraic subgroup of $\GL_{n}(\C)$ which reflects the algebraic relations  between the solutions of the equation. Let $G^{\delta}$ be its parametrized difference Galois group. This is a differential algebraic subgroup of $\GL_{n}(\widetilde{\C})$, where $\widetilde{\C}$ is a differential closure of $\C$ , {\it i.e.}, it is a group of matrices whose entries are the zeros of  differential algebraic polynomials with coefficients in $\widetilde{\C}$. As mentioned above, this parametrized difference Galois group reflects the differential algebraic relations between the solutions of the equation. The main result of the first part of the present paper, see Theorem~\ref{theo3}, can be stated as follows. The technical assumption on the Galois group in the following theorem could be roughly rephrazed as the assumption that the Galois group is ``sufficiently big'', which means that there are few algebraic relations among the solutions of the $q$-difference equation.

\begin{theointro}
Assume that the derived subgroup $G^{\circ,der}$ of the neutral component $G^{\circ}$ of $G$ is an irreducible almost simple algebraic subgroup of $\SL_{n}(\C)$. Then, $G^{\delta}$ is a subgroup of $G(\widetilde{\C})$ containing $G^{\circ, der}(\widetilde{\C})$.
\end{theointro}

Since $G^{\delta}$ is sufficiently big, we have for instance the following consequences on the solutions, see {Corollary \ref{coro1}}. 

\begin{propointro}
Let $h(z)$ be a non zero Laurent  series solution of \eqref{equa q diff intro}. Let $G$ be the difference Galois group of \eqref{equa q diff intro} and consider the derived subgroup $G^{\circ,der}$ of the neutral component $G^{\circ}$ of $G$.
\begin{itemize}
\item  Assume that $n\geq 2$ and $G^{\circ, der}=\SL_{n}(\C)$.
Then, $h(z),\dots,h(q^{n-1}z)$ are differentially algebraically independent over $\C(z)$.
\item Assume that $n$ is even and $G^{\circ, der}=\mathrm{Sp}_{n}(\C)$.  
Then, the series $h(z),\dots,h(q^{n-1}z)$ are differentially algebraically independent over $\C(z)$.
\item Assume that $n\geq 3$ and $G^{\circ, der}=\mathrm{SO}_{n}(\C)$.
Then, $h(z),\dots,h(q^{n-2}z)$ are differentially algebraically independent over $\C(z)$. 
\end{itemize}
\end{propointro}

An important family of $q$-difference equations is given by the generalized $q$-hypergeometric equations. Assume that $0< |q|<1$. Let us fix $n\geq s$, two integers, let $\underline{a}=(a_{1},\dots,a_{n})\in (q^{\R})^{n}$, $\underline{b}=(b_{1},\dots,b_{s})\in (q^{\R}\setminus q^{-\N})^{s}$, $\l\in \C^{\times}$, and define $\sq (f(z))=f(qz)$. Let us consider the generalized $q$-hypergeometric operator:
$$
z\l \displaystyle \prod_{i=1}^{n} (a_{i}\sq -1)-\displaystyle \prod_{j=1}^{s} \left(\frac{b_{j}}{q}\sq -1\right).
$$
When $b_{1}=q$, this operator admits as solution the $q$-hypergeometric series:
$$
\begin{array}{lll}
_{n}\Phi_{s}(\underline{a},\underline{b},\l,q;z)&=&\displaystyle \sum_{m=0}^{\infty}\dfrac{(\underline{a};q)_{m}}{(\underline{b};q)_{m}}\l^{m}z^{m}\\
&=&\displaystyle \sum_{m=0}^{\infty} \frac{\displaystyle \prod_{i=1}^{n} (1-a_{i})(1-a_{i}q)\dots(1-a_{i}q^{m-1})}{\displaystyle \prod_{j=1}^{s}(1-b_{j})(1-b_{j}q)\dots(1-b_{j}q^{m-1})}\l^{m}z^{m} .
\end{array} $$ 
Using \cite{Ro08,Ro11,Ro12}, we see that, in many cases, the algebraic group $G^{\circ,der}$ is either $\SL_{n}(\C)$, $\mathrm{SO}_{n}(\C)$ or the symplectic group $\mathrm{Sp}_{n}(\C)$ (for $n$ even).  
Therefore, the above results ensure that, in many cases, the $q$-hypergeometric series
are differentially transcendental. To the best of our knowledge, the only previously known result in this direction was due to Hardouin and Singer \cite{HS} about some $q$-hypergeometric equations of order two. \par 

The first part of the present paper is organized as follows. Section $ \ref{sec1}$ contains reminders about difference Galois theory. Section $ \ref{sec:parampv}$ contains reminders and complements about the parametrized difference Galois theory developed in \cite{HS}. In particular, we study the notion of  projective isomonodromy. Roughly speaking, we show that if the difference Galois group of \eqref{equa q diff intro} is large, then we have two possibilities: either the parametrized difference Galois is large, or any solution of \eqref{equa q diff intro} satisfies a linear differential equation. In Section  $\ref{sec hyptr}$, we prove the above Theorem by showing that the latter case in the previous alternative does not occur. In Section~$ \ref{sec applications}$, we apply our results to the $q$-hypergeometric equations.

 \vskip 5 pt
$$\ast\ast\ast$$

In the second part of the paper, we study the algebraic  $\qpr$-difference equations satisfied by the solutions of the equation (\ref{equa q diff intro}), where $\qpr$ is a non zero complex number with $|\qpr|\neq 1$ such that $q$ and $\qpr$ are multiplicatively independent, i.e., $\log(q/ \qpr)\notin \Q$. These relations are reflected by the parametrized difference Galois group introduced by Ovchinnikov and Wibmer in \cite{OvWib}. Our main results are formally similar to those mentioned above. However, the proofs are more involved in this case because the parametrized difference Galois groups are difference affine algebraic groups. These are more subtle than the differential algebraic groups.  We obtain the following result, see Corollary~\ref{coro2}.

\begin{theointro}

Let $A \in \GL_n(\C(z))$ and let $G$ be the difference Galois group of the $q$-difference system $\sq(Y)=AY$ over the $\sq$-field $\C(z)$.  {Assume that one of the following holds 
\begin{itemize}
\item    $n\geq 2$ and $G^{\circ, der}=\SL_{n}(\C)$;
\item  $n\geq 3$ and $G^{\circ, der}=\mathrm{SO}_{n}(\C)$;
\item  $n$ is even and $G^{\circ, der}=\mathrm{Sp}_{n}(\C)$.
\end{itemize}}

If there exists  $f \in \bigcup_{j=1}^{\infty} \C((z^{1/j}))$ such that $Y_0=(f,\sq(f),\dots,\sq^{n-1}(f))^t$ is a vector solution of $\sq(Y)=AY$, then $f$ is $\sp$-transcendental  over $ \bigcup_{j=1}^\infty \C(z^{1/j})$.

\end{theointro}

The second part of the paper is organized as follows. Section  $ \ref{sec:parampvdiscr}$ contains reminders and complements about the parametrized difference Galois theory developed by Ovchinnikov and Wibmer in \cite{OvWib}.
Then,  we  split our study in two cases, depending on the $\sp$-transcendence of the determinant of the fundamental matrix of solutions. Since the latter is solution of an order one $q$-difference equation, we have to compute the parametrized difference Galois group of such equations. This is the goal of Section $ \ref{sec6}$. Then, in Section $ \ref{sec7}$, we deal with projective isomonodromy, and we find basically the same type of result as in the first part. If the difference Galois group of \eqref{equa q diff intro} is large, then we have two possibilities: either the parametrized difference Galois group is large, or any solution of \eqref{equa q diff intro} satisfies a linear $\qpr$-difference equation. In Section $ \ref{sec8}$, we prove that the latter case does not occur when the determinant of the fundamental matrix of solutions is $\sp$-algebraic. Hopefully, in all cases, we are able to prove the $\sp$-transcendence  of Laurent series solutions of \eqref{equa q diff intro}. We apply our main results to the $q$-hypergeometric series in Section~$\ref{sec9}$. 

\vskip 5 pt 

\noindent {\bf General conventions.} All rings are commutative with identity and contain the field of rational numbers. In particular, all fields are of characteristic zero. 
%%%%%%%%%%%%%%%%%%%%%%%%%%%%%%%%%%%%%%%%%%%%%%%%%%%%%%%%%%%%%%%%%%%%%%%%%%%

\vskip 5 pt 

\noindent {\bf Acknowledgement.} We thank M. Singer for pointing out some inaccuracies in a previous version of the paper. {We would like to thank the anonymous referee for her/his   careful reading and helpful comments.}

\part{Differential relations for solutions of $q$-difference equations} 

\section{Galois theories of difference equations}\label{sec1}
%%%%%%%%%%%%%%%%%%%%%%%%%%%%%%%%%%%%%%%%%%%%%%%%%%%%%%%%%%%%%%%%%%%%%%%%%%%%%%

%%%%%%%%%%%%%%%%%%%%%%%%%%%%%%%%%%%%%%%%%%%%%%%%%%%%%%%
%%%%%%%%%%%%%%%%%%%%%%%%%%%%%%%%%%%%%%%%%%%%%%%%%%%%
\subsection{Difference, differential and difference differential algebra}

In this paper we will use standard notions notations and results from difference and differential algebra. Some of them are recalled below and all of them can be found in \cite[Section 2 and Section 6.2]{HS} and in the references therein (notably \cite{Cohn:difference,VdPS,VdPS97}). 

\subsubsection{Difference differential algebra}
%%%%%%%%%%%%%%%%%%%%%%%%%%%%%%%%%%%%%%%%%%%%%%%%%%
%%%%%%%%%%%%%%%%%%%%%%%%%%%%%%%%%%%%%%%%%%%%%%%%%

A $(\sq, \d)$-ring $(R,\sq,\d)$ is a  ring $R$ endowed with  a ring automorphism $\sq$ of $R$ and a derivation $\delta$ of $R$ commuting with $\sq$. If there is no possible confusion, we  write  $R$ instead of $(R,\sq, \d)$.

In the sequel we  use the notions of $\pd$-ideals, $\pd$-morphisms, $\pd$-algebras, $\pd$-fields, {\it etc}. We will not recall here these definitions but we would like to mention as a general convention that the operator predicate indicates that the algebraic structure of the attribute is compatible with the operators. For instance, a $\pd$-ideal is an ideal setwise invariant by $\sq$ and $\d$. We refer to \cite[Section~2 and Section~6.2]{HS} for more details.

For any $(\sq, \d)$-ring $R$, we denote by $R^{\sq}$ and by $R^{\d}$ the rings of $\sq$ and $\d$ constants respectively of the $\pd$-ring $R$, {\it i.e.}, 
$$
R^{\sq}=\{c \in R \ | \ \sq(c)=c\} \text{ and } R^\delta=\{c \in R \ | \ \delta(c)=0\}.
$$
If $R^{\sq}$ (resp. $R^{\d}$) is a field, it is called the field of  $\sq$-constants (resp. $\d$-constants).

\subsubsection{Differential algebra}\label{section diff alg}
%%%%%%%%%%%%%%%%%%%%%%%%%%%%%%%%%%%%%%%%%%%%%%%%%%%%%%%%%%%%%%%%%%%%%%%%%%%%%%%%%%%%%%%%%%%%%%%%%%%%

If $\sq=\mathrm{Id}$, any $(\sq, \d)$-attribute will be called a $\d$-attribute. For instance, a $\delta$-ring $R$ is a ring $R$  endowed with a derivation $\d:R\rightarrow R$. 

Let $\K$ be a $\delta$-field. Let  $R$ be a $\K$-$\delta$-algebra and let 
 $a_1,\dots,a_n \in R$. We denote by $\K\{a_1,\dots,a_n\}_{\d}$ the $\K$-$\delta$-subalgebra  of $R$ generated by   $a_1,\dots,a_n$.  
If $R$ is moreover a field, we denote by $\K\langle a_1,\dots,a_n\rangle_{\d}$ the $\K$-$\delta$-subfield of $R$ generated by  $a_1,\dots,a_n$. We denote by $\K\{y_1,\ldots,y_n\}_{\d}$ the $\K$-$\delta$-algebra of $\delta$-polynomials in the differential indeterminates $y_1,\ldots,y_n$ and with coefficients in the $\d$-field $\K$; it is the $\K$-algebra of polynomials with coefficients in $\K$ and in the indeterminates $\delta^j y_i$ with $j \geq 0$ and $1\le i\le n$ (we emphasize that $\delta^j y_i$ is simply a notation for indeterminates) endowed with the unique derivation extending the derivation of $\K$ and such that $\delta(\delta^j y_i)=\delta^{j+1} y_i$.

Let $R$ be a $\K$-$\delta$-algebra and let $a_1, \dots,a_n \in R$. If there exists a nonzero $P \in \K\{y_1,\ldots,y_n\}_{\d}$ such that 
$P(a_1,\dots,a_n)=0$, then we say that  $a_1,\dots,a_n$ are 
$\d$-algebraically dependent over $\K$. Otherwise, we say that $a_1,\dots,a_n$ are $\d$-transcendental over $\K$, or $\d$-algebraically independent over $\K$. 

 A $\delta$-field $\k$ is called $\d$-closed if, for every set of $\delta$-polynomials $\mathcal F$,  the system of $\d$-equations $\mathcal F=0$ has a solution in some $\delta$-field extension of $\k$ if and only if it has a solution in $\k$. Note that the field of $\delta$-constants $\k^{\delta}$ of any $\d$-closed field $\k$ is algebraically closed. A fundamental fact is that any $\d$-field $\k$ is contained in a $\d$-closed field. In fact, for any such $\k$ there is a $\d$-closed field $\widetilde{\k}$
containing $\k$ such that for any $\d$-closed field $K$ containing $\k$, there is a $\d$-$\k$-isomorphism of $\widetilde{\k}$ into $K$.
 Moreover, if $\k^\delta$ is algebraically closed then  $\widetilde{\k}^{\delta}=\k^\delta$. We refer to \cite{kolchin1974constrained} for more details.\par 

From now on, we consider a $\d$-closed field $\k$. A subset $W \subset \k^{n}$ is Kolchin-closed (or $\delta$-closed) if there exists ${S \subset \k\{y_{1},\dots,y_{n}\}_{\d}}$ such that
$$
W=
\left\{ a \in \k^n\:|\:  f(a)=0 \mbox{ for all } f \in S \right\}.
$$
The Kolchin-closed subsets of $\k^{n}$ are the closed sets of a topology on $\k^{n}$, called the Kolchin topology. 
The Kolchin-closure of $W \subset \k^{n}$ is the closure of $W$ in $\k^{n}$ for the Kolchin topology.

Following Cassidy in \cite[Chapter~II, Section~1, Page~905]{C72}\label{def:LDAG}, we say that a subgroup $G \subset \GL_{n}(\k) \subset \k^{n \times n}$ is a linear $\d$-algebraic group if $G$ is the intersection
of a Kolchin-closed subset of $\k^{n \times n}$ (identified with $\k^{n^{2}}$) with $\GL_n(\k)$. 

For a $\delta$-subfield $F$ of $\k$ , we say that 
 a linear $\de$-algebraic group $G \subset \GL_{n}(\k)$ is defined over $F$ if $G$ is the zero set of $\delta$-polynomials with coefficients in $F$. For  $G \subset \GL_{n}(\k)$ a linear $\de$-algebraic group defined over $F$ and $L$ a $\delta$-field extension of $F$, we denote by  $G(L)$ the set of $L$-points of $G$.

A $\delta$-closed subgroup, or $\delta$-subgroup for short, of a linear $\d$-algebraic group is a subgroup which  is Kolchin-closed. The Zariski-closure of a  linear $\d$-algebraic group $G\subset \GL_n(\k)$ is denoted by $\overline{G}$ and is a linear algebraic group defined over $\k$.

\subsubsection{Difference algebra}

If $\d=0$, any $(\sq, \d)$-attribute will be called a $\sq$-attribute. For instance a $\sq$-ring $R$ is a ring $R$  endowed with a ring automorphism $\sq:R\rightarrow R$. 

%%%%%%%%%%%%%%%%%%%%%%%%%%%%%%%%%%%%%%%%%%%%%%%%%%%%%%%
%%%%%%%%%%%%%%%%%%%%%%%%%%%%%%%%%%%%%%%%%%%%%%%%%%%%
\subsection{Difference and Parametrized Difference  Galois theories}\label{sec:parampv}
%%%%%%%%%%%%%%%%%%%%%%%%%%%%%%%%%%%%%%%%%%%%%%%%%%
%%%%%%%%%%%%%%%%%%%%%%%%%%%%%%%%%%%%%%%%%%%%%%%%%
\subsubsection{Parametrized Difference  Galois theory}\label{sec:paramgalois}
For details on what follows, we refer to  \cite{HS}. 

Let $\K$ be a $(\sq,\d)$-field such that $\k=\K^{\sq}$ is a $\d$-closed field and consider a linear difference system 
\begin{equation}\label{eq9}
\sq (Y)=AY
\end{equation} 
with $A \in \GL_{n}(\K)$ for some integer $n \geq 1$.

By  \cite[$\S$ 6.2.1]{HS}, there exists a $\K$-$(\sq,\d)$-algebra $S$  such that 
\begin{itemize}
\item[1)] there exists $U \in \GL_{n}(S)$ such that $\sq (U)=AU$ (such a $U$ is called a
fundamental matrix of solutions of (\ref{eq9}));
\item[2)] $S$ is generated, as $\K$-$\de$-algebra, by the entries of $U$ and $\det(U)^{-1}$;
\item[3)] the only $(\sq,\d)$-ideals of $S$ are $\{0\}$ and $S$.
\end{itemize}
Such a $S$ is unique up to isomorphism of $\K$-$(\sq,\d)$-algebras and is called a $(\sq,\d)$-Picard-Vessiot ring, or $(\sq,\d)$-PV ring for short, for (\ref{eq9})  over $\K$. A $(\sq,\d)$-PV ring is not always an integral domain but it is a direct sum of integral domains stable by $\delta$ and transitively permuted by $\sq$. The total quotient ring $\cQ_S$ of $S$ has a natural structure of $S$-$(\sq,\delta)$-algebra and is called a total $(\sq,\d)$-PV ring for \eqref{eq9} over $\K$. We have ${\cQ_S}^{\sq}=\k$. 

The $(\sq,\d)$-Galois group $\Gal^{\d}(\cQ_S/\K)$ of $S$ over $\K$ 
is the group of 
$\K$-$\pd$-automorphisms of $\cQ_{S}$, {\it i.e.}, 
$$
\Gal^{\d}(\cQ_S/\K)=\{ \phi \in \mathrm{Aut}(\cQ_S/\K) \ | \ \sq\circ\phi=\phi\circ \sq \text{ and } \d\circ\phi=\phi\circ \d \}.
$$

According to \cite[Proposition~6.18]{HS}, for any $\phi \in \Gal^{\d}(\cQ_S/\K)$, there exists a unique $C(\phi) \in \GL_{n}(\k)$ such that $\phi(U)=UC(\phi)$ and the faithful representation
\begin{eqnarray*}
 \rho_U : \Gal^\de(\cQ_S/\K) & \rightarrow & \GL_{n}(\k) \\ 
 \phi & \mapsto & C(\phi)
\end{eqnarray*}
identifies  $\Gal^\de(\cQ_S/\K)$ with a $\d$-closed subgroup of $\GL_{n}(\k)$. \par 

\subsubsection{Difference Galois theory}\label{sec:classicalgaloisgroup}
If the derivation $\delta$ is always considered to be trivial, a $(\sq,\d)$-PV ring $R$ for \eqref{eq9} over $\K$ will be simply called a Picard-Vessiot ring, or PV ring for short.  The corresponding total $(\sq,\d)$-PV ring $\cQ_R$ will be simply called a total Picard-Vessiot ring, or total PV ring for short. The corresponding $(\sq,\d)$-Galois group will be simply called the difference Galois group and denoted by $\Gal(\cQ_R/\K)$. The faithful representation $\rho_{U}$ identifies $\Gal(\cQ_R/\K)$ with a linear algebraic subgroup of $\GL_{n}(\k)$. We refer to \cite[Theorem 1.13]{VdPS97} for more informations.

\subsubsection{From parametrized to non parametrized difference Galois groups}

Let $S$ be a $(\sq,\d)$-PV ring over $\K$ for \eqref{eq9} and let $U \in \GL_n(S)$ be a fundamental matrix of solutions.
The $\K$-$\sq$-algebra $R$ generated by the entries of $U$ and $\det(U)^{-1}$ is a PV ring
for \eqref{eq9} over $\K$ and we have $\cQ_R \subset \cQ_S$. One can identify  $\Gal^{\d}(\cQ_S/\K)$ with a subgroup of $\Gal(\cQ_R/\K)$ by restricting the elements of  $\Gal^{\d}(\cQ_S/\K)$ to $\cQ_{R}$; with this identification, we have the following result:

\begin{propo}[{\cite[Proposition 2.8]{HS}}]\label{propo:zarclosurePPvgaloisgroup}
The group $\Gal^{\d}(\cQ_S/\K)$ is a Zariski-dense subgroup of  $\Gal(\cQ_R/\K)$.
\end{propo}

\subsection{A technical result}

In order to use the Galois theory exposed in Section~\ref{sec:paramgalois} above, we  need to work with a base $(\sq, \d)$-field $\K$ such that $\k=\K^{\sq}$ is a $\d$-closed field. 
Unfortunately, most of the function fields arising naturally as base $(\sq, \d)$-fields do not satisfy this condition.  The following Lemma will be used in order to remedy this problem.

 \begin{lem}[{\cite[Lemma 2.3]{DHR}}]\label{lem:extconst}
  Let $F$ be a $\pd$-field such that ${\k=F^{\sq}}$ is algebraically closed. Let $\widetilde{\k}$ be a $\d$-closed field 
  containing $\k$. Then, the ring $\widetilde{\k} \otimes_\k F$ is an integral domain whose
  fraction field $\K$ is a $\pd$-field extension of $F$ such that $\K^{\sq}=\widetilde{\k}$.
  \end{lem}

%%%%%%%%%%%%%%%%%%%%%%%%%%%%%%%%%%%%%%%%%%%%%
%%%%%%%%%%%%%%%%%%%%%%%%%%%%%%%%%%%%%%%%%%%%%%%
%%%%%%%%%%%%%%%%%%%%%%%%%%%%%%%%%%%%%%%%%%%%%%%%%%%%%%%%%%%%%%%%%%%%%%
%%%%%%%%%%%%%%%%%%%%%%%%%%%%%%%%%%%%%%%%%%%%%%%%%%%%%%%%%%%%%%%%%%%%%%%
%%%
\subsection{Transcendence results}\label{sec24}
Let $\K$ be a $(\sq,\d)$-field such that $\k=\K^{\sq}$ is $\d$-closed. Let $S$ be a $(\sq,\d)$-PV ring for (\ref{eq9}) over $\K$, let $\cQ_S$ be the corresponding total $(\sq,\d)$-PV ring,  and let $\Gal^{\d}(\cQ_S/\K)$ be the associated $(\sq,\d)$-Galois group.  

In the following result, we denote by  $\mathrm{SO}_n(\k)$ the special orthogonal group 
$$
{\mathrm{SO}_{n}(\k)=\{ C\in \SL_{n}(\k)| C^{t}C=\mathrm{I}_{n}\}}
$$ 
and, if $n$ is even, by $\mathrm{Sp}_n(\k)$ the symplectic group 
$$\mathrm{Sp}_n(\k)=\{ C\in \GL_{n}(\k)|C^{t}JC=J \} \text{ where } J=\begin{pmatrix}0&\mathrm{I}_{n/2}\\-\mathrm{I}_{n/2}&0\end{pmatrix}.$$

\begin{propo}\label{propo:transcontinu} 
Let $U\in \GL_n(S)$ be a fundamental matrix of solutions of (\ref{eq9}) and 
let $u$ be a row (resp. column) vector  of $U$. If $n \geq 2$ and if there exists
$P \in \GL_n(\k)$ such that the image of $\Gal^\d(\cQ_S/\K)$ by the representation $\rho_{UP}$  contains
\begin{itemize}
\item $\SL_n(\k)$ or $\Sp_n(\k)$ then the entries of $u$ are $\delta$-algebraically independent over $\K$;
\item $\SO_n(\k)$  then any  $n-1$  distinct elements among the entries of $u$ are   $\delta$-algebraically independent over $\K$.
\end{itemize}
\end{propo}

\begin{proof}
For the sake of clarity, we assume that $u=(u_{1},\ldots,u_{n})$ is the first row of $U$. The proof in the other cases is similar.  
\begin{center}\boxed{\SL_{n}(\k)\hbox{-case.}} \end{center}
We first explain the strategy of the proof in the $\SL_{n}(\k)$-case. Let ${X=(X_{i,j})_{1 \leq i,j \leq n}}$ be $\delta$-indeterminates. Let $\mathfrak{I}$ be the kernel of the unique morphism of $\K$-$\delta$-algebras 
$
\K\{X, \frac{1}{\det(X)} \}_{\d} \rightarrow S
$ 
such that $X \mapsto U$. We denote by $(x_{1},\ldots,x_{n})=(X_{1,1},\ldots,X_{1,n})$ the first row of $X$.  The $\d$-algebraic relations with coefficients in $\K$ between $u_{1},\ldots,u_{n}$ correspond to the elements of $\mathfrak{I} \cap  \K\{x_1,\dots,x_{n}\}_{\d}$. So everything amounts to prove that $\mathfrak{I} \cap  \K\{x_1,\dots,x_{n}\}_{\d} = \{0\}$. In order to prove this, we will relate $\mathfrak{I}$ to the ideal defining the $\d$-algebraic group $\Gal^\d(\cQ_S/\K)$. Such a relation follows from the fact that the $(\sq,\d)$-PV ring $S$ is the coordinate ring of a  $\Gal^\d(\cQ_S/\K)$-torsor over $\K$.  

We shall now give the details of the proof, still in the $\SL_{n}(\k)$-case. As above, we let $\mathfrak{I}$ be the kernel of the unique morphism of $\K$-$\delta$-algebras 
$
\varphi : \K\{X, \frac{1}{\det(X)} \}_{\d} \rightarrow S
$ 
such that $X \mapsto U$ and we denote by $\mathcal{V}$ the $\delta$-algebraic variety over $\K$ defined by $\mathfrak{I}$.  On the other hand, we let $G$ be the image of $\Gal^\d(\cQ_S/\K)$ by the representation $\rho_{U}$, we let ${\mathfrak{L}}$ be the $\delta$-ideal of $\k\{X, \frac{1}{\det(X)} \}_{\d}$ of the equations of $G$ and we let $\mathcal{G}$ be the $\delta$-algebraic variety over $\K$ defined by ${\mathfrak{L}}$; in other words, $\mathcal{G}$ is the $\delta$-linear algebraic group over $\K$ obtained from $G$ by extension of scalars from $\k$ to $\K$. Both $\mathcal{V}$ and $\mathcal{G}$ can be seen in $\GL_{n,\K}$.   The following map is well-defined and makes $\mathcal{V}$ a $\mathcal{G}$-torsor over $\K$ (this is the content of \cite[Proposition 6.24]{HS}): 
\begin{eqnarray*}
 \mathcal{V} \times_{\K} \mathcal{G} & \rightarrow & \mathcal{V} \times_{\K} \mathcal{V} \\ 
 (v,M) & \mapsto & (v,vM).
 \end{eqnarray*}
 So, if we let $\widetilde{\K}$ be a $\delta$-closure of $\K$, we have $\mathcal{V}(\widetilde{\K}) \neq \emptyset$ and  $\mathcal{V}(\widetilde{\K})=Z_{0} \cdot \mathcal{G}(\widetilde{\K})$ for any $Z_{0} \in \mathcal{V}(\widetilde{\K})$. In terms of the $\delta$-ideals $\widetilde{\mathfrak{I}}$ and $\widetilde{{\mathfrak{L}}}$ of $\widetilde{\K}\{X, \frac{1}{\det(X)} \}_{\d} $ defining $\mathcal{V}(\widetilde{\K})$ and $\mathcal{G}(\widetilde{\K})$ respectively, the equality $\mathcal{V}(\widetilde{\K})=Z_{0} \cdot \mathcal{G}(\widetilde{\K})$ is equivalent to
$$\widetilde{\mathfrak{I}}= \left\{ \widetilde{P}(Z_0^{-1}X) \ \left| \ \widetilde{P} \in \widetilde{{\mathfrak{L}}}\right\}\right. .$$
Since the image of   $\Gal^\d(\cQ_S/\K)$ by the representation $\rho_{UP}$  contains $\SL_n(\k)$, we see that $G$ contains ${H=P \SL_n(\k) P^{-1}(= \SL_n(\k))}$.  So, $\mathcal{G}(\widetilde{\K})$ contains ${\widetilde{H}=\SL_n(\widetilde{\K})}$ and, hence, {$\widetilde{\mathfrak{I}}$}  is contained in the radical  $\delta$-ideal ${\{\det(X)-\det(Z_0)\}_{\d}}$ of {$\widetilde{\K}\{X, \frac{1}{\det(X)} \}_{\d}$} generated by $\det(X)-\det(Z_0)$. We now claim the equality of ideals
$
{\{\det(X)-\det(Z_0)\}_{\d} \cap  \widetilde{\K}\{x_1,\dots,x_{n}\}_{\d}  =\{0\}}
$.
Indeed, let us consider ${\widetilde{P}=\widetilde{P}(X)=\widetilde{P}(x_{1},\ldots,x_{n}) \in \{\det(X)-\det(Z_0)\}_{\d} \cap  \widetilde{\K}\{x_1,\dots,x_{n}\}_{\d}}$. For any $(a_{1},\ldots,a_{n}) \in \widetilde{\K}^{n} \setminus \{(0,\ldots,0)\}$, there exists a matrix ${A \in M_{n}(\widetilde{\K})}$ with first row $(a_{1},\ldots,a_{n})$ such that $\det(A)=\det(Z_{0})$, so that ${\widetilde{P}(a_{1},\ldots,a_{n})=\widetilde{P}(A)=0}$ because $\widetilde{P} \in \{\det(X)-\det(Z_0)\}_{\d}$. Therefore, $\widetilde{P}$ vanishes on  $\widetilde{\K}^{n} \setminus \{(0,\ldots,0)\}$ and, hence, $\widetilde{P}=0$. 
We now have the desired result because  
\begin{multline*}
\mathfrak{I} \cap \K\{x_1,\dots,x_{n}\}_{\d} \subset \widetilde{\mathfrak{I}} \cap  \widetilde{\K}\{x_1,\dots,x_{n}\}_{\d}\\
\subset \{\det(X)-\det(Z_0)\}_{\d} \cap  \widetilde{\K}\{x_1,\dots,x_{n}\}_{\d}  =\{0\}.
\end{multline*}

\begin{center}\boxed{\Sp_{n}(\k)\hbox{-case.}} \end{center}
We now consider the $\Sp_n(\k)$-case. Arguing and using the same notations as in the $\SL_{n}(\k)$-case treated above, we see that it is sufficient to prove that the equality 
$\widetilde{\mathfrak{I}} \cap  \widetilde{\K}\{x_1,\dots,x_{n}\}_{\d}=\{0\}$ 
holds true if $H=P\Sp_n(\k)P^{-1}$ instead of $P\SL_{n}(\k)P^{-1}$. If $H=P\Sp_{n}(\k)P^{-1}$ then $\mathcal{G}(\widetilde{\K})$ contains $$\widetilde{H}=\{Q \in \GL_n(\widetilde{\K})|  Q D_sQ^t=D_s \mbox{ and } \det(Q)=1 \} $$ with  $D_s= PJ P^{t}$ and, hence, {$\widetilde{\mathfrak{I}}$} is contained in the radical $\delta$-ideal ${\{X D_s X^t -Z_0 D_s Z_0^t, \det(X)-\det(Z_0)\}_{\d}}$ of {$\widetilde{\K}\{X, \frac{1}{\det(X)} \}_{\d}$} generated by ${X D_s X^t -Z_0 D_s Z_0^t}$ and $\det(X)-\det(Z_0)$. Of course, $\widetilde{H}$ is nothing but the symplectic group for the symplectic form with matrix $D_s$ in the canonical basis of $\widetilde{\K}^n$. 
The first row of $Z_0$ is non zero and, hence, is the first vector of a symplectic basis of $\widetilde{\K}^n$ with respect to the symplectic form with matrix $D_s$ in the canonical basis of $\widetilde{\K}^n$. This proves that $Z_0=RS$ for some $R \in \GL_{n}(\widetilde{\K})$ with first row $(1,0,\ldots,0)$  
and some $S \in \widetilde{H}$. 
Then, $RD_sR^t=Z_0D_sZ_0^t   $ and ${\det(R)=\det(Z_0)}$. So, setting $Y= R^{-1}X$ and denoting by $(y_1, \dots,y_n)$ the first row of $Y$, we have $\widetilde{\K}\{X, \frac{1}{\det(X)} \}_{\d}=\widetilde{\K}\{Y, \frac{1}{\det(Y)} \}_{\d}$, $\widetilde{\K}\{x_1,\dots,x_{n}\}_{\d}=\widetilde{\K}\{y_1,\dots,y_{n}\}_{\d}$ and  $$\widetilde{\mathfrak{I}} \subset \{X D_s X^t -Z_0 D_s Z_0^t, \det(X)-\det(Z_0)\}_{\d}=\{YD_sY^t-D_s, \det(Y)-1 \}_{\d}.$$ Now, we claim that $\widetilde{\mathfrak{I}} \cap \widetilde{\K}\{y_1,\dots,y_{n}\}_{\d} {=} \{0  \}$. Indeed, consider $$\widetilde{P}=\widetilde{P}(Y)=\widetilde{P}(y_{1},\ldots,y_{n}) \in \{YD_sY^t-D_s, \det(Y)-1 \}_{\d} \cap \widetilde{\K}\{y_1,\dots,y_{n}\}_{\d}.$$ Using the fact that any nonzero vector of a symplectic vector space is the first vector of some symplectic basis, we see that, for any ${(a_{1},\ldots,a_{n}) \in \widetilde{\K}^{n}}$, there exists a matrix $A \in \widetilde{H}$ with first row $(a_{1},\ldots,a_{n})$. So, ${\widetilde{P}(a_{1},\ldots,a_{n})=\widetilde{P}(A)=0}$ because $\widetilde{P} \in \{YD_sY^t-D_s, \det(Y)-1 \}_{\d}$ and, hence, $\widetilde{P}=0$. We now have the desired result because  
\begin{multline*}
\mathfrak{I} \cap \K\{y_1,\dots,y_{n}\}_{\d} \subset \widetilde{\mathfrak{I}} \cap  \widetilde{\K}\{y_1,\dots,y_{n}\}_{\d}\\
\subset \{YD_sY^t-D_s, \det(Y)-1 \}_{\d} \cap  \widetilde{\K}\{y_1,\dots,y_{n}\}_{\d}  =\{0\}
.\end{multline*}
\begin{center}\boxed{\SO_{n}(\k)\hbox{-case.}} \end{center}
We now consider the $\SO_n(\k)$-case. For the sake of clarity, we will prove that $u_{1},\ldots,u_{n-1}$ are $\d$-algebraically independent over $\K$, the general case being similar. Arguing and using the same notations as in the $\SL_{n}(\k)$-case treated above, we see that it is sufficient to prove that the equality 
$\widetilde{\mathfrak{I}} \cap  \widetilde{\K}\{x_1,\dots,x_{n-1}\}_{\d}=\{0\}$ 
holds true if $H=P\SO_n(\k)P^{-1}$ instead of $P\SL_{n}(\k)P^{-1}$. If $H=P\SO_n(\k)P^{-1}$ then $\mathcal{G}(\widetilde{\K})$ contains ${\widetilde{H}=\{Q \in \GL_n(\widetilde{\K}) | Q DQ ^t =D \mbox{ and } \det(Q)=1 \}}$ with $D= P P^{t}$ and, hence, {$\widetilde{\mathfrak{I}}$} is contained in the radical  $\delta$-ideal $\{X DX^t -Z_0DZ_0^t,\det(X)-\det(Z_0)\}_{\d}$ of {$\widetilde{\K}\{X, \frac{1}{\det(X)} \}_{\d}$} generated by $ X DX^t -Z_0DZ_0^t$ and $\det(X)-\det(Z_0)$. Of course, $\widetilde{H}$ is the special orthogonal group for the bilinear form on $\widetilde{\K}^{n}$ with matrix $D$ with respect to the canonical basis of $\widetilde{\K}^{n}$.  {We can decompose $Z_0$ as $R Q$ where $R \in \GL_{n}(\widetilde{\K})$ is lower triangular and $Q \in \widetilde{H}$. Then, $RDR^t=Z_0DZ_0^t $ and ${\det(R)=\det(Z_0)}$.}
So, setting $Y= R^{-1}X$ and denoting by $(y_1, \dots,y_n)$ the first row of $Y$, we have $\widetilde{\K}\{X, \frac{1}{\det(X)} \}_{\d}=\widetilde{\K}\{Y, \frac{1}{\det(Y)} \}_{\d}$, $\widetilde{\K}\{x_1,\dots,x_{n-1}\}_{\d}=\widetilde{\K}\{y_1,\dots,y_{n-1}\}_{\d}$ and $$\widetilde{\mathfrak{I}} \subset \{X DX^t -Z_0DZ_0^t,\det(X)-\det(Z_0)\}_{\d}= \{YDY^t-D, \det(Y)-1 \}_{\d}.$$ Now, we claim that $ \{YDY^t-D, \det(Y)-1 \}_{\d} \cap \widetilde{\K}\{y_1,\dots,y_{n-1}\}_{\d} = \{0  \}$. Indeed, consider $$\widetilde{P}=\widetilde{P}(Y)=\widetilde{P}(y_{1},\ldots,y_{n-1}) \in\{YDY^t-D, \det(Y)-1 \}_{\d} \cap \widetilde{\K}\{y_1,\dots,y_{n-1}\}_{\d}.$$ By the Graam-Schmidt process, for any $(a_{1},\ldots,a_{n-1}) \in \widetilde{\K}^{n-1}$, there exists $a_{n} \in \widetilde{\K}$ and a matrix $A \in \widetilde{H}$ with first row $(a_{1},\ldots,a_{n})$, so that ${\widetilde{P}(a_{1},\ldots,a_{n-1})=\widetilde{P}(A)=0}$ because $\widetilde{P} \in \{YDY^t-D, \det(Y)-1 \}_{\d}$ and, hence, $\widetilde{P}=0$. We now have the desired result because  
\begin{multline*}
\mathfrak{I} \cap \K\{y_1,\dots,y_{n-1}\}_{\d} \subset \widetilde{\mathfrak{I}} \cap  \widetilde{\K}\{y_1,\dots,y_{n-1}\}_{\d}\\
\subset \{YDY^t-D, \det(Y)-1 \}_{\d} \cap  \widetilde{\K}\{y_1,\dots,y_{n-1}\}_{\d}  =\{0\}
. 
\end{multline*}
\end{proof}

\subsection{Projective isomonodromy}

Let $\K$ be a $\pd$-field with $\k=\K^{\sq}$ algebraically closed. Let  $\widetilde{\k}$ be a $\delta$-closure of $\k$. Let $\Const=\widetilde{\k}^\delta=\k^\delta$ be the (algebraically closed) field of constants of $\widetilde{\k}$. Lemma~\ref{lem:extconst} ensures that  
$\widetilde{\k}\otimes_\k \K$ is an integral domain and that  
${\L=\operatorname{Frac}(\widetilde{\k}\otimes_\k \K)}$ is a $(\sq,\delta)$-field extension of $\K$ such that $\L^{\sq}=\widetilde{\k}$. We let $\cQ_S$ be the total ring of quotients of a $(\sq,\d)$-PV ring  $S$ over $\L$ of the difference system 
$$
\sq(Y)=AY
$$ 
where $A \in \GL_n(\K)$.

\begin{propo}\label{propo: galois param proj const}
The following properties are equivalent: 
\begin{enumerate}
\item $\Galdelta(\cQ_S/\L)$ is conjugate over $ \GL_{n}(\widetilde{\k})$ to a subgroup of $\widetilde{\k}^{\times} \SL_{n}(\Const)$;  
\item there exists $B \in \K^{n \times n}$ such that 
\begin{equation}\label{gen integ hypertranalgclos bis}
\sq(B)A=AB+\delta(A)-\frac{1}{n}\delta(\det(A))\det(A)^{-1}A.
\end{equation}
\end{enumerate}
\end{propo}

\begin{proof}
 The proof of this proposition is the same as the proof of \cite[Proposition~2.10]{DHR} and hence is omitted.
\end{proof}

In what follows, we denote by $N_{G}(H)$ the normalizer of $H$ in $G$. A subgroup of the linear group is called irreducible if the induced representation is irreducible.

\begin{lem}\label{lem normalisateur}
Let $H$ be an irreducible subgroup of $\SL_{n}(\Const)$. Then, 
$$
N_{\GL_{n}(\widetilde{\k})}(H)=\widetilde{\k}^{\times} N_{\SL_{n}(\Const)}(H).
$$ 
\end{lem}

\begin{proof}
Let $M \in \GL_{n}(\widetilde{\k})$ be in the normalizer of $H$. Consider $N \in H$. 
We have $MNM^{-1} \in H$. In particular, we have $\delta(MNM^{-1})=0$, {\it i.e.}, ${\delta(M)NM^{-1}-MNM^{-1} \delta(M) M^{-1}=0}$, 
so $M^{-1} \delta(M)$ commutes with $N$. It follows from Schur's lemma that $M^{-1}\delta(M) = c I_{n}$ 
for some $c \in \widetilde{\k}^{\times}$. So, the entries of $M=(m_{i,j})_{1 \leq i,j \leq n}$ are solutions of 
$\delta(y)=c y$. Let $i_{0},j_{0}$ be such that $m_{i_{0},j_{0}} \neq 0$. Then, $M=m_{i_{0},j_{0}} M'$ with 
$M'=\frac{1}{m_{i_{0},j_{0}}} M \in \GL_{n}(\widetilde{\k}^{\delta})=\GL_{n}(\Const)$. Since $\Const$ is algebraically closed, we can
write $M=\lambda M''$ for ${M'' \in \SL_n(\Const)}$ and $\lambda \in \Const^\times$. Hence, the normalizer of $H$ in 
$\GL_{n}(\widetilde{\k})$ is included in $\k^{\times}  N_{\SL_{n}(\Const)}(H)$. It follows that 
$N_{\GL_{n}(\widetilde{\k})}(H) \subset \widetilde{\k}^{\times} N_{\SL_{n}(\Const)}(H)$. 
The other inclusion is obvious. 
\end{proof}
For any  algebraic subgroup $G$ of  $\GL_{n}(\k)$, let $G^{\circ}$ be the neutral component of $G$ and $G^{\circ, der}$ be the derived subgroup of $G^{\circ}$.
We recall that a linear algebraic group $G$ is almost simple if it is  infinite, non-commutative and if every proper normal closed subgroup of $G$ is finite.  
In particular, $G$ is connected. 
Moreover, $G$ equals its derived subgroup $G^{der}$.

\begin{propo}\label{propo:projisomonodgencondition bis}
Assume that the difference Galois group $G$ of $\sq(Y)=AY$ over the $\sq$-field $\K$ satisfies the following property: the algebraic group $G^{\circ,der}$ is an irreducible almost simple algebraic subgroup of $\GL_{n}(\k)$ defined over $\Const$. 
Then, we have the following alternative:
\begin{enumerate}
\item $\Galdelta(\cQ_S/\L)$ is conjugate to a subgroup of $\widetilde{\k}^{\times} N_{\SL_{n}(\Const)}(G^{\circ,der}(\Const))$ containing $G^{\circ,der}(\Const)$; 
\item $\Galdelta(\cQ_S/\L)$ is equal to a subgroup of $G(\widetilde{\k})$ containing $G^{\circ,der}(\widetilde{\k})$.
\end{enumerate}

Furthermore, the first case holds if and only if there exists $B \in \K^{n \times n}$ such that 
\begin{equation}\label{gen integ hypertranalgclos}
\sq(B)A=AB+\delta(A)-\frac{1}{n}\delta(\det(A))\det(A)^{-1}A.
\end{equation}
\end{propo}

\begin{proof}
Let $R$ be the $\L$-$\sq$-algebra generated by the entries of a fundamental matrix of solutions $U \in \GL_{n}(\cQ_S) $ and by $\det(U)^{-1}$; this is a PV ring for $\sq(Y)=AY$ over the $\sq$-field $\L$.
Using \cite[Corollary 2.5]{CHS}, we see that ${\Gal(\cQ_R/\L)=G(\widetilde{\k})}$. So, ${\Gal(\cQ_R/\L)^{\circ,der}=G^{\circ,der}(\widetilde{\k})}$. Since $\Galdelta(\cQ_S/\L)$ is Zariski-dense in $\Gal(\cQ_R/\L)$ (see Proposition \ref{propo:zarclosurePPvgaloisgroup}), we have that  $\Galdelta(\cQ_S/\L)^{\circKol, \derKol}$\footnote{This is the Kolchin-closure of the derived subgroup of  $\Galdelta(\cQ_S/\L)^{\circKol}$ where the notation $\circKol$ means that we consider the identity component of the group for the Kolchin topology; see \cite[Section~4.4.1]{DHR}.} is Zariski-dense in the group ${\Gal(\cQ_R/\L)^{\circ ,der}=G^{\circ,der}(\widetilde{\k})}$. By \cite[Theorems 19 and 20]{CassidyCSSDAG},  $\Galdelta(\cQ_S/\L)^{\circKol,\derKol}$ is  either conjugate to $G^{\circ,der}(\Const)$ or equal to $G^{\circ,der}(\widetilde{\k})$. Since $\Galdelta(\cQ_S/\L)^{\circKol,\derKol}$ is a normal subgroup of $\Galdelta(\cQ_S/\L)$, Lemma \ref{lem normalisateur} ensures that $\Galdelta(\cQ_S/\L)$ is either conjugate to a subgroup of $\widetilde{\k}^{\times} N_{\SL_{n}(\Const)}(G^{\circ,der}(\Const))$ containing $G^{\circ,der}(\Const)$ or is equal to a subgroup of $G(\widetilde{\k})$ containing $G^{\circ,der}(\widetilde{\k})$.

The remaining statement is a direct consequence of Proposition \ref{propo: galois param proj const}. 
\end{proof}

\section{Large $(\sq,\d)$-Galois group of $q$-difference equations}\label{sec hyptr}
%%%%%%%%%%%%%%%%%%%%%%%%%%%%%%%%%%%%%%%%%%%%%%%%%%%%%%%%%%%%
%%%%%%%%%%%%%%%%%%%%%%%%%%%%%%%%%%%%%%%%%%%%
 In this section, we focus our attention on $q$-difference equations over $\C(z)$.
 Let us consider the field $\Kq$ and the algebraic closure $\Kqbar$ 
of $\Kq$ in (the algebraically closed field)  $\C((z^{*}))=\bigcup_{j=1}^\infty \C((z^{1/j}))$.  Let $q$ be a non zero complex number  such that $|q|\neq 1$. We choose a consistent system $(q_j)_{j \geq 1}$ of roots of $q$; this means that $(q_j)_{j \geq 1}$ is a sequence of complex numbers such that, for all positive integer $j$,  
$q_j^j=q$ and, for all positive integers $j,k,l$, if $j=lk$ then $q_{j}^l=q_{k}$. This allows us to extend the action of $\sq$ to $\Kqbar$ by setting $\sq(f)=f(q_jz^{1/j})$ for $f \in \Kqbar \cap \C((z^{1/j}))$. We have $\Kqbar^{\sq}=\C$. The derivation $\delta=z \frac{d}{dz}$ endows $\Kqbar$ with a structure of $(\sq,\d)$-field. Note also that $\Kq$ is a $(\sq,\d)$-subfield of $\Kqbar$ with $\Kq^{\sq}=\C$.\par 
Let $(\widetilde{\C},\d)$ be a $\delta$-field that contains $(\C,\d)$ and which  is $\delta$-closed. According to Lemma~\ref{lem:extconst},  the $(\sq,\d)$-field 
$$
\Lqbar=\operatorname{Frac}(\widetilde{\C} \otimes_\C \Kqbar)
$$ 
is a $(\sq,\d)$-field extension of $\Kqbar$ such that $\Lqbar^{\sq}=\widetilde{\C}$. 

Consider the  $q$-difference  system
\begin{equation}\label{eq102}
\sq (Y)=AY
\end{equation}
with $A \in \GL_n(\C(z))$. In what follows, we let $S$ be a $(\sq,\d)$-PV ring over $\L$ for the equation (\ref{eq102}), $\cQ_S$ be the total ring of quotients of $S$, and we denote by $\Galdelta(\cQ_S/\L)$ the corresponding $(\sq,\d)$-Galois group over $\L$.

The theorem below shows that if the difference Galois group of a $q$-difference system is large, the same holds for the parametrized difference Galois group.

\begin{theo}\label{theo3}
 Let $G$ be the difference Galois group of the $q$-difference system (\ref{eq102}) over the $\sq$-field $\Kq$. Assume that $G^{\circ, der}$ is an irreducible almost simple algebraic subgroup of $\SL_{n}(\C)$. Then, $\Galdelta(\cQ_S/\Lqbar)$ is a subgroup of $G(\widetilde{\C})$ containing $G^{\circ, der}(\widetilde{\C})$.
\end{theo}

Before giving the proof of Theorem \ref{theo3}, we state and prove some preliminary results.

 \begin{lem} \label{lemm:preservationgroupbaseextlog}
Let $G$ be the difference Galois group of \eqref{eq102} over the $\sq$-field $\Kq$. 
Let $H{\subset \GL_n(\widetilde{\C})}$ be the difference Galois group of \eqref{eq102} over the $\sq$-field $\Lqbar$.
Then, {${H^{\circ}(\widetilde{\C})}=G^{\circ}(\widetilde{\C})$}. 
  \end{lem}

\begin{proof}
Since $\Kqbar$ is an algebraic extension of $\Kq$, \cite[Theorem 7]{Ro15} implies that the difference Galois group  $G'$ of \eqref{eq102} over the $\sq$-field $\Kqbar$ has the same connected component as $G$.  By  \cite[Corollary 2.5]{CHS}, the group $H$ is isomorphic to $G'(\widetilde{\C})$. Therefore, the group $H^{\circ}$ is isomorphic to $G'^{\circ}(\widetilde{\C})=G^{\circ}(\widetilde{\C})$.
\end{proof}

\begin{rem}
As a straightforward consequence of Lemma \ref{lemm:preservationgroupbaseextlog}, we obtain that if $G^{\circ, der}$ is an irreducible almost simple algebraic subgroup of $\GL_{n}(\C)$ then ${H}^{\circ,der}$ equals $G^{\circ, der}(\widetilde{\C})$ and is  an irreducible almost simple algebraic subgroup of $\GL_{n}(\widetilde{\C})$.
\end{rem}
\begin{lem}\label{lem:ppvformalsol}
Assume that the system \eqref{eq102}, has a solution $u=(u_{1},\dots,u_{n})^{t}$ with coefficients in $\Kqform$. Then, there exists a
$(\sq,\d)$-PV ring $T$ over $\Lqbar$ of \eqref{eq102} that contains the $\Lqbar$-$\delta$-algebra $\Lqbar\{ u_{1},\dots,u_{n} \}_{\d}$. 
\end{lem}

\begin{proof}
The result is obvious if $u=(0,\dots,0)^{t}$. We shall now assume that ${u \neq (0,\dots,0)^{t}}$.
We equip $\displaystyle\Kqform$ with the structure of $(\sq,\delta)$-field given by $\sq(f(z))=f(qz)$ and $\delta =  z \frac{d}{dz}$. It is easily seen that we have ${\displaystyle\Kqform^{\sq}=\C}$. 
We let $F=\Kqbar\langle u_{1},\dots,u_{n} \rangle_{\d}$ be the $\delta$-subfield of $\displaystyle\Kqform$ generated over $\Kqbar$ by the series $u_{1},\dots,u_{n}$; this is a $(\sq,\delta)$-subfield of $\displaystyle\Kqform$ such that ${F^{\sq}=\C}$. By Lemma~\ref{lem:extconst}, $\widetilde{\C} \otimes_\C F $ is an integral domain and its field of fractions ${\mathbf{L}_{1}=\Lqbar\langle u_{1},\dots,u_{n} \rangle_{\d}}$ is a $(\sq,\delta)$-field such that ${\mathbf{L}_{1}}^{\sq}=\widetilde{\C}$. We consider a total $(\sq,\d)$-PV extension $\cQ_{S_{1}}$ for \eqref{eq102} over $\L_{1}$ and we let $U \in \GL_{n}(\cQ_{S_{1}})$ be a fundamental matrix of solutions of this difference system. We can assume that the first column of $U$ is $u$.  Let $T$ be the  $\Lqbar$-$(\sq,\delta)$-algebra generated by the entries of $U$ and by $\det(U)^{-1}$. Since the total ring of quotient $\cQ_T$ of $T$ is contained in $\cQ_{S_1}$ and $\cQ_{S_1}^{\sq}=\widetilde{\C}$, the $\sq$-constant field of  $\cQ_T$ is $\widetilde{\C}$. By \cite[Proposition 6.17]{HS}, the ring $T$ is is a $(\sq,\d)$-PV ring for \eqref{eq102} over $\Lqbar$ that  contains $\Lqbar\{ u_{1},\dots,u_{n} \}_{\d}$ by construction.
\end{proof}

\begin{lem}\label{sol commune}
Let us consider a vector $u=(u_{1},\dots,u_{n})^{t}$ with coefficients in ${\displaystyle\Kqform}$ which is solution of (\ref{eq102}). Assume moreover that each $u_{i}$ satisfies some nonzero linear differential equation with coefficients in $ \Kq$.  Then, the $u_{i}$ actually belong to $\Kqbar$.  
\end{lem}

\begin{proof}
According to the cyclic vector lemma, there exists $P \in \GL_{n}(\Kq)$ such that $Pu = (f,\sq(f),\dots,\sq^{n-1}(f))^{t}$ for some $f \in \displaystyle\Kqform$, which is a solution of a nonzero linear $q$-difference equation  of order $n$ with coefficients in $\Kq$. Moreover, $f$ satisfies a nonzero linear differential equation with coefficients in $ \Kq$, because it is a $\Kq$-linear combination of the $u_{i}$. Let $j\in  \N^{*}$ such that $f\in \C((z^{1/j}))$. Up to taking a ramification of the variable $z$, it follows from \cite[Theorem 7.6]{R92} that $f$ belongs to $\C(z^{1/j})$. Hence, the entries of $u=P^{-1}(Pu)=P^{-1}(f,\sq(f),\dots,\sq^{n-1}(f))^{t}$ actually belong to $\C(z^{1/j})\subset \Kqbar$ as expected.    
\end{proof}

\begin{proof}[Proof of Theorem \ref{theo3}]
Using Lemma \ref{lemm:preservationgroupbaseextlog} and Proposition \ref{propo:projisomonodgencondition bis}, we 
are reduced to proving that the $(\sq,\d)$-Galois group over the $\pd$-ring $\Lqbar$ of $\sq (Y)=AY$ is not conjugate to 
a subgroup of $\widetilde{\C} \cdot N_{\SL_{n}(\C)}(G^{\circ,der}(\C))$. Suppose to the  contrary that it is conjugate to 
a subgroup of $\widetilde{\C} \cdot N_{\SL_{n}(\C)}(G^{\circ,der}(\C))$. 
Let $\sqrt[n]{\det A}$ be a $n$-th root of $\det A$ in $\Kqbar$. 
We consider ${A'=(\sqrt[n]{\det A})^{-1} A \in \SL_{n}(\Kqbar)}$. Let $\C(\{z\})$ be the field of fraction of the ring of convergent power series $\C\{z\}$.
Since ${\Kqbar \subset \displaystyle\bigcup_{j=1}^{\infty}\C(\{z^{1/j}\})}$, up to taking a ramification of the variable $z$, we can apply Lemma~\ref{lem:convergentformal solution} to the system ${\sq(Y)=A' Y}$ and we get that there exist $c \in \C^{\times}$ and $r \in \Q$ such that ${\sq(Y)=A'' Y}$, with ${A''=cz^{r}A' \in \GL_{n}(\Kqbar)}$,  has a nonzero solution $u=(u_{1},\dots,u_{n})^{t}$ with coefficients in  $ \displaystyle \bigcup_{j=1}^{\infty}\C(\{ z^{1/j}\})\subset \Kqform$. In virtue of Lemma \ref{lem:ppvformalsol}, there exists a $(\sq,\d$)-PV ring $S$ over the $(\sq,\d)$-ring $\Lqbar$ for $\sq(Y)=A'' Y$ containing the entries of $u$. We let $U'' \in \GL_{n}(S)$ be a fundamental matrix of solutions of $\sq(Y)=A'' Y$ whose first column is $u$.\par 
 We claim that  the neutral component of the derived groups of the difference Galois groups of the systems $\sq(Y)=AY$ and $\sq(Y)=A''Y$ over $\L$ coincide and are therefore equal to $G^{\circ,der}(\widetilde{\C})$. Indeed, we first note that  $A''=hA$ for some $h \in \L^\times$ and we let $R$ be a   Picard-Vessiot ring over $\L$ for the system $\sq(Y)=\begin{pmatrix}
A& 0 \\
0& h 
\end{pmatrix}Y$. There exists $U \in \GL_n(R)$ and $v \in R^\times$ such that $\sq(U)=AU$ and $\sq(v)=hv$. Then, $\L[U,\frac{1}{\det(U)}] \subset R$ (resp. $\L[vU,\frac{1}{v^n\det(U)}] \subset R$) is a  Picard-Vessiot ring for $\sq(Y)=AY$ (resp $\sq(Y)=A''Y$) over $\L$. In the representation
attached to $U$ and $vU$, one can easily conclude to the equality of the derived groups, and therefore, the equality of the  neutral component of the derived groups. This proves the claim. \par

 Now, since the $(\sq,\d)$-Galois group  of $\sq (Y)=AY$  over $\L$ is  conjugate to 
a subgroup of $\widetilde{\C} \cdot N_{\SL_{n}(\C)}(G^{\circ,der}(\C))$, Proposition \ref{propo:projisomonodgencondition bis} ensures that there exists ${B \in \Kqbar^{n \times n}}$ such that 
\begin{equation}
\sq(B)A=A B+\d(A)-\frac{1}{n}\d(\det(A))\det(A)^{-1}A.
\end{equation}
An easy computation shows that 
\begin{equation}
\sq(B)A''=A''B+\d(A'')-\frac{1}{n}\d(\det(A''))\det(A'')^{-1}A''.
\end{equation}

Since the determinant $\mathbf{d}=\det (U'')$ satisfies the $q$-difference equation ${\sq(\mathbf{d})=(\det A'') \mathbf{d}=(cz^{r})^{n} \mathbf{d}}$, we obtain the integrability of the system of equations 
$$
\begin{cases}
\sq (Y)=A''Y\\ 
\delta (Y)=  (B+ \frac{\d \mathbf{d}}{n\mathbf{d}}) Y.
\end{cases} 
$$
 So, there exists ${D\in \mathrm{GL}_{n}(\widetilde{\C})}$ such that $V=U''D\in \GL_{n}(S)$ satisfies 
\begin{equation}\label{integ eq}
\begin{cases}
\sq(V)=A''V \\ 
\d(V)= (B+ \frac{\d \mathbf{d}}{n\mathbf{d}}) V.
\end{cases} 
\end{equation}

We recall that $\sq \circ \d=\d\circ\sq$. Note that $\frac{\d \mathbf{d}}{\mathbf{d}} \in S$ is such that 
${\sq\left(\frac{\d \mathbf{d}}{\mathbf{d}}\right)=\frac{\d \mathbf{d}}{\mathbf{d}}+nr}$. 
So, $\Lqbar\{ \frac{\d \mathbf{d}}{\mathbf{d}}\}_{\d} \subset S$ is a $(\sq,\d)$-PV ring over the $(\sq,\d)$-ring $\Lqbar$. The corresponding $(\sq,\d)$-Galois group is Kolchin-connected because it is a $\d$-subgroup of the additive group $\mathbb{G}_a(\widetilde{\C})$ and, hence, according to \cite[Proposition 11]{C72}, it is the vector space of solutions of a linear differential operator. Therefore,  
$\Lqbar\{ \frac{\d \mathbf{d}}{\mathbf{d}}\}_{\d}$ is an integral domain and, hence, we can consider its field of fraction $\Lqbar\langle \frac{\d \mathbf{d}}{\mathbf{d}}\rangle_{\d} \subset \mathcal{Q}_{S}$.  

Note that, since $\sq\left(\frac{\d \mathbf{d}}{\mathbf{d}}\right)=\frac{\d \mathbf{d}}{\mathbf{d}}+nr$, we have $\sq \left(\d \left(\frac{\d \mathbf{d}}{\mathbf{d}}\right)\right)=\d (\frac{\d \mathbf{d}}{\mathbf{d}})$, and therefore, $\delta\left(\frac{\d \mathbf{d}}{\mathbf{d}}\right) \in S^{\sq}=\widetilde{\C}$. Consequently, $\Lqbar\langle \frac{\d \mathbf{d}}{\mathbf{d}}\rangle_{\d} =\Lqbar( \frac{\d \mathbf{d}}{\mathbf{d}})$.

Using (\ref{integ eq}), we get $\d(U'')D+U''\d(D)=\d(U''D)=\d(V) =(B +\frac{\d(\mathbf{d})}{n\mathbf{d}}) U''D$ so $$\d(U'')= \left(B +\frac{\d(\mathbf{d})}{n\mathbf{d}}\right)  U''-U''\d(D)D^{-1}.$$ 
The previous formula implies that the  $\Lqbar( \frac{\d \mathbf{d}}{\mathbf{d}})$-vector subspace of $\mathcal{Q}_{S}$ generated by the entries of $U''$ and all their successive $\d$-derivatives is of finite dimension. In particular, any $u_{i}$ satisfies a nonzero linear $\d$-equation $\mathcal{L}_i(y)=0$ with coefficients in $\Lqbar [\frac{\d \mathbf{d}}{\mathbf{d}}]$. 

We claim that any $u_{i}$ satisfies a nonzero linear $\d$-equation with coefficients in $\Lqbar$.

If $nr=0$, we have 
$\sq\left(\frac{\d \mathbf{d}}{\mathbf{d}}\right)=\frac{\d \mathbf{d}}{\mathbf{d}}+nr=\frac{\d \mathbf{d}}{\mathbf{d}}$,
and therefore $\frac{\d \mathbf{d}}{\mathbf{d}} \in S^{\sq}=\widetilde{\C}$, which proves our claim.

Assume that $nr \neq 0$. The equation $\mathcal{L}_i(y)=0$ can be rewritten as $\sum_{j=0}^{\nu} \mathcal{L}_{i,j}(y) (\frac{\d \mathbf{d}}{\mathbf{d}})^{j}=0$ where the $\mathcal{L}_{i,j}(y)$ are linear $\d$-operators with coefficients in $\Lqbar$, not all zero. 

Let us now prove that $\frac{\d \mathbf{d}}{\mathbf{d}}$ is transcendental over $\Lqbar\langle u_{1},\dots,u_{n} \rangle_{\d}$. Indeed, suppose to the contrary that there is a non zero relation 
\begin{equation}\label{eq min}
\sum_{k=0}^{\kappa} a_{k}\left(\frac{\d \mathbf{d}}{\mathbf{d}}\right)^{k}=0
\end{equation}
with $\kappa \geq 1$ and $a_{0},\ldots,a_{\kappa-1},a_{\kappa}=1 \in \Lqbar\langle u_{1},\dots,u_{n} \rangle_{\d}$. We can and will assume that $\kappa \geq 1$ is minimal. Applying $\sq$ to equation (\ref{eq min}), we get  
\begin{equation} \label{sq eq min}
\sum_{k=0}^{\kappa} \sq (a_{k}) \left(\frac{\d \mathbf{d}}{\mathbf{d}}+nr\right)^{k}=0.
\end{equation}
Since $\kappa$ is minimal and $a_\kappa=\sq(a_\kappa)=1$, the coefficients of any $\left(\frac{\d \mathbf{d}}{\mathbf{d}}\right)^{k}$ in (\ref{eq min}) and (\ref{sq eq min}) are equal. In particular, equating the coefficients of $\left(\frac{\d \mathbf{d}}{\mathbf{d}}\right)^{\kappa-1}$, we get 
$$
a_{\kappa -1}=\sq (a_{\kappa -1})+\kappa nr.
$$ 
Since $a_{\kappa-1} \in \Kqform$, the term of degree $0$ in $a_{\kappa -1}-\sq (a_{\kappa -1})$ is equal to $0$ and, hence, is not equal to $\kappa nr\neq 0$. A contradiction proving that $\frac{\d \mathbf{d}}{\mathbf{d}}$ is transcendental over $\Lqbar\langle u_{1},\dots,u_{n} \rangle_{\d}$.\par 

It follows that $\frac{\d \mathbf{d}}{\mathbf{d}}$ is transcendental over $\Lqbar\langle u_{1},\dots,u_{n} \rangle_{\d}$ and that all the $\mathcal{L}_{i,j}(u_{i})$ are equal to zero. This proves our claim, that is, any $u_{i}$ satisfies some nonzero linear $\d$-equations with coefficients in $\Lqbar$. \par 
Since the $u_{i}$ belong to $\displaystyle \Kqform$, we obtain that any $u_{i}$ satisfies a nonzero linear $\d$-equation with coefficients in $\Kqbar$. Since $\Kqbar$ is an algebraic extension of $\Kq$, we get that any $u_{i}$ satisfies a nonzero linear $\d$-equation with coefficients in $\Kq$. \par 
The vector $u$ is a  solution of $\sq (Y)=A''Y$. Then, letting $p$ be a denominator of $r$ and considering the $pn$-th tensor power of this $q$-difference system, we get that $u^{\otimes pn}$ satisfies a linear $q$-difference equation with coefficients in $\Kq$. Since any $u_{i}$ satisfies a nonzero linear $\d$-equation with coefficients in $\Kq$, we find that $u^{\otimes pn}$ satisfies a nonzero linear $\d$-equation with coefficients in $\Kq$. It follows from Lemma~\ref{sol commune} that the entries of $u^{\otimes pn}$ belong to $\Kqbar$ and, hence, any $u_{i}$ belongs to $\Kqbar$. Therefore, the first column of $U''$ is fixed by the difference Galois group of $\sq(Y)=A'' Y$ over $\L$ and this contradicts the fact that this group contains $G^{\circ, der}(\widetilde{\C})$, which is irreducible by hypothesis.
\end{proof}

%%%%%%%%%%%%%%%%%%%%%%%%%%%%%%%%%%%%%%%%%%%%%%%%%%%%%%%%%%%%%%%%%%%%%%%%%%%%%%%%%%%%%%%%%%%%%%%%%%%%%%%%%%%%%
%%%%%%%%%%%%%%%%%%%%%%%%%%%%%%%%%%%%%%%%%%%%%%%%%%%%%%%%%%%%%%%%%%%%%%%%%%%%%%%%%%%%%%%%%%%%%%%%%%%%%%%%%%%

%%%%%%%%%%%%%%%%%%%%%%%%%%%%%%%%%%%%%%%%%%%%%%%%%%%%%%%%%%%%%%%%%%%
%%%%%%%%%%%%%%%%%%%%%%%%%%%%%%%%%%%%%%%%%%%%%%%%%%%%%%%%%%%%%%%%%%
\section{Applications}\label{sec applications}
%%%%%%%%%%%%%%%%%%%%%%%%%%%%%%%%%%%%%%%%%%%%%%%%%%%%%%%%%%%%%
%%%%%%%%%%%%%%%%%%%%%%%%%%%%%%%%%%%%%%%%%%%%%%%%%%%%%%%%%%%%

%%%%%%%%%%%%%%%%%%%%%%%%%%%%%%%%%%%%%%%%%%%%%%%%%%%%%%%%%%%%%%%%%%%%%%%%%%%%%%%%%%%%%%%%
%%%%%%%%%%%%%%%%%%%%%%%%%%%%%%%%%%%%%%%%%%%%%%%%%%%%%%%%%%%%%%%%%%%%%%%%%%%%%%%%%%%%%%%

\subsection{User friendly criterias for transcendence}
The goal of this subsection is to use Theorem \ref{theo3}, in order to give transcendence criteria. We refer to  Section $ \ref{sec hyptr}$ for the notations used in this section.

\begin{coro}\label{coro3}
Let $G$ be the difference Galois group of the $q$-difference system (\ref{eq102}) over the $\sq$-field $\Kq$. Let us assume that (\ref{eq102}) admits a non zero vector solution $u=(u_{1},\dots,u_{n})^{t}$ with entries in $\Kqform$.
\begin{itemize}
\item  Assume that $n\geq 2$ and $G^{\circ, der}=\SL_{n}(\C)$.
Then, the series $u_{1},\dots,u_{n}$ are $\d$-algebraically independent over $\C(z)$. In particular, any $u_{i}$ is $\d$-transcendental over $\C(z)$. 
\item Assume that $n\geq 3$ and $G^{\circ, der}=\mathrm{SO}_{n}(\C)$.
Then, the series $u_{1},\dots,u_{n-1}$ are $\d$-algebraically independent over $\C(z)$. 
\item Assume that $n$ is even and $G^{\circ, der}=\mathrm{Sp}_{n}(\C)$.  
Then, the series $u_{1},\dots,u_{n}$ are $\d$-algebraically independent over $\C(z)$. 
\end{itemize}
\end{coro}

\begin{proof} 
 Thanks to Lemma~\ref{lem:ppvformalsol}, there exists a $(\sq,\d)$-PV ring $S$ for the system (\ref{eq102}) over $\Lqbar$ containing $\Lqbar\{u_{1},\dots,u_{n} \}_{\d}$. Let $U \in \GL_{n}(S)$ be a fundamental matrix of solutions of the system (\ref{eq102}) whose first column is $u$. Since $G^{\circ, der}$ is equal to $\mathrm{SO}_{n}(\C)$, (resp. $\SL_{n}(\C)$, resp. $\mathrm{Sp}_{n}(\C)$), with Theorem \ref{theo3}, we find that the $(\sq,\d)$-Galois group of (\ref{eq102}) contains $\mathrm{SO}_{n}(\widetilde{\C})$, (resp. $\mathrm{SL}_{n}(\widetilde{\C})$, resp. $\mathrm{Sp}_{n}(\widetilde{\C})$).
 The results of  Section $ \ref{sec24}$ yield the desired conclusion.
 \end{proof}

Consider now the following $q$-difference equation 
\begin{equation}\label{equa generique user-friendly}
a_{n}(z) y(q^{n}z) + a_{n-1}(z) y(q^{n-1}z) + \cdots + a_{0}(z) y(z) = 0
\end{equation}
for some integer $n \geq 1$, and some $a_{0}(z),\dots,a_{n}(z) \in \C(z)$ with ${a_{0}(z)a_{n}(z)\neq 0}$. 
In what follows, by ``difference Galois group of equation (\ref{equa generique user-friendly})'', we mean the difference Galois group of the associated system 
\begin{equation}\label{syst generique user-friendly}
\sq (Y)=AY, \text{ with } A=\begin{pmatrix}
0&1&0&\cdots&0\\
0&0&1&\ddots&\vdots\\
\vdots&\vdots&\ddots&\ddots&0\\
0&0&\cdots&0&1\\
-\frac{a_{0}}{a_{n}}& -\frac{a_{1}}{a_{n}}&\cdots & \cdots & -\frac{a_{n-1}}{a_{n}}
\end{pmatrix} \in \GL_{n}(\C(z)).
\end{equation}

\begin{coro}\label{coro1}Let $G$ be the difference Galois group of the $q$-difference system (\ref{syst generique user-friendly}) over the $\sq$-field $\Kq$.  Let us assume that (\ref{equa generique user-friendly}) admits a non zero solution $g\in \Kqform$. 
\begin{itemize}
\item   Assume that $n\geq 2$ and $G^{\circ, der}=\SL_{n}(\C)$. 
Then, $g(z), g(qz),\ldots,g(q^{n-1}z)$ are $\d$-algebraically independent over $\C(z)$.
\item Assume that $n\geq 3$ and $G^{\circ, der}=\mathrm{SO}_{n}(\C)$. 
Then, $g(z), g(qz),\ldots,g(q^{n-2}z)$ are $\d$-algebraically independent  over $\C(z)$.
\item Assume that $n$ is even and $G^{\circ, der}=\mathrm{Sp}_{n}(\C)$. 
Then, the series $g(z), g(qz),\ldots,g(q^{n-1}z)$ are $\d$-algebraically independent over $\C(z)$.
\end{itemize}
\end{coro}

\begin{proof}
Let us note that if $g(z) \in \displaystyle \Kqform$ is a nonzero solution of (\ref{equa generique user-friendly}), then ${u_{1}=(g(z), g(qz),\ldots,g(q^{n-1}z))^{t}}$ is a nonzero solution of (\ref{syst generique user-friendly}) with entries in $\Kqform$. This is  a direct consequence of Corollary \ref{coro3}.
\end{proof}

\subsection{Generalized Hypergeometric series}\label{sec42}

In this subsection, we follow the notations of \cite{Ro11,Ro12} and we assume that $0< |q|<1$.  Once for all, we fix a determination $\log(q)$ of the logarithm of $q$ and, for all $\alpha\in \C$, we set $q^{\alpha}:=e^{\alpha \log (q)}$. Note that for all $\alpha,\beta\in \C$, we have $q^{\alpha+\beta}=q^{\alpha}q^{\beta}$. Let us fix $n,s\in \N^{*}$, let $\underline{a}=(a_{1},\dots,a_{n})\in (q^{\R})^{n}$, $\underline{b}=(b_{1},\dots,b_{s})\in (q^{\R}\setminus q^{-\N})^{s}$, $\l\in \C^{\times}$, and consider the $q$-difference operator:

\begin{equation}\label{eq3}
z\l \displaystyle \prod_{i=1}^{n} (a_{i}\sq -1)-\displaystyle \prod_{j=1}^{s} \left(\frac{b_{j}}{q}\sq -1\right).
\end{equation} 

When $b_{1}=q$, this operator admits as solution the power series:
$$
\begin{array}{lll}
_{n}\Phi_{s}(\underline{a},\underline{b},\l,q;z)&=&\displaystyle \sum_{m=0}^{\infty}\dfrac{(\underline{a};q)_{m}}{(\underline{b};q)_{m}}\l^{m}z^{m}\\
&=&\displaystyle \sum_{m=0}^{\infty} \frac{\displaystyle \prod_{i=1}^{n} (1-a_{i})(1-a_{i}q)\dots(1-a_{i}q^{m-1})}{\displaystyle \prod_{j=1}^{s}(1-b_{j})(1-b_{j}q)\dots(1-b_{j}q^{m-1})}\l^{m}z^{m} .
\end{array} $$

Until the end of the subsection, let us assume that $s=n \geq 2$ and that ${\underline{a}=(a_{1},\dots,a_{n})\in (q^{\Q})^{n}}$, $\underline{b}=(b_{1},\dots,b_{s})\in (q^{\Q}\setminus q^{-\N})^{s}$.\par 
According to \cite[Propositions 6 and 7]{Ro11}, the operator (\ref{eq3}) is irreducible over $\C(z)$ if and only if, for all  $(i,j)\in \{1,\dots,n\}^{2}$, $a_{i}\not\in b_{j}q^{\Z}$. We say that (\ref{eq3}) is $q$-Kummer induced if it is irreducible, and there exists a divisor $d\neq 1$ of $n$, and two permutations $\mu,\nu$ of $\{1,\dots,n\}$, such that, for all $i\in \{1,\dots,n\}$, $a_{i}\in a_{\mu (i)}q^{1/d}q^{\Z}$, and  $b_{i}\in b_{\nu (i)}q^{1/d}q^{\Z}$. 

\begin{theo}[{\cite[Theorem 6]{Ro11}}]\label{theo1}
 Let us assume that (\ref{eq3}) is irreducible  and not $q$-Kummer induced. Let $G$ be the difference Galois group  of the $q$-difference system (\ref{eq3}) over the $\sq$-field $\Kq$. Then, $G^{\circ, der}$ is either $\SL_{n}(\C)$, $\mathrm{SO}_{n}(\C)$ (only when $n$ is odd), or $\mathrm{Sp}_{n}(\C)$ (only when $n$ is even). Moreover, $G^{\circ, der}$ is $\mathrm{SO}_{n}(\C)$ (resp. $\mathrm{Sp}_{n}(\C)$) if and only if 
 \begin{itemize}
 \item $\prod_{i=1}^{n}a_{i}\in q^{\Z}\prod_{j=1}^{n}b_{j}$;
 \item there exists $c\in \mathbb{C}^{*}$, there exist two permutations $\mu_{1}$, $\mu_{2}$ of $\{1,\dots, n\}$, such that, for all $i,j\in \{1,\dots, n\}$, $ca_{i}a_{\mu_{1}(i)}\in q^{\Z}$, $cb_{j}b_{\mu_{2}(j)}\in q^{\Z}$;
 \item $n$ is odd (resp. even).
 \end{itemize}
\end{theo}

Theorem \ref{theo1} and Corollary \ref{coro1} yield the following result.

\begin{coro}
Let us assume that (\ref{eq3}) is irreducible  and not $q$-Kummer induced. Let $G$ be the difference Galois group  of the $q$-difference system (\ref{eq3}) over the $\sq$-field $\Kq$ and let $G^{\d}$, be the $\d$-Galois group of the $q$-difference system (\ref{eq3}) over the field  $\L$. 
\begin{itemize}
\item Assume that $G^{\circ, der}=\SL_{n}(\C)$ (resp.  that $n$ is odd and ${G^{\circ, der}=\mathrm{SO}_{n}(\C)}$, resp. that $n$ is even and $G^{\circ, der}=\mathrm{Sp}_{n}(\C)$). 
Then, $G^{\d}$ contains  $\SL_{n}(\widetilde{\C})$, (resp.  $\mathrm{SO}_{n}(\widetilde{\C})$, resp. $\mathrm{Sp}_{n}(\widetilde{\C})$). 
\item If we further assume that $b_{1}=q$, then we obtain that the series $_{n}\Phi_{n}(\underline{a},\underline{b},\l,q;z),\dots, \sq^{\kappa}\left( _{n}\Phi_{n}(\underline{a},\underline{b},\l,q;z) \right)$  with $\kappa =n-1$ (resp. ${\kappa =n-2}$, resp. $\kappa =n-1$) are $\d$-algebraically independent over $\C(z)$.
\end{itemize}
\end{coro}

\begin{proof}
The first point is a straightforward consequence of Theorems \ref{theo3}, and~\ref{theo1}. We conclude with Corollary \ref{coro1}. 
\end{proof}

\subsection{Irregular generalized Hypergeometric functions}\label{sec43}

In this subsection, we assume that  $n>s$, $n\geq 2$. Let $\underline{a}=(a_{1},\dots,a_{n})\in (q^{\R})^{n}$, $\underline{b}=(b_{1},\dots,b_{s})\in (q^{\R}\setminus q^{-\N})^{s}$, $\l\in \C^{\times}$, $0<|q|<1$.

\begin{theo}[{\cite[Page 1]{Ro12}}]\label{theo2}
Let $G$ be the difference Galois group  of the $q$-difference system (\ref{eq3}) over the $\sq$-field $\Kq$. 
For $(i,j)\in \{1,\dots,n\}\times  \{1,\dots,s\}$, let $\a_{i},\b_{j}\in \R$ such that $a_{i}=q^{\a_{i}}$ and $b_{i}=q^{\b_{j}}$. Assume that for all $(i,j)\in \{1,\dots,n\}\times  \{1,\dots,s\}$, $\a_{i}-\b_{j}\notin \Z$, and that the algebraic group generated by $\mathrm{Diag}(e^{2i\pi \a_{1}},\dots,e^{2i\pi \a_{n}})$ is connected.  Then, $G=\mathrm{GL}_{n}(\C)$.
\end{theo}

\begin{coro}
Let  $G^{\d}$, be the $\d$-Galois group of the $q$-difference system (\ref{eq3}) over the field  $\L$. Assume that for all ${(i,j)\in \{1,\dots,n\}\times  \{1,\dots,s\}}$, we have $\a_{i}-\b_{j}\notin \Z$, and that the algebraic group generated by $\mathrm{Diag}(e^{2i\pi \a_{1}},\dots,e^{2i\pi \a_{n}})$ is connected. Then, ${G^{\d}=\GL_{n}(\widetilde{\C})}$. Furthermore, if $b_{1}=q$, then the series $_{n}\Phi_{s}(\underline{a},\underline{b},\l,q;z),\dots, \sq^{n-1}\left( _{n}\Phi_{s}(\underline{a},\underline{b},\l,q;z) \right)$ are $\d$-algebraically independent over $\C(z)$.
\end{coro}

\begin{proof}
Theorems \ref{theo3} and \ref{theo2} ensure that 
$G^{\d}$ contains $\SL_{n}(\widetilde{\C})$. So, the group $G^{\d}$ is equal to $G_{M}\SL_{n}(\widetilde{\C})$, where $G_{M}\subset \widetilde{\C}^{\times}$ is the $\d$-Galois group of the $q$-difference equation $\sq y =\det(A)y=\frac{(-1)^{n}z\l+(-1)^{s+1}}{z\l \prod_{i=1}^{n} a_{i}} y$, and $A$ is the matrix associated to \eqref{eq3}. It is easily seen that there do not exist $c\in\C^{\times}$, $m\in \Z$, and $f \in \C(z)^{\times}$ such that $\det(A)=cz^{-1}\frac{\sq(f)}{f}$. By  \cite[Corollary 3.4]{HS}, we deduce that $G_{M}=\widetilde{\C}^{\times}$ and then $G^{\d}=\GL_{n}(\widetilde{\C})$. We conclude with Corollary~\ref{coro1}. 
\end{proof}

\part{$\qpr$-difference relations for solutions of $q$-difference equations }

\section{Parametrized difference Galois theory}\label{sec:parampvdiscr}

\subsection{Difference  algebra}\label{sec51}
%%%%%%%%%%%%%%%%%%%%%%%%%%%%%%%%%%%%%%%%%%%%%%%%%%%%%%%%%%%%%%%

We refer to \cite{OvWib} for more details on what follows. By a $\ssp$-ring, we mean a ring equipped with two commuting endomorphisms $\sq$ and $\sp$, such that $\sq$ is an automorphism. We do not make any  assumption on $\sp$. The definition of $\ssp$-fields, $\K$-$\ssp$-algebras for $\K$ a $\ssp$-field  and $\ssp$-ideals are straightforward.\par   
We say that a $\K$-$\ssp$-algebra $R$ is $\sp$-finitely generated if there exist $a_1,\dots,a_n$ such that $R$ is generated as $\K$-algebra by the $a_i$'s and their transforms via $\sp$. We then  write  $R =\K\{a_1,\dots,a_n \}_{\sp}$. 
We say that a $\K$-$\ssp$-field extension $R$ is $\sp$-finitely generated if there exist $a_1,\dots,a_n$ such that $R$ is generated as $\K$-field extension by the $a_i$'s and their transforms via $\sp$. We then write $R=\K\langle a_1,\dots,a_n \rangle_{\sp}$. \par 
 Let $(\k,\sp)$ be a difference field. Let $R$ be a $\k$-$\sp$-algebra. If $R$ is a field, we say that $R$ is inversive if $\sp$ is surjective on $R$. We call $R$  $\sp$-separable if $\sp$ is injective on $R \otimes_{\k} \widetilde{\k}$ for every $\sp$-field extension  $\widetilde{\k}/\k$.\par 
 The ring of $\sp$-polynomials in the differential indeterminates $y_1,\ldots,y_n$ and with coefficients in $(\k,\sp)$,  denoted by $\k\{y_1,\ldots,y_n\}_{\sp}$, is the ring of polynomials in the indeterminates $\{\sp^j y_i\:|\: j \in \N, 1\le i\le n\}$ with coefficients in $\k$. Let $R$ be a $\K$-$\sp$-algebra and let $a_1, \dots,a_n \in R$. If there exists a nonzero $\sp$-polynomial $P \in \K\{y_1,\ldots,y_n\}_{\sp}$ such that 
$P(a_1,\dots,a_n)=0$, then we say that  $a_1,\dots,a_n$ are 
$\sp$-algebraically dependent over $\K$. Otherwise, we say that $a_1,\dots,a_n$ are $\sp$-transcendental over $\K$, or $\sp$-algebraically independent over $\K$. Following \cite[A.V.141]{Bourbaki}, we say that a zero characteristic field extension $\widetilde{\k}|\k$ is a regular field extension if $\k$ is relatively algebraically closed in $\widetilde{\k}$.\par 

We would like to  prove some lemmas about the extension of constants. 
\begin{lem}\label{lem:extconst2}
  Let $F$ be a $\ssp$-field and let  $\k=F^{\sq}$ be the field of $\sq$-constants of $F$.{ We assume that $\k$ is an inversive $\sp$-field.}  Let $\widetilde{\k}$ be a regular $\sp$-field extension of  $\k$ considered as a field of $\sq$-constants. Then, the ring $\widetilde{\k} \otimes_\k F$ is an integral domain whose
  fraction field $\widetilde{F}$ is a $\ssp$-field extension of $F$ such that $\widetilde{F}^{\sq}=\widetilde{\k}$.

  \end{lem}
 
\begin{proof}

 Since $\widetilde{\k}$ is a regular extension of $\k$, the ring $\widetilde{\k} \otimes_\k F$
is an integral domain.
Moreover since  $\widetilde{\k}$ is  a $\sp$-separable $\sp$-field extension of $\k$ by  \cite[Corollary~A.14]{DVHaWib1}, the operator $\sp$ is injective on $\widetilde{\k} \otimes_\k F$ and thus extends to $\widetilde{F}$. The rest of the proof  is essentially \cite[Lemma 2.3]{DHR}.

\end{proof}

\begin{lem}\label{lem:descentsolutions}
 Let $F$ be a $\ssp$-field and let  $\k=F^{\sq}$ be the field of $\sq$-constants of $F$.{ We assume that $\k$ is an inversive $\sp$-field.}  Let $\widetilde{\k}$ be a regular  $\sp$-field extension of  $\k$ considered as a field of $\sq$-constants. By Lemma \ref{lem:extconst2}, we can consider the $\ssp$-field
 $\widetilde{F}=\mathrm{Frac}(\widetilde{\k} \otimes_\k F)$. Let ${A \in \GL_n(F)}$ and let $\V_\k$ (resp. $\V_{\widetilde{\k}}$) be the solution space of $\sq(Y)=AY$ in $F^n$ (resp. in $\widetilde{F}^n$). Then, $\V_{\widetilde{\k}} =\V_\k \otimes_\k \widetilde{\k}$.
\end{lem}
\begin{proof}
Obviously, we have $ \V_\k \otimes_\k \widetilde{\k} \subset \V_{\widetilde{\k}}$. Let $f \in \V_{\widetilde{\k}}$ be a non zero solution.  Set $S= F \otimes_\k \widetilde{\k}$. Let us consider 
$$
 \mathfrak{a}=\{ r \in S  | rf \in S \}.
$$ 
Since $\sq(f)=Af$, the  ideal $\mathfrak{a}$ is a non zero $\sq$-ideal of $S$. By \cite[Lemma~1.11]{VdPS97}, the ring $S$ is $\sq$-simple. Therefore $1 \in \mathfrak{a}$
and $f \in S$. Let $(e_i)_{i \in I}$ be a basis of $\widetilde{\k}$ over $\k$ and let us write $f =\sum_{i \in I}f_i e_i$ with $f_i \in F$.
Then, $\sq(f)=Af$ implies $\sq(f_i)=Af_i$, which ends the proof. \end{proof}

\subsection{Parametrized Difference Galois theory}\label{subsec:paramdiffgaldiscrete}

 In this section, we study the $\sp$-algebraic relations satisfied by the solutions of  $q$-difference equations over $\C(z)$. We  consider the subfield $\Kqp = \bigcup_{j=1}^\infty \C(z^{1/j})$ of the  field $\Kqpform =\bigcup_{j=1}^\infty \C((z^{1/j}))$.  Let $q$ (resp. $\tq$) be a non zero complex number  such that $|q|\neq 1$ (resp. $|\tq|\neq 1$). We choose a consistent system $(q_j)_{j \geq 1}$ (resp.$(\tq_j)_{j \geq 1}$) of roots of $q$ (resp. $\tq$);  This allows us to extend the action of $\sq$ (resp. $\sp$) to $\Kqpform$ as in \S \ref{sec hyptr}. Then,  $\Kqp$ is a $(\sq,\sp)$-subfield of $\Kqpform$ with $\Kqp^{\sq}=\Kqpform^{\sq}=\C$.\par

Given a $\ssp$-field $\K$ and $A \in \GL_n(\K)$, the $\sp$-Galois theory developed in \cite{OvWib}
aims at understanding the algebraic relations between  the solutions of $\sq(Y)=AY$ and their successive transforms 
with respect to $\sp$ from a Galoisian point of view. In this article,  we will  restrict ourselves to the case where the base field is $\Kqp$. In particular, our base field is an inversive $\sq$ and $\sp$-field, that is $\sp$ and $\sq$ are automorphisms of $\Kqp$. In this part of the paper, the word parametrized refers to the parametric action of the discrete operator $\sp$ whereas in the first part, it was related to the parametric action of the derivative.  Therefore the word parametrized does not refer to the same parametric action depending on the part of the paper. Since the two parts are almost independent, this convention will not lead to confusions. It will also avoid heavy terminology.\par
The following definition concerns the notion of ``minimal ring of solution'' in the context of parametrized difference equations. It summarizes in our context \cite[Definitions  2.2, 2.6, 2.18 and Proposition 2.21]{OvWib}.

\begin{defi}
Let  $A \in \GL_n(\Kqp)$. A  $\Kqp$-$(\sq,\sp)$-pseudofield  extension $\cQ_S$, see Remark \ref{rem:pseudofieldidempotent} below, is a $\ssp$-Picard-Vessiot extension   for $\sq(Y)=AY$ over $\Kqp$ if 
there exists a fundamental matrix $U \in \GL_n(\cQ_S)$ such that $\sq(U)=AU$, $\cQ_S=\Kqp\langle U\rangle_{\sp}$ and $\cQ_S^{\sq}=\C$. The $(\sq,\sp)$-algebra $S= \Kqp\{U, \frac{1}{\det(U)}  \}_{\sp}$ is called $\ssp$-Picard-Vessiot ring for   $\sq(Y)=AY$ over $\Kqp$. In particular, $S$ is   $\sq$-simple, i.e., it has no proper $\sq$-ideal and $\cQ_S$ is the total ring of quotients of $S$.
\end{defi}

\begin{rem}\label{rem:pseudofieldidempotent} In the above definition, the term \emph{pseudofield} needs to be explained. 
We say that a $\sq$-ring $L$ is a pseudofield if there exist 
orthogonal idempotent elements $e_1,\dots,e_r$ such that 
\begin{itemize}
\item $L=Le_1 \oplus Le_2 \oplus \hdots \oplus Le_r$,
\item $\sq(e_i)=e_{i+1 \mathrm{ mod } r}$ for any $i=1,\dots,r$,
\item $L e_i$ is a field for any $i=1,\dots,r$.
\end{itemize}
Therefore, the notation $\cQ_S=\Kqp\langle U\rangle_{\sp}$ is somehow abusive since
$\Kqp\langle U\rangle_{\sp}$ is not   the $\sp$-field generated by $U$ but the pseudofield generated by $U$ and 
its transforms with respect to $\sp$. Nonetheless, we prefer to abuse notation rather than introducing one more complicated notation.
\end{rem}

We have the following result:

\begin{lem}\label{lem:relativalgclosedbasefieldPPVfield}
Let  $A \in \GL_n(\Kqp)$ and let $\cQ_S$ be a $\ssp$-Picard-Vessiot extension  for $\sq(Y)=AY$  over $\Kqp$. If $\cQ_S$ is a field
then $\Kqp$ is relatively algebraically closed in $\cQ_S$.
\end{lem}
\begin{proof}
 Let $g \in \cQ_S$ be algebraic over $\Kqp$. Let $U$ be a fundamental solution matrix such that $\cQ_S =\Kqp \langle U \rangle_\sp$. There exists $l \in \N^*$ such that $g \in \Kqp (U, \sp(U),\dots, \sp^{l}(U))$.  Since $g$ is algebraic over $\C(z^{1/t})$ for some $t\in \N^*$ and $\C(z^{1/t})$ is a $\sq$-field, any transform $\sq^i(g)$ for $i \in \N^*$ is algebraic over $\C(z^{1/t})$. Since $\C(z^{1/t})(U, \sp(U),\dots, \sp^{l}(U))$ is a finitely generated field extension, the relative algebraic closure of $\C(z^{1/t})$
inside  $\C(z^{1/t})(U, \sp(U),\dots, \sp^{l}(U))$ is finite. Moreover, $\C(z^{1/t})(U, \sp(U),\dots, \sp^{l}(U))$ is a $\sq$-field. This proves that the $\sq$-field extension $\C(z^{1/t}) \subset \C(z^{1/t})\langle g \rangle_{\sq}$ is finite. By \cite[Proof of Proposition~12.2]{VdPS97}, there exists $m \in \N^*$ such that  $\C(z^{1/t})\langle g \rangle_{\sq} \subset \C(z^{1/tm})$. This proves that $\Kqp$ is relatively algebraically closed in $\cQ_S$.
\end{proof}

 The following proposition shows that, up to considering iterates of the operators $\sp$ and $\sq$, one can always reduce our study to the case where  the  $\ssp$-Picard-Vessiot extension $\cQ_S$ is a field and   the base field $\Kqp$ is relatively algebraically closed in $\cQ_S$. This  allows us to bypass some  difficulties in difference algebra,  that are due to algebraic extensions. 

 \begin{propo}\label{propo:existenceparamfieldsolutionqdiff}
 Let $A \in \GL_n( \Kqp)$ and let $(u_1,\dots,u_n) \in (F^*)^n$ with $F$ a $(\sq,\sp)$-field extension of $K$  such that $(u_1,\dots,u_n)^{t}$ is a solution of $\sq(Y)=AY$ and  $F^\sq=\C$.
\begin{trivlist}
\item (1)
There exist  positive integers $r,s$ and  a $(\sq^{r},\sp^s)$-Picard-Vessiot  extension $L_A$ for  the system $\sq^r(Y)=\sq^{r-1}(A) \hdots \sq(A) A Y$   over $\Kqp$ that contains $   u_1,  u_2, \dots , u_n  $,  such that $L_A$ is a field and $\Kqp$ is relatively algebraically closed in $L_A$.
\item (2) If $A \in \GL_n(\C(z))$ and  $G$ denotes the difference Galois group of the system ${\sq(Y)=AY}$ over $\C(z)$, then the difference Galois group of ${\sq^{r}(Y)=\sq^{r-1}(A) \hdots \sq(A) A Y}$ over $\Kqp$ coincides with the connected component of $G$. It is therefore connected.
\end{trivlist}
  \end{propo}
  \begin{proof}
  \begin{trivlist}
  \item (1)
Let us note that without loss of generality, we can assume that ${F=\Kqp\langle u_1,\dots,u_n \rangle_{\sp}}$.  Since $F^{\sq}=\C$ is algebraically closed and $\sp$ is surjective on $\C$, there exists a $\ssp$-Picard-Vessiot extension  ${\cQ_S= F\langle U \rangle_\sp}$ for the system
 $\sq(Y)=AY$ over $F$ by \cite[Corollary~2.29]{OvWib}.  Since $(u_1,\dots,u_n) \in F^n$, we can assume that it is the first column of $U$.  Moreover by \cite[Lemma 2.11 and Proposition 2.21] {OvWib}, the endomorphism $\sp$ is injective on $\cQ_S$. Let $e_1,\dots,e_r$ be the 
orthogonal idempotents relative to the pseudofield structure of $\cQ_S$ as in Remark~\ref{rem:pseudofieldidempotent}. It is easily seen that $\sq^r(e_i)=e_i$ for any $i=1,\dots,r$. Moreover, since $\sp$ is injective on $\cQ_S$, it permutes the orthogonal idempotents 
$e_i$ so that there exists a positive  integer $s$ dividing  $r!$ such that 
$\sp^s(e_i)=e_i$ for any $i=1,\dots,r$. This proves that $\cQ_S=K_1 \oplus \hdots \oplus K_r$
where $K_i=e_i\cQ_S$ is a $(\sq^r, \sp^s)$-field extension of $F$. If we denote by $\pi: \cQ_S \rightarrow K_1$ the projection of $\cQ_S$ on $K_1$, it is a surjective $(\sq^r, \sp^s)$-morphism. Then, denoting $U_1 =\pi(U) \in \GL_n(K_1)$, we find that $K_1=F\langle U_1 \rangle_{\sp^s}$. Since $\cQ_S^{\sq}=\C$ is algebraically closed, the elements of $\cQ_S$
fixed by any iterates of $\sq$ are in $\C$ by \cite[Remark 2.25]{OvWib} so that $K_1^{\sq^r}=\C$. To conclude note that $L_A=\C(z^*)\langle U_1 \rangle_{\sp^s} \subset K_1$ is a $(\sq^r, \sp^s)$-field extension of $\C(z^*)$ that is $\sp^s$-generated  by the entries of the fundamental solution matrix $U_1$ of the system $\sq^r(Y)=\sq^{r-1}(A) \hdots \sq(A) A Y$.
This proves that $L_A$ is a $ (\sq^{r},\sp^s)$-Picard-Vessiot  extension that is a field. Finally, we apply
Lemma \ref{lem:relativalgclosedbasefieldPPVfield} to $L_A$ replacing $\sq$ and $\sp$ by their suitable iterates.
\item (2)
Since $\Kqp$ is an algebraic extension of $\C(z)$, \cite[Theorem 7]{Ro15} implies that the difference Galois group  $G'$ of $\sq(Y)=AY$ over the $\sq$-field $\Kqp$ has the same connected component as $G$.  Let $\cQ_R=\Kqp(U_1) \subset L_A$. By \cite[Lemma~2.20]{OvWib} the field $\cQ_R$ is a Picard-Vessiot extension for ${\sq^{r}(Y)=\sq^{r-1}(A) \hdots \sq(A) A Y}$ in the sense of  \S \ref{sec:classicalgaloisgroup}.
By \cite[Theorem 12]{Ro15},   the Galois group $H$ of ${\sq^{r}(Y)=\sq^{r-1}(A) \hdots \sq(A) A Y}$ over $\Kqp$ is a normal algebraic subgroup of $G'$ and the quotient $G'/H$ is finite. To conclude, we need to prove that $H$ is connected. Let $H^{\circ}$ be its connected component, the field $\cQ_R^{{H}^{\circ}}$ is a finite extension of $\Kqp$ by the Galois correspondence for difference Galois group  \cite[Theorem~1.29]{VdPS97}.  Since $\Kqp$ is relatively algebraically closed in the field $L_A$, we find that ${\cQ_R}^{H^{\circ}}=\Kqp^H =\Kqp$ so that ${H}^{\circ}=H$ by applying the Galois correspondence again.
\end{trivlist}
\end{proof}

Unlike differential algebraic  groups, a $\sp$-algebraic group  is not entirely determined by its set of  points in some difference closure. This comes essentially from the fact that
 there are many new  type  of nilpotent elements in  difference algebra. For instance, any element $b$ such that $\sp^n(b)=0$. This last equation implies $b=0$ if and only if $\sp$ is injective, which is not necessarily the case in arbitrary $\sp$-rings. The following example illustrates the fact that one has to consider points of $\sp$-algebraic groups in arbitrary  $\sp$-rings and not only in $\sp$-fields.

\begin{ex}
Consider the following system of difference equations ${(S_1)=\{ y^2=1 \}}$ and $ (S_2)=\{y^2 =1 \mbox{ and } \sp(y)=y \}$ over $\C$. We denote by $\V_{S_1}(R)$ (resp. $\V_{S_2}(R)$) the zeros of $(S_1)$ (resp. $(S_2)$) in some $\C$-$\sp$-algebra $R$. Then, ${\V_{S_1 }(\k)=\V_{S_2 }(\k)=\{1,-1\}}$ for any $\sp$-field extension $\k$ of $\C$.
However, if we consider the ring of sequences $ (a_n)_{n \in \Z} \in \C^\Z$ with the action of $\sp$ given by the shift operator, then  $\V_{S_1 }(\C^\Z)=\{(a_n)_{n \in \Z}| a_n =1 \mbox{ or } -1 \mbox{ for all } n \in \Z\}$ whereas $\V_{S_2 }(\C^\Z)$ is the union of the constant sequence $1$ and the  constant sequence $-1$.  
\end{ex}
 Therefore, we need to adopt the following  functorial approach. We denote by $\Alg_{\C,\sp}$ the category of $\C$-$\sp$-algebras
 and by $\Sets$ the category of sets. 
 
 \begin{defi}[\cite{OvWib}, Definition 2.50]\label{def:spgaloisgroup}
 Let $A \in \GL_n(\Kqp)$ and let ${\cQ_S=\Kqp\langle U\rangle_{\sp}}$ be a $\sp$-PV extension for $\sq(Y)=AY$ over $\Kqp$. Set $S=\Kqp\{U,\frac{1}{\det(U)}\}_{\sp}$. Then, the $\sp$-Galois group of $\cQ_S$ over $\Kqp$ is defined as the functor:

$$\begin{array}{llll}
 \spGal(\cQ_S/\Kqp):& \Alg_{\C, \sp} &\rightarrow &\Sets \\
&B &\mapsto&   \Aut^{\sp}(S \otimes_\C B/ \Kqp \otimes_\C B),
\end{array}$$
where, $\sq$ acts as the identity on $B$ and $\Aut^{\ssp}(S \otimes_\C B/ \Kqp \otimes_\C B)$ is the group of automorphisms of $S \otimes_\C B$ inducing the identity on $ \Kqp \otimes_\C B$ and commuting with $\sp$ and $\sq$.
 \end{defi}
 It is proved in \cite[Lemma 2.51]{OvWib} that this functor is represented by a  finitely $\sp$-generated $\C$-$\sp$-Hopf algebra 
 $\C\{ \spGal(\cQ_S/\Kqp)\}$ (see  Definition~\ref{defi:sHopfalgebra}). Therefore,   $\spGal(\cQ_S/\Kqp)$ is a $\sp$-algebraic group  (see Definition~\ref{defi:sgroupscheme}). For a brief introduction to $\sp$- algebraic groups,  we refer to Section~$ \ref{secA2}$. Any algebraic group $G$ over $\C$ gives rise to a $\sp$-algebraic group  $\bold{G}$ over $\C$ by Proposition \ref{propo:algschelmediffschemezarclosure}. We would like to recall that, since they are not defined with respect to the same geometry, the $\sp$-algebraic group  $\bold{G}$ should not be confused with the  algebraic group  $G$. \par

  In the notation of Definition \ref{def:spgaloisgroup}, if $B$ is a $\C$-$\sp$-algebra, then the matrix ${U \otimes 1 \in \GL_n(S \otimes_\C B)}$ is a fundamental matrix of solutions of $\sq(Y)=AY$ in $S \otimes_\C B$. Then, for any $\phi \in \spGal(\cQ_S/\Kqp)(B)$, the matrix $\phi(U \otimes 1)$ is also a fundamental matrix of solutions of $\sq(Y)=AY$ in $S\otimes_\C B$. Thus, there exists ${[\phi]_U \in \GL_n((S \otimes_\C B)^{\sq})=\GL_{n,\C}(B)}$ such that ${\phi (U \otimes 1)= (U\otimes 1) [\phi]_U}$. Here $\bold{\GL_{n,\C}}$ is the $\sp$-algebraic group corresponding to the general linear algebraic group of size $n$ over $\C$ (see Example~\ref{exa:fundadiffsubgroupGln}).  
 
 \begin{propo}\label{propo:repspgaloisgroup}
The functor  $\rho_U:$
$$
\begin{array}{lll}
\spGal (\cQ_S/\Kqp)& \rightarrow& \bold{\GL_{n,\C}}\\
\phi \in \spGal (\cQ_S/\Kqp)(B) &\mapsto &[\phi]_U \in GL_{n,\C}(B),
\end{array}
 $$
where $B\in \Alg_{\C,\sp}$ is a $\sp$-closed embedding (see \cite[Definition A.3]{DVHaWib1}).
 \end{propo}
 \begin{proof}
The proof is  the exact  analogue of \cite[Proposition 2.5]{DVHaWib1}  and its proof is between the lines of  \cite[Lemma 2.51]{OvWib}.
 \end{proof}
 
 This proposition allows to identify the $\sp$-Galois group with a $\sp$-subgroup  of $\bold{\GL_{n,\C}}$ via the choice
 of a fundamental matrix of solutions $U$. Another choice of fundamental matrix of solutions leads to a conjugate representation. Therefore, $\spGal(\cQ_S/\Kqp)$ is entirely determined  by a $\sp$-Hopf  ideal $\mathfrak{I}$ of  $ \C\{\bold{\GL_{n,\C}}\}=\C\{X,\frac{1}{det(X)}\}_{\sp}$ (see Example \ref{exa:fundadiffsubgroupGln}). The elements of $\mathfrak{I}$ are $\sp$-polynomials and  we call  them   the defining equations of $\spGal(\cQ_S/\Kqp)$  in $\bold{\GL_{n,\C}}$.
 
 In $\sp$-Galois theory, one has a complete Galois correspondence (\cite[Theorem 2.52 and Lemma 2.53]{OvWib}). We only recall the following results.
 \begin{propo}\label{propo:spgaloiscorresptransdeg}
 Let $A \in \GL_n(\Kqp)$ and let $\cQ_S $ be a $\ssp$-Picard-Vessiot extension of $\sq(Y)=AY$ over $\Kqp$. Then,
\begin{multline*}
\cQ_S^{ \spGal(\cQ_S/\Kqp)}=\\
\{ x=\frac{r}{s} \in \cQ_S | \forall B \in  \Alg_{\C, \sp}, \forall g \in \spGal(\cQ_S/\Kqp)(B),\\ g(r \otimes 1). (s\otimes 1)=(r \otimes 1).(g(s\otimes 1))\}=\Kqp. 
\end{multline*}
Moreover, we have $\spdim(\spGal(\cQ_S/\Kqp))=\sptrdeg(\cQ_S/\Kqp)$ (for precise definitions see \cite[\S A.7]{DVHaWib1}). 
 \end{propo} 
 The last equality means that the complexity of the defining equations of  $\spGal(\cQ_S/\Kqp)$ corresponds
 precisely to the complexity of the $\sp$-difference algebraic relations satisfied by the solutions of the system 
 $\sq(Y)=AY$ in $\cQ_S$. \\ \par 
 
The relation between  the $\ssp$-Picard-Vessiot theory and the  non parametrized Picard-Vessiot theory as developed in \cite{VdPS97} is  explained below.

\begin{propo}\label{propo:schematicalgebraicgaloisgroupcomparaison}
 Let $A \in \GL_n(\Kqp)$ and let $\cQ_S$ be a $\sp$-PV extension of $\sq(Y)=AY$ over $\Kqp$. Set $R=\C(z^*)[U,\frac{1}{\det(U)}] \subset \cQ_S$ and denote by $\cQ_R$ the total ring of quotients of $R$. The following holds:
 
 \begin{itemize}
 \item The $\Kqp$-$\sq$-algebra $\cQ_R$ is a Picard-Vessiot extension for ${\sq(Y)=AY}$ over $\Kqp$ as in $\S \ref{sec:classicalgaloisgroup}$;
 \item The  $\sp$-Galois group $ \spGal(\cQ_{S}/\Kqp)$ is a Zariski dense subgroup of $ \Gal(\cQ_{R}/\Kqp)$ (see Proposition \ref{propo:algschelmediffschemezarclosure}).
 \end{itemize}
 
\end{propo}
 \begin{proof}
 The first statement is \cite[Lemma 2.20]{OvWib} and the second statement is a discrete analogue of  \cite[Proposition~2.15]{DVHaWib1}.\end{proof}
 
If the matrix $A \in \GL_n(\C(z))$ and the $\ssp$-Picard-Vessiot extension $\cQ_S$ is a field, one can relate the difference Galois group of $\sq(Y)=AY $ over $\C(z)$ and the difference Galois group of the system over $\Kqp$ as follows.

\begin{lem}\label{lem:compgroupeetgroupconnected}
Let $A \in \GL_n(\C(z))$ and let $\cQ_S$ be a $\sp$-PV extension of ${\sq(Y)=AY}$ over $\Kqp$. If $\cQ_S$ is a field, then the difference Galois group of $\sq(Y)=AY$ over $\Kqp$ equals the connected component of the difference Galois group of the system over $\C(z)$.
\end{lem}
 \begin{proof}
 The proof is completely analogous to the last paragraph of the  proof of Proposition \ref{propo:existenceparamfieldsolutionqdiff} and relies on the fact that  $\Kqp$ has no non trivial finite $\sq$-field extensions. 
 \end{proof}

\subsection{Discrete Isomonodromy}
In $\sp$-Galois theory, one can define a notion of discrete isomonodromy as follows.

\begin{defi}\label{defi:isodiscrete}
Let $A \in \GL_n(\Kqp)$. The system $\sq(Y)=AY$  is called $\sp$-isomonodromic if there exist $B \in \GL_n(\Kqp)$ and $d \in \N^*$ such that 
\begin{equation}\label{eq:isodicrete}
\sq(B)A =\sp^d(A)B.
\end{equation}
\end{defi}
\begin{rem}\label{rem:isomondromyconnectionovwibmer}
Our definition is slightly more general than in \cite[Definition~2.54]{OvWib}, where $\sp$-isomonodromic means that there exists ${B \in \GL_n(\Kqp)}$  such that  $\sq(B)A =\sp(A)B$, i.e., $d=1$ in our definition. However, we can apply  most of the results of \cite{OvWib} by replacing $\sp$ by $\sp^d$.
\end{rem}

 We have the following Galoisian interpretation of $\sp$-isomonodromy.  We say that a $\sp$-subgroup   $H \subset \bold{\GL_{n,\k}}$ defined over a $\sp$-field $\k$ 
 is $\sp^d$-constant if, for all $\k$-$\sp$-algebras $S$, we have  $\sp^d(g)=g$, for all  $ g\in H(S)$. This is equivalent to the fact that
 the defining ideal  $\mathfrak{I}_H \subset \k\{X,\frac{1}{\det(X)}\}_{\sp}$ of $H \subset \bold{\GL_{n,\k}}$ contains the polynomial $\sp^d(X)-X$ (see Example \ref{exa:definingideal}).

 \begin{propo}\label{propo:caracgalspisomono}
 Let $A \in \GL_n(\Kqp)$ and let $\cQ_S$ be a $\sp$-PV  extension for $\sq(Y)=AY$ over $\Kqp$. Assume that $\cQ_S$ is a field. The system ${\sq(Y)=AY}$  is  $\sp$-isomonodromic over $\Kqp$ if and only if there exists a regular $\sp$-field extension $\widetilde{\C}$ of $\C$  and an integer $d \geq 1$ such that $\spGal(\cQ_S/\Kqp)_{\widetilde{\C}}$\footnote{The subscript $\widetilde{\C}$ means that we consider the  base change of $\spGal(\cQ_S/\Kqp)$ over $\C$ to $\widetilde{\C}$.} is conjugated to a $\sp^d$-constant subgroup of $\bold{\GL_{n,\widetilde{\C}}}$. 
 \end{propo}
 
 We refer to \cite[Theorem 2.55]{OvWib} for an analogous result in a different setting. \par 
Note that, since $\C$ is a $\sp$-inversive field, \cite[Corollary~A.14]{DVHaWib1} implies that any field extension of $\C$ is $\sp$-separable (see Section $ \ref{sec51}$).
 
 Before proving Proposition \ref{propo:caracgalspisomono}, we need an intermediate lemma about extension of $\sq$-constants. We have the following result:
\begin{lem} \label{lem:baseextensionforsPV}
 Let $\widetilde{\C}$ be a  $\sp$-field extension of $\C$ and let $\cQ_S/ \Kqp$ be a $\sp$-PV extension  for $\sq(Y)=AY$. 
 By Lemma \ref{lem:extconst2}, we may consider $\widetilde{\Kqp}$ (resp. $\widetilde{\cQ_S}$) the $\ssp$-field attached to $\Kqp\otimes_{\C} \widetilde{\C}$  (resp. $\cQ_S \otimes_{\C} \widetilde{\C}$).
Then $\widetilde{\cQ_S}$ is a $\ssp$-Picard-Vessiot extension for $\sq(Y)=AY$ over $\widetilde{\Kqp}$ and the $\sp$-Galois group $\widetilde{G}$ of $\widetilde{\cQ_S}/ \widetilde{\Kqp}$ is obtained from the $\sp$-Galois group $G$ of $\cQ_S/\Kqp$ by base extension, {\it i.e.}, ${\widetilde{G}=G_{\widetilde{\C}}}$.
\end{lem}

\begin{proof}[Proof of Lemma \ref{lem:baseextensionforsPV}]
As $\widetilde{\cQ_S}^{\sq}=\widetilde{\C}=\widetilde{\Kqp}^{\sq}$, it is clear that $\widetilde{\cQ_S}|\widetilde{\Kqp}$ is a $\ssp$-Picard-Vessiot extension.
Let $S \subset \cQ_S$, (resp. $\widetilde{S}\subset\widetilde{\cQ_S}$), denotes the corresponding $\ssp$-Picard-Vessiot ring. Then $\widetilde{S}$ is obtained from $S\otimes_{\C}\widetilde{\C}$ by localizing at the multiplicatively closed set of all non-zero divisors of $\Kqp\otimes_{\C}\widetilde{\C}$. 
It follows that, for every $\widetilde{\C}$-$\sp$-algebra $B$,
$$
\begin{array}{lll}
G_{\widetilde{\C}}(B)&=&\Aut^{\ssp}(S\otimes_{\C} B|\Kqp\otimes_{\C} B)\\
&=&\Aut^{\ssp}\Big((S\otimes_{\C}\widetilde{\C})\otimes_{\widetilde{\C}}B\Big| (\Kqp\otimes_{\C}\widetilde{\C})\otimes_{\widetilde{\C}}B\Big),
\end{array}$$
{\it i.e.},
$$G_{\widetilde{\C}}(B)=
\Aut^{\ssp}(\widetilde{S}\otimes_{\widetilde{\C}}B|\widetilde{\Kqp}\otimes_{\widetilde{\C}}B)=\widetilde{G}(B).
$$
This ends the proof.
\end{proof}

\begin{proof}[Proof of Proposition \ref{propo:caracgalspisomono}]
 In \cite[Theorem 2.55]{OvWib}, it is proved that if  the system is $\sp$-isomonodromic  then there exists a $\sp$-field extension $\widetilde{\C}$ of $\C$  and an integer $d \geq 1$ such that $\spGal(\cQ_S/\Kqp)_{\widetilde{\C}}$ is conjugated to a $\sp^d$-constant subgroup of $\bold{\GL_{n,\widetilde{\C}}}$ (see Remark \ref{rem:isomondromyconnectionovwibmer}). In the proof of \cite[Theorem 2.55]{OvWib}, we note that any $\sp$-field extension $\widetilde{\C}$ of $\C$ that contains
a fundamental matrix of solutions of a given  equation of the form ${\sp^d(Y)=DY}$ for some given $D \in \GL_n(\C)$ is convenient. We claim that we can find
among these extensions a regular one. Indeed consider $\widetilde{\C}= \C(X_0,\dots,X_{d-1})$ where the $X_i$'s are $n \times n$-matrices of indeterminate. We can endow $\widetilde{\C}$ with a structure of $\sp$-extension of $\C$ by setting $\sp(X_i)=X_{i+1}$ for $i=0,\dots, d-1$
and $\sp(X_{d-1})= DX_0$.  Then, $X_0 \in \GL_n(\widetilde{\C})$ is a solution of $\sp^d(X_0)=DX_0$ and since $\widetilde{\C}$ is 
a pure extension of $\C$, it is also a regular extension.

Conversely, let us assume that there exists a regular $\sp$-field extension $\widetilde{\C}$ of $\C$  and an integer $d \geq 1$ such that $\spGal(\cQ_S/\Kqp)_{\widetilde{\C}}$ is conjugated to a $\sp^d$-constant subgroup of $\bold{\GL_{n,\widetilde{\C}}}$. Endow $\widetilde{\C}$ with a structure  of $\sq$-constants field and consider the $\ssp$-fields $\widetilde{\cQ_S}$ and $\widetilde{\Kqp}$ as in Lemma \ref{lem:baseextensionforsPV}. We find that the $\sp$-Galois group of  $\widetilde{\cQ_S}$ over $\widetilde{\Kqp}$
equals $\spGal(\cQ_S/\Kqp)_{\widetilde{\C}}$ and is thus conjugate to a $\sp^d$-constant group over $\widetilde{\C}$.
By \cite[Theorem 2.55]{OvWib}, the system $\sq(Y)=AY$ is $\sp$-isomonodromic over $\widetilde{\C (z^*)}$, {\it i.e.}, there exist $\widetilde{B} \in \GL_n(\widetilde{\C (z^*)})$ and $d \in \N^\times$ such that $\sq(\widetilde{B})=\sp^d(A)\widetilde{B}A^{-1}$.   By Lemma \ref{lem:descentsolutions}, the solution space in $\widetilde{\C(z^*)}^{n\times n}$ of the $q$-difference equation $\sq(Y)= \sp^d(A)YA^{-1}$ is generated as a $\widetilde{\C}$-vector space by the solution space of the equation in $\Kqp^{n\times n}$. Since the condition $\det(Y) \neq 0$ is an open condition, there exists
$B \in \GL_n(\Kqp)$ such that $\sq(B)=\sp^d(A)BA^{-1}$ and the system $\sq(Y)=AY$ is $\sp$-isomonodromic over $\Kqp$.
\end{proof}

\subsection{Transcendence results}\label{sec53}
 Let $A \in \GL_n(\Kqp)$ and consider 
\begin{equation}\label{eq11}
\sq Y=AY.
\end{equation}
Let $\cQ_S$ be a  $\ssp$-Picard-Vessiot extension of $\sq(Y)=AY$ over $\Kqp$. Let $U \in \GL_{n}(\cQ_S)$ be a fundamental matrix of solutions of the system \eqref{eq11}, and let $\Gal^{\sp}(\cQ_S/\Kqp)$ be the $\sp$-Galois group of $\cQ_S$ identified  with a $\sp$-subgroup of $\GL_{n,\C}$ via the faithful representation attached to the fundamental matrix of solutions $U$.

Let $\bold{\SL_{n,\C}}$ (when $n\geq 2$), $\bold{\mathrm{SO}_{n,\C}}$ (when $n\geq 3$) and $\bold{\mathrm{Sp}_{n,\C}}$ (when $n$ is even) be  the
$\sp$-algebraic groups over $\C$ corresponding respectively to the special linear group, the special orthogonal group and the symplectic group (see
Section \ref{secA2}).

\begin{propo}
Assume that $n\geq 2$. Assume that $\cQ_S$ is a field. Let $u=(u_{1},\dots,u_{n})$ be a row (resp. column) vector  of $U$. If there exists
${\tilde{C} \in \GL_n(\C)}$ such that the image of the $\sp$-Galois group by the representation $\rho_{U\tilde{C}}$ associated to the fundamental matrix of solutions $U\tilde{C}$ contains
\begin{itemize}
\item $\bold{\SL_{n,\C}}$ or $\bold{\Sp_{n,\C}}$, then $u_{1},\dots,u_{n}$ are $\sp$-algebraically independent over $\Kqp$;
\item $\bold{\SO_{n,\C}}$,  then any  $n-1$  distinct elements among the $u_i$'s are   $\sp$-algebraically independent over $\Kqp$.
\end{itemize}
\end{propo}
\begin{proof}
The proof is a discrete analogue of the proof of Proposition \ref{propo:transcontinu}. We just  explain the strategy of the proof in the $\bold{\SL_{n,\C}}$-case to show where one has to adapt the proof to the discrete parameter case. Let $X=(X_{i,j})_{1 \leq i,j \leq n}$ be $\sp$-indeterminates. Let $\mathfrak{I}$ be the kernel of the unique morphism of $\K$-$\sp$-algebras 
$
\K\{X, \frac{1}{\det(X)} \}_{\sp} \rightarrow S=\C(z^*)\{U, \frac{1}{\det (U)} \}_\sp
$ 
such that $X \mapsto U$. We denote by $(x_{1},\ldots,x_{n})=(X_{1,1},\ldots,X_{1,n})$ the first row of $X$.  The $\sp$-algebraic relations with coefficients in $\Kqp$ between $u_{1},\ldots,u_{n}$ correspond to the elements of $\mathfrak{I} \cap  \K\{x_1,\dots,x_{n}\}_{\sp}$. So everything amounts to prove that $\mathfrak{I} \cap  \K\{x_1,\dots,x_{n}\}_{\sp} = \{0\}$. In order to proving this, we will relate $\mathfrak{I}$ to the ideal defining the $\sp$-algebraic group $\Gal^{\sp}(\cQ_S/\Kqp)$. Such a relation follows from the fact that the $\sp$-PV ring $S$ is the coordinate ring of a  $\Gal^\sp(\cQ_S/\Kqp)$-torsor over $\Kqp$.  

We shall now give the details of the proof, still in the $\bold{\SL_{n,\C}}$-case. As above, we let $\mathfrak{I}$ be the kernel of the unique morphism of $\Kqp$-$\sp$-algebras 
$
\varphi : \Kqp\{X, \frac{1}{\det(X)} \}_{\sp} \rightarrow S
$ 
such that $X \mapsto U$ and we denote by $\mathcal{V}$ the $\sp$-algebraic variety over $\Kqp$ defined by $\mathfrak{I}$.  On the other hand, we let $G$ be the image of $\Gal^\sp(\cQ_S/\Kqp)$ by the representation $\rho_{U}$, we let ${\mathfrak{L}}$ be the $\sp$-ideal of $\C\{X, \frac{1}{\det(X)} \}_{\sp}$ of the equations of $G$ and we let $\mathcal{G}$ be the $\sp$-algebraic variety over $\Kqp$ defined by ${\mathfrak{L}}$; in other words, $\mathcal{G}$ is the $\sp$-algebraic group   over $\Kqp$ obtained from $G$ by extension of scalars from $\C$ to $\Kqp$. Both $\mathcal{V}$ and $\mathcal{G}$ can be seen in $\bold{\GL_{n,\Kqp}}$.   The following map is well-defined and makes $\mathcal{V}$ a $\mathcal{G}$-torsor over $\Kqp$, see \cite[Lemmas 2.49 and 2.51]{OvWib}: 
\begin{eqnarray*}
 \mathcal{V} \times_{\K} \mathcal{G} & \rightarrow & \mathcal{V} \times_{\K} \mathcal{V} \\ 
 (v,M) & \mapsto & (v,vM).
 \end{eqnarray*}
The $\ssp$-Picard-Vessiot extension $\cQ_S$ is a $\sp$-field extension of $\Kqp$. The injection of $S$ into $\cQ_S$ yields to a point of  $\mathcal{V}(\cQ_S)$. This proves that the torsor $\mathcal{V}$ is trivial over $\cQ_S$. 
There exists  $Z_{0} \in \mathcal{V}(\cQ_S)$ such that the  $\sp$-ideals  $(\mathfrak{I})$ and $({\mathfrak{L}})$ of $\cQ_S\{X, \frac{1}{\det(X)} \}_{\d} $ defining $\mathcal{V}_{\cQ_S}$ and $\mathcal{G}_{\cQ_S}$ satisfy
$$(\mathfrak{I})= \{ P(Z_0^{-1}X) \ \vert \ P \in ({\mathfrak{L})}\}.$$

Since the image of   $\Gal^{\sp}(\cQ_S/\Kqp)$ by the representation $\rho_{U\tilde{C}}$  contains $\bold{\SL_{n,\C}}$, we see that $G$ contains ${H=\tilde{C} \bold{\SL_{n,\C}} \tilde{C}^{-1}(= \bold{\SL_{n,\C}}))}$.  Hence, $(\mathfrak{I})$  is contained in the  $\sp$-ideal $\{\det(X)-\det(Z_0)\}$ of {$\cQ_S\{X, \frac{1}{\det(X)} \}_{\sp}$} generated by $\det(X)-\det(Z_0)$ (see \cite[Example 12.12]{Wibaffinediffgroup}).

 We claim the equality of ideals
$$
\{\det(X)-\det(Z_0)\}_{\sp} \cap  \cQ_S\{x_1,\dots,x_{n}\}_{\sp}  =\{0\}.
$$
Indeed, let us consider $${P=P(X)=P(x_{1},\ldots,x_{n}) \in \{\det(X)-\det(Z_0)\} \cap  \cQ_S\{x_1,\dots,x_{n}\}_{\sp}}.$$ 
Let $\widetilde{\K}$ be any $\sp$-field extension of $\cQ_S$.
For any non zero vector ${(a_{1},\ldots,a_{n}) \in \widetilde{\K}^{n} \setminus \{(0,\ldots,0)\}}$, there exists a matrix $A \in M_{n}(\widetilde{\K})$ with first row $(a_{1},\ldots,a_{n})$ such that ${\det(A)=\det(Z_{0})}$, so ${P(a_{1},\ldots,a_{n})=P(A)=0}$ because $P \in \{\det(X)-\det(Z_0)\}_{\sp}$. Therefore, $P$ vanishes on  ${\widetilde{\K}^{n} \setminus \{(0,\ldots,0)\}}$ and, hence, $P=0$  by \cite[Theorem 2.6.5]{Levin}.
We now have the desired result because  

\begin{multline*}
\mathfrak{I} \cap \Kqp\{x_1,\dots,x_{n}\}_{\sp} \subset (\mathfrak{I}) \cap  \cQ_S\{x_1,\dots,x_{n}\}_{\sp}\\\subset \{\det(X)-\det(Z_0)\}_{\sp} \cap  \cQ_S\{x_1,\dots,x_{n}\}_{\sp}  =\{0\}.
\end{multline*}
\end{proof}

\section{$q$-difference equations of rank one}\label{sec6}

We recall that $q,\qpr\in \C^{\times}$ with $|q|,|\qpr|\neq 1$. From now, we assume that $q$ and $\qpr$ are multiplicatively independent, that is for any $\ell,m\in \Z$, $q^{\ell}\qpr^{m}=1 \Rightarrow \ell=m=0$. In other words, $\log(q/ \qpr)\notin \Q$. The goal of the section is to compute the $\sp$-Galois group of order one equation.
For any $a(z) \in \C(z)$, we denote by $\divi a(z)$ the divisor of $a(z)$ on $\C^{\times}$, {\it i.e.}, 
$$
\divi a(z) = \sum_{\alpha \in \C^{\times}} v_{\alpha}(a(z))\crochets{\alpha} 
$$
where $v_{\alpha}(a(z))$ denotes the valuation of $a(z)$ at $\alpha$.  
Let $\pi : \C^{\times} \rightarrow \C^{\times}/q^{\Z}$ be the natural projection. For any $a(z) \in \C(z)^{\times}$, we set   
$$
\qdivi a(z) = \sum_{\alpha \in \C^{\times}} v_{\alpha}(a(z))\crochets{\pi(\alpha)}. 
$$
The proof of the following lemma is inspired by the proof of \cite[Lemma~2.1]{VdPS97}.

\begin{lem}[Lemme 3.5 in \cite{hardouin2008hypertranscendance}]   \label{lem qdivisor}
Consider $a(z) \in \C(z)^{\times}$. Then, the following properties are equivalent:
\begin{itemize}
\item[(i)] there exist $c \in \C^{\times}$, $m \in \Z$ and $b(z) \in \C(z)^{\times}$ such that ${a(z)=cz^{m}\frac{b(qz)}{b(z)}}$; 
\item[(ii)] $\qdivi a(z)=0$.
\end{itemize}
\end{lem}

\begin{propo}\label{prop:aetb}
Let $a(z), b(z) \in \C(z)^{\times}$ be such that 
$$
a(z)^{k_{0}} a(\qpr z)^{k_{1}} \cdots a(\qpr^{r} z)^{k_{r}}= \frac{b(qz)}{b(z)}
$$
for some $k_{0},\ldots,k_{r} \in \Z$ with $k_{0}k_{r} \neq 0$. Then, $\qdivi a(z)=0$, {\it i.e.}, in virtue of Lemma~\ref{lem qdivisor}, there exist $c \in \C^{\times}$, $m \in \Z$ and $b_{1}(z) \in \C(z)^{\times}$ such that ${a(z)=cz^{m}\frac{b_{1}(qz)}{b_{1}(z)}}$. 
\end{propo}

\begin{proof}
Assume to the contrary that $\qdivi a(z)\neq 0$. 
We set 
$$
\qdivi a(z)=\sum_{i=1}^{m} n_{i}\crochets{\zeta_{i}}
$$
where the $\zeta_{i}$ are pairwise distinct elements of $\C^{\times}/q^{\Z}$ and the $n_{i}$ are non zero integers. 
We have 
\begin{equation}\label{bla}
0=\qdivi \frac{b(qz)}{b(z)} = \qdivi a(z)^{k_{0}} a(\qpr z)^{k_{1}} \cdots a(\qpr^{r}z)^{k_{r}}
=
\sum_{j=0}^{r} k_{j} \sum_{i=1}^{m} n_{i} \crochets{\qpr^{-j} \zeta_{i}}. 
\end{equation}
Let 
$$
I=\{i \in \{1,\ldots,m\} \ \vert \ \zeta_{i} \in \qpr^{\Z}\zeta_{1}\}.
$$ 
Let $i_{1},\ldots,i_{s}$ be pairwise distinct integers such that $I=\{i_{1},\ldots,i_{s}\}$. Up to renumbering, we can assume that 
$$
\zeta_{i_{1}} \prec \cdots \prec \zeta_{i_{s}}
$$
where, for any $x,y \in \C^{\times}/q^{\Z}$, $x \prec y$ means that $y = \qpr^{k} x$ for some $k \in \N^{*}$. 
Then, we have $\qpr^{-r} \zeta_{i_{1}} \prec \qpr^{-j} \zeta_{i_{k}}$ for all $j \in \{0,\ldots,r\}$ and $k \in \{1,\ldots,s\}$ such that $(j,k)\neq (r,1)$.  In particular, $\qpr^{-r} \zeta_{i_{1}} \neq \qpr^{-j} \zeta_{i}$ for all $j \in \{0,\ldots,r\}$ and $i \in I$ such that $(j,i)\neq (r,i_{1})$ (indeed, if $x \prec y$ then $x \neq y$ because $q$ and $\qpr$ are multiplicatively independent).

Moreover, for $j \in \{0,\ldots,r\}$ and $i \in \{1,\ldots,m\}\setminus I$, $\qpr^{-r} \zeta_{i_{1}}$ and $\qpr^{-j} \zeta_{i}$ are not in the same $\qpr^{\Z}$-orbit and hence are not equal. 

So, we have proved that $\qpr^{-r} \zeta_{i_{1}} \neq \qpr^{-j} \zeta_{i}$ for all $j \in \{0,\ldots,r\}$ and ${i \in \{1,\ldots,m\}}$ such that $(j,i)\neq (r,i_{1})$. Therefore, the coefficient of $\crochets{\qpr^{-r} \zeta_{i_{1}}}$ in equation (\ref{bla}) is equal to $0$, {\it i.e.}, $k_{r}n_{i_{1}}=0$, whence a contradiction. 
\end{proof}

The  following proposition gives an example of $\sp$-isomonodromic equation of rank one.

\begin{propo}\label{propo:sprang1}
Let $a(z)\in \C(z)^{\times}$. Let $\cQ_S$ be a  $\ssp$-Picard-Vessiot extension for $\sq(y)= a(z) y$ over $\C(z^*)$. Let $u \in \cQ_s$ be a non zero solution of 
$\sq(y) = a(z) y $. Let $\cQ_S=\Kqp\langle u \rangle_{\sp}$. Then, the following statements are equivalent:
\begin{enumerate}
\item  $u$ and all its transforms with respect to $\sp$ are algebraically dependent over $\Kqp$,
\item there exist $c \in \C^{\times}$, $m \in \Z$ and $b(z) \in \C(z)^{\times}$ such that ${a(z)=cz^{m}\frac{b(qz)}{b(z)}}$,

\item the  group $\spGal(\cQ_S/\Kqp)$ can be embedded as a subgroup of ${H \subset \bold{\GL_{1,\C}}}$ with $H$ a $\sp$-algebraic subgroup defined by 
$$H(B)=\left\{\lambda \in \GL_{1,\C}(B) \left| \sp\left(\frac{\sp(\lambda)}{\lambda} \right) =\frac{\sp(\lambda)}{\lambda}\right\}\right.$$ for any $ B \in \Alg_{\C,\sp}$.
\end{enumerate}
Moreover, the following statements are equivalent:
\begin{trivlist}
 \item (a) there exist $c \in \C^{\times}$and $b(z) \in \C(z)^{\times}$ such that $a(z)=c\frac{b(qz)}{b(z)}$,
\item (b) the  group $\spGal(\cQ_S/\Kqp)$ is $\sp$-constant.
\end{trivlist}
\end{propo}
 \begin{proof}
 
 Let us prove $(1) \Rightarrow (2)$.
 Relying on the classification of the $\sp$-algebraic subgroups of $\bold{\GL_{1,\C}}$, \cite[Theorem 3.1]{OvWib}
 ensures that the first statement is equivalent to the existence of $b(z) \in \C(z)^{\times}$, $t \in \N$ and  $n_0, \dots,n_{t} \in \Z$ not all zero, such that the following equation holds \begin{equation} \label{eq:relation rank1}
 a(z)^{n_0}\sp(a(z)^{n_1}) \cdots \sp^t(a(z)^{n_t})=\frac{\sq(b(z))}{b(z)}.
 \end{equation}
 Proposition \ref{prop:aetb} shows then that the first statement implies the second.\par 
 Let us prove $(2) \Rightarrow (3)$. Assume that the second statement holds. By \cite[Proposition 4.9]{OvWib}, for any $B \in \Alg_{\C,\sp}$ and $g \in \spGal(\cQ_S/\Kqp)(B)$ there exists $\lambda_g \in B^{\times}$
 such that $g(u) =\lambda_g.u$. Since $a(z)=cz^m\frac{b(qz)}{b(z)}$, an easy computation gives
\begin{multline*}
\sq \left(\frac{\sp (u)}{u} \times \frac{b}{ \sp (b)} \right)=\frac{\sp (a)\sp (u) \sq(b)}{au \sp (\sq(b))}\\
=\qpr^{m}\frac{\sp (u)}{u}\frac{\sp (\sq(b))b\sq (b)}{\sp(b)\sq(b)\sp (\sq(b))}
=\qpr^{m}\frac{\sp (u)}{u}\times\frac{b}{\sp(b)}.
\end{multline*}
Therefore, if we set $h=\frac{\sp(b)}{b}$, we find
 $$\sq\left( \frac{\sp(\frac{\sp(u)}{u}) }{\frac{\sp(u)}{u}}.\frac{h}{\sp(h)}\right)= \frac{\sp(\frac{\sp(u)}{u}) }{\frac{\sp(u)}{u}}.\frac{h}{\sp(h)}.$$
 Since $\cQ_S^{\sq}=\C$, there exists $d \in \C$ such that we have the equality ${\frac{\sp(\frac{\sp(u)}{u}) }{\frac{\sp(u)}{u}}=d \frac{\sp(h)}{h}}$, {\it i.e.}, $\frac{\sp(\frac{\sp(u)}{u}) }{\frac{\sp(u)}{u}} \in \Kqp$ and is left invariant by the $\sp$-Galois group. This implies that 
 for any $B \in \Alg_{\C,\sp}$ and ${g \in \spGal(\cQ_S/\Kqp)(B)}$, we find $\sp(\frac{\sp(\lambda_g)}{\lambda_g})= \frac{\sp(\lambda_g)}{\lambda_g}$ and we deduce that the $\sp$-Galois group can be represented as a subgroup of $H$.\par 
 Let us prove $(3) \Rightarrow (1)$. If the third statement holds, then $\spGal(\cQ_S/\Kqp)$ is a strict subgroup of $\bold{\GL_{1,\C}}$. By Proposition  \ref{propo:spgaloiscorresptransdeg},
this implies that  $u$ and all its transforms with respect to $\sp$ are algebraically dependent over $\Kqp$. This proves $(1)$.\par 

Let us prove $(a) \Rightarrow (b)$. If there exist $c \in \C^{\times}$ and $b(z) \in \C(z)^{\times}$ such that $a(z)=c\frac{b(qz)}{b(z)}$ then 
$\frac{\sp(a)}{a}=\frac{\sq(h)}{h}$ where $h(z)=\frac{\sp(b(z))}{b(z)}$. Proposition \ref{propo:caracgalspisomono} allows to conclude that the  group $\spGal(\cQ_S/\Kqp)$ is $\sp$-constant. Let us prove ${(b) \Rightarrow (a)}$. If the  group $\spGal(\cQ_S/\Kqp)$ is $\sp$-constant then $u$ and all its transforms with respect to $\sp$ are algebraically dependent over $\Kqp$. By the above,  there exist ${c \in \C^{\times}}$, $m\in \Z$ and $b(z) \in \C(z)^{\times}$ such that $a(z)=cz^m\frac{b(qz)}{b(z)}$. 
However Proposition~\ref{propo:caracgalspisomono} states that   there exists $h(z) \in \Kqp$ such that $\sp(a)/a =\sq(h)/h$. An easy computation shows that $m=0$.
\end{proof}

\section{Discrete projective isomonodromy}\label{sec7}

The following proposition allows to characterize the $\sp$-Galois group of a $q$-difference system with large difference Galois group. The notion of large difference Galois group will be made more precise in the proposition, that we will apply later for the groups $\SL_{n}(\C)$ (when $n\geq 2$), $\mathrm{SO}_{n}(\C)$ (when $n\geq 3$) and $\mathrm{Sp}_{n}(\C)$ (when $n$ is even).

\begin{propo}\label{propo:hypertransgdetcontainsSln}
Let $A \in \GL_n(\C(z))$.  Let $G$ be the difference Galois group of $\sq(Y)=AY$ over the $\sq$-field $\C(z^*)$. Assume that its derived subgroup  $G^{ der}$ is an irreducible  most simple algebraic subgroup of $\GL_{n}(\C)$ and has toric constant centralizer (see Definition \ref{defi:spconstantcentralizer}).
Let $\cQ_S$ be a $\ssp$-Picard-Vessiot extension of $\sq(Y)=AY$ over $\Kqp$ and assume that 
$\cQ_S$ is a field.

Then, we have the following alternative:
\begin{enumerate}
\item there exist $d\in \mathbb{N}^{\times}$ and a regular  $\sp$-field extension $\widetilde{\C}$ of $\C$ such that  $\spGal(\cQ_S/\Kqp)_{\widetilde{\C}} \subset \GL_{n,\widetilde{\C}}$ is conjugate to a  $\sp$-group $H$ such that, for all ${B \in \Alg_{\widetilde{\C},\sp}}$ and $g \in H(B)$, there exists 
$\lambda_g \in B^\times$ such that $\sp^d(g)=\lambda_g g$; 
\item $\spGal(\cQ_S/\Kqp)$ contains   $\bold{G^{der}}$, the $\sp$-algebraic group associated to $G^{der}$, see Proposition \ref{propo:algschelmediffschemezarclosure}.
\end{enumerate}
Moreover, if the first case holds then there exist  $\widetilde{U} \in \GL_n( \widetilde{\cQ_S})$, with $\widetilde{\cQ_S}$ the fraction field of $\cQ_S \otimes_{\C} \widetilde{\C}$, see Lemma \ref{lem:extconst2}, a fundamental matrix of solutions, $d\in \mathbb{N}^{\times}$ and $B \in \GL_n( \widetilde{\Kqp})$, with $\widetilde{\Kqp}$ the fraction field of $\Kqp \otimes_{\C} \widetilde{\C}$, $g \in \widetilde{\cQ_S}^{\times}$,  such that 
\begin{equation}\label{gen integ hypertransclos}
\sp^d(\widetilde{U})=gB\widetilde{U}.
\end{equation}
\end{propo}
\begin{rem}
Since, we have assumed that the $\ssp$-Picard-Vessiot extension is a field the difference Galois group $G$ is connected by Lemma \ref{lem:compgroupeetgroupconnected} and $\Kqp$ is relatively algebraically closed in $\cQ_S$ by Lemma \ref{lem:relativalgclosedbasefieldPPVfield}.
\end{rem}

\begin{proof}[Proof of Proposition \ref{propo:hypertransgdetcontainsSln}]
 Since $\spGal(\cQ_S/\Kqp)$ is Zariski-dense in $ \bold{G}$, we find that the derived group  $\mathcal{D}(\spGal(\cQ_S/\Kqp))$  of $\spGal(\cQ_S/\Kqp)$ is Zariski-dense in $\bold{G^{ der}}$ by Proposition \ref{propo:closurederivedgroup}. Since  $\Kqp$ is relatively algebraically closed in $\cQ_S$, straightforward  analogues of \cite[Lemma~6.3]{DVHaWib2} and \cite[Proposition 4.3, (iii)]{DVHaWib1} show  that   the $\sp$-algebraic group 
$\spGal(\cQ_S/\Kqp)$ is  absolutely $\sp$-integral, see Definition~\ref{defi:absolutelysintegr}.   By Lemma~\ref{lem:derivedgroupabsintegral}, the $\sp$-algebraic group  
$\mathcal{D}(\spGal(\cQ_S/\Kqp))$  is  absolutely $\sp$-integral. Since $\C$ is inversive for $\sp$, 
Theorem \ref{theo:caracspgroupalmostsimple} implies the existence of a  $\sp$-field extension $\widetilde{\C}$ of $\C$, such that either $  \mathcal{D}(\spGal(\cQ_S/\Kqp))_{\widetilde{\C}} =\bold{G^{ der}}_{\widetilde{\C}}$, the base change of $\bold{G^{der}}$ to  $\widetilde{\C}$  or there exists  an integer $d\geq 1$ such that $  \mathcal{D}(\spGal(\cQ_S/\Kqp))_{\widetilde{\C}}$ is conjugate to a $\sp^d$-constant subgroup of $\bold{G^{der}}_{\widetilde{\C}}$. 
 The group $\bold{G^{der}_{\widetilde{\C}}}$ is irreducible almost simple and has toric constant centralizer. Since  $\mathcal{D}(\spGal(\cQ_S/\Kqp))_{\widetilde{\C}}$ is a normal subgroup of $\spGal(\cQ_S/\Kqp)_{\widetilde{\C}}$, Lemma \ref{lem5} ensures that $\spGal(\cQ_S/\Kqp)_{\widetilde{\C}}$ either  contains  $\bold{G^{der}}_{\widetilde{\C}}$ or is  conjugate to a $\sp$-algebraic group  $H$ over $\widetilde{\C}$ such that for all $B \in \Alg_{\widetilde{\C},\sp}$ and $g \in H(B)$ there exists 
$\lambda_g \in B^\times$ such that $\sp^d(g)=\lambda_g g$.

We shall prove that if the first case holds then there  there exist   ${\widetilde{U} \in \GL_n(\widetilde{\cQ_S} )}$ a fundamental matrix of solutions,  a positive integer $d$ and $B \in \GL_n( \widetilde{\Kqp})$, ${g \in\widetilde{\cQ_S}^{\times}}$, such that 
\begin{equation}\label{eq4}
\sp^d(\widetilde{U})=gB\widetilde{U}.
\end{equation}

Thus, let us  assume that there exist a positive integer $d$ and a $\sp$-field extension $\widetilde{\C}$ of $\C$ such that $\spGal(\cQ_S/\Kqp)_{\widetilde{\C}}$ is conjugate to a  $\sp$-group $H$ such that, for all $B \in \Alg_{\widetilde{\C},\sp}$ and ${g \in H(B)}$, there exist
$\lambda_g \in  \GL_{1,\widetilde{\C}}(B)$ such that ${\sp^d(g)=\lambda_g g}$. By Lemma \ref{lem:baseextensionforsPV}, we construct a  $\ssp$-Picard-Vessiot extension $\widetilde{\cQ_S}$ for ${\sq(Y)=AY}$ 
over $\widetilde{\C(z)}$ such that ${\spGal(\cQ_S/\Kqp)_{\widetilde{\C}}= \spGal(\widetilde{\cQ_S}/\widetilde{\Kqp})}$. 
By Proposition \ref{propo:repspgaloisgroup}, we can choose   $\widetilde{U} \in \GL_n(\widetilde{\cQ_S})$, a fundamental matrix of solutions, such that for any $\phi \in 
\spGal(\widetilde{\cQ_S}/\widetilde{\Kqp}) (B)$, we have ${\sp^d([\phi]_{\widetilde{U}})=\lambda_\phi [\phi]_{\widetilde{U}}}$ and $\lambda_\phi \in \GL_{1}(B)$. Then, for any ${\phi \in 
\spGal(\widetilde{\cQ_S}/\widetilde{\Kqp}) (B)}$, we have ${\phi( \sp^d({\widetilde{U}}).{\widetilde{U}}^{-1})=\lambda_\phi \sp^d({\widetilde{U}}).{\widetilde{U}}^{-1}}$. Let $g$
be a non-zero entry of $  \sp^d({\widetilde{U}}).{\widetilde{U}}^{-1}$. It is easy to see that 
 the matrix ${B=\frac{1}{g}\sp^d({\widetilde{U}}).{\widetilde{U}}^{-1} \in \GL_n(\widetilde{\cQ_S})}$ is fixed by $\spGal(\widetilde{\cQ_S}/\widetilde{\Kqp})$.
 By Proposition~\ref{propo:spgaloiscorresptransdeg},  $B \in \GL_n(\widetilde{\Kqp})$.
 \end{proof}

\section{$q$-difference equations with power series solutions}\label{sec8}

Let $A \in \GL_n(\C(z))$. We recall that $q,\qpr\in \C^{\times}$,  $|q|,|\qpr|\neq 1$, with $q$ and $\qpr$ multiplicatively independant.
Consider the $q$-difference system 
\begin{equation}\label{eq:systeminitcriteria}
\sq(Y)=AY.
\end{equation}
 The aim of the present section is to 
study the $\sp$-Galois group of \eqref{eq:systeminitcriteria} under the following assumption.
\begin{hypotheses}\label{hypo:genassumption}
Let $G$ be the Galois group of \eqref{eq:systeminitcriteria} over $\Kqp$ and let $\cQ_S$ be a $\sp$-Picard-Vessiot extension for \eqref{eq:systeminitcriteria} over $\Kqp$. Assume that
 \begin{enumerate}
 \item $n\geq 2$.
 \item  $G^{ der}$ is either $\SL_n(\C)$,  $\SO_{n}(\C)$ (when $n\geq 3$) or $\Sp_{n}(\C)$ (when $n$ is even);
 \item there exists a non zero vector solution $Y_0 =(
 u_1 , u_2 ,\dots, u_n)^{t}  \in (\Kqpform)^n$;
 \item  $\cQ_S$ is a field and contains the entries of $Y_{0}$.
 \end{enumerate}
\end{hypotheses}

The study of the $\sp$-Galois group will rely on the combination of two arguments.  The first arguments is  the classification of Zariski dense  $\sp$-algebraic subgroups of almost simple algebraic groups, that essentially says that one has a dichotomy: either the $\sp$-Galois group is large or the solutions of the system satisfy  a  linear $\sp$-equation. The second argument is more analytic and  allows to conclude that the second case can not happen since any power series vector  $Y_0$, solution of a $\sq$ and a $\sp$-linear equation is rational. In contradiction with the simplicity of the difference Galois group.  The analytic argument is a rephrasing of Sch{\"a}fke and Singer \cite{schafke2016consistent}, see also Bezivin and Boutabaa \cite{BeBo}, for an earlier result which is a little weaker, i.e., it is assumed that $|q|,|\tq|$ are multiplicatively independent.
\begin{lem}\label{lem:solcommuneqq1}
Let us consider a non zero vector $u=(u_{1},\dots,u_{n})^{t}$ with coefficients in $\Kqpform^n$ such that $\sq(u)=Au$ for some $A \in \GL_{n}(\Kqp)$. Assume moreover that each $u_{i}$ satisfies some nonzero linear $\qpr$-difference equation with coefficients in $\Kqp$.  Then, the $u_{i}$ actually belong to $\Kqp$.  
\end{lem}

\begin{proof}[Proof of Lemma \ref{lem:solcommuneqq1}]
 One can find $r \in \N^*$ such that $u \in \C((z^{1/r}))$, ${A \in \GL_n(\C(z^{1/r}))}$ and the $\sp$-equation satisfied by the $u_i$'s has coefficients in $\C(z^{1/r})$. Setting, $x =z^{1/r}$ and replacing $q$ (resp. $\tq$)  by $q_r$ (resp. by $\tq_r$) as defined in $\S \ref{subsec:paramdiffgaldiscrete}$, we see that it is sufficient to prove the lemma for $r=1$.

Since $u=(u_{1},\dots,u_{n})^{t}$ has coefficients in $\C((z))^n$, and any entry of $u$ satisfies some nonzero linear $\qpr$-difference equation with coefficients in $\C(z)$, according to the cyclic vector lemma, there exists $P \in \GL_{n}(\C(z))$ such that $Pu = (f,\sq(f),\dots,\sq^{n-1}(f))^{t}$ for some $f \in \C((z))$ which is a solution of a nonzero linear $q$-difference equation, {\it i.e.}, a $\sq$-difference equation, of order $n$ with coefficients in $\C(z)$. Moreover, $f$ satisfies a nonzero linear $\sp$-equation with coefficients in $ \C(z)$, because it is a $\C(z)$-linear combination of the $u_{i}$ and the $u_{i}$ themselves satisfy such equations. It follows from \cite[Corollary~15]{schafke2016consistent}, see also  \cite[Remark 7.5]{BeBo}, that $f$ belongs to $\C(z)$. Hence, the entries of $u=P^{-1}(Pu)=P^{-1}(f,\sq(f),\dots,\sq^{n-1}(f))^{t}$ actually belong to $\C(z)$, as expected.    
\end{proof}

We now split our study depending whether the determinant of the $\sp$-Galois group of \eqref{eq:systeminitcriteria} over $\C(z^{*})$ is a strict subgroup of $\bold{\GL_{1,\C}}$ or not. Since the latter is equal to the $\sp$-Galois group of the order one equation ${\sigma_{q}Y=\det(A)Y}$, following Proposition \ref{propo:sprang1} it is a strict subgroup of $\bold{\GL_{1,\C}}$ if and only if there exist $b \in  \C(z)^{\times} $, $c\in \C^{\times}$ and $m \in \Z$ such that $\det(A)=cz^m\frac{b(qz)}{b(z)}$. Let us first consider this situation in Theorem \ref{theo:sphyperalgdetSLn}. See Theorem \ref{theo:sphypertrgdetSLn} for the other case.
\subsection{$\sp$-algebraic determinant group}
 The goal of the subsection is to prove:
\begin{theo}\label{theo:sphyperalgdetSLn}
Assume that the hypothesis \ref{hypo:genassumption} holds and  that there exist $b \in  \C(z)^{\times} $, $c\in \C^{\times}$, and $m \in \Z$, such that $\det(A)=cz^m\frac{b(qz)}{b(z)}$. 

 Then, the $\sp$-Galois group $\spGal(\cQ_S/\Kqp)$ contains $\bold{G^{der}}$.\par 
\end{theo}

\begin{proof}[Proof of Theorem \ref{theo:sphyperalgdetSLn}]
  Let $\cQ_S$ be a $\ssp$-Picard-vessiot extension of ${\sq(Y)=AY}$ over $\Kqp$ as in Assumption \ref{hypo:genassumption}. Since $Y_0=(u_1,\dots,u_n)^{t} \in \cQ_S^{\times}$, there exists a fundamental matrix of solutions $U \in \GL_n(\cQ_S)$ whose first column is precisely $Y_0$.
We let $G$ denote the difference Galois group of $\sq (Y)=AY$ over the field $\Kqp$, and we let  $\spGal(\cQ_S/\Kqp)$ be the $\sp$-Galois group over the $\ssp$-field $\Kqp$. By assumption, $G^{der}$ is either $\SL_n(\C)$ (when $n\geq 2$),  $\SO_{n}(\C)$ (when $n\geq 3$) or $\Sp_{n}(\C)$ (when $n$ is even). 
By Proposition \ref{propo:hypertransgdetcontainsSln}, 
we have the following alternative:
\begin{enumerate}
\item there exists a positive integer $d$ and a regular  $\ssp$-field extension $\widetilde{\C}$ of $\C$ such that $\spGal(\cQ_S/\Kqp)_{\widetilde{\C}}$ is conjugate to a  $\sp^d$-constant subgroup of  $G^{ der}_{\widetilde{\C}}$; 
\item $\spGal(\cQ_S/\Kqp)$ contains $\bold{G^{ der}}$.
\end{enumerate}
Moreover, if the first case holds, then there exist  $\widetilde{U} \in \GL_n( \widetilde{\cQ_S})$, with $\widetilde{\cQ_S}$ the fraction field of $\cQ_S \otimes_{\C} \widetilde{\C}$, a fundamental matrix of solutions,  a positive integer $d$ and $B \in \GL_n( \widetilde{\Kqp})$, with $\widetilde{\Kqp}$, $g \in \widetilde{\cQ_S}^{\times}$, such that 
\begin{equation}\label{gen integ hypertransclos2}
\sp^d(\widetilde{U})=gB\widetilde{U}.
\end{equation}

We claim that the first case can not hold. Suppose to the contrary that there exists a regular  $\sp$-field extension $\widetilde{\C}$ of $\C$  such that there exist  $\widetilde{U} \in \GL_n( \widetilde{\cQ_S})$ a fundamental matrix of solutions,  a positive integer $d$, $B \in \GL_n( \widetilde{\Kqp})$ and $g \in \widetilde{\cQ_S}^{\times}$, such that  \eqref{gen integ hypertransclos2} holds.  This means that there exists ${C \in \GL_n(\widetilde{\cQ_S}^{\sq})=\GL_n(\widetilde{\C})}$ such that $\widetilde{U} =UC$, and therefore $$\sp^d(U)=gBUC\sp^{-d}(C).$$
This formula implies that the (finite dimensional) $\widetilde{\Kqp}\langle g\rangle_{\sp}$-vector space generated by the entries of $U$ is stable by $\sp^d$. In particular, any $u_{i}$ (recall that the $u_{i}$ are the entries of the first column $Y_0$ of $U$) satisfies a nonzero linear $\qpr$-equation $\cL_i(y)$ with coefficients in $\widetilde{\Kqp}\langle g\rangle_{\sp}$. We claim that $\widetilde{\Kqp}\langle g\rangle_{\sp}=\widetilde{\Kqp}(g)$. Indeed,  we have  $\sp^d(\widetilde{U}) =gB \widetilde{U}$ and $\sq(\widetilde{U})=A\widetilde{U}$. Thus, 
\begin{multline*}
\sq \left( \frac{\sp^d(\det(\widetilde{U}))}{\det(\widetilde{U}) }\right)= \left( \frac{\sp^d(\sq\det(\widetilde{U}))}{\sq(\det(\widetilde{U})) }\right)\\
= \left( \frac{\sp^d(\det(A))}{\det(A) }\right)\frac{\sp^d(\det(\widetilde{U}))}{\det(\widetilde{U}) }= \left( \frac{\sp^d(cz^m\frac{b(qz)}{b(z)})}{cz^m\frac{b(qz)}{b(z)} }\right)\frac{\sp^d(\det(\widetilde{U}))}{\det(\widetilde{U}) }\\
={\qpr}^{md}\frac{\sq(\frac{\sp^d(b(z))}{b(z)})}{\frac{\sp^d(b(z))}{b(z)}}  \frac{\sp^d(\det(\widetilde{U}))}{\det(\widetilde{U}) } ={\qpr}^{md}\frac{\sq(h)}{h}  \frac{\sp^d(\det(\widetilde{U}))}{\det(\widetilde{U}) },
\end{multline*}
where $h(z)=\frac{\sp^d(b(z))}{b(z)}$. Using $\sp^d(\det(\widetilde{U})) =g^{n}\det(B) \det(\widetilde{U})$ in the equality
$$
\sq \left( \frac{\sp^d(\det(\widetilde{U}))}{\det(\widetilde{U}) }\right)={\qpr}^{md}\frac{\sq(h)}{h}  \frac{\sp^d(\det(\widetilde{U}))}{\det(\widetilde{U}) },
$$ allows us to deduce that $\sq \left( g^{n}\det (B)\right)={\qpr}^{md}\frac{\sq(h)}{h}   g^{n}\det (B)$.
Thus, we have 
 $\sq(g^nl)={\qpr}^{md}g^nl$ with ${l=\det(B)/h \in \widetilde{\Kqp}}$. 
Hence we have ${\sq(\sp(g^nl))={\qpr}^{md}\sp(g^nl)}$. Therefore, there exists $c \in \widetilde{\C}^\times$ such that $\sp(g^n l)= c g^n l$. 
Then $\left( \frac{\sp(g)}{g} \right)^n \in \widetilde{\Kqp}$. As we may see in the proof of Proposition \ref{propo:hypertransgdetcontainsSln}, $\spGal(\cQ_S/\Kqp)$ is absolutely $\sp$-integral. By Definition~\ref{defi:absolutelysintegr}, it follows that
$\spGal(\widetilde{\cQ_S}/\widetilde{\Kqp}) =\spGal( \cQ_S/\Kqp)_{\widetilde{\C}}$ is absolutely $\sp$-integral too. By \cite[Proposition~4.3,~(iii)]{DVHaWib1}, the field extension $\widetilde{\C(z^*)} \subset \widetilde{\cQ_S}$ is
$\sp$-regular in the sense of \cite[Definition~4.1]{DVHaWib1}. In particular it is a regular extension and since we are in characteristic zero,
\cite[Proposition in A.V.143]{Bourbaki} proves that $\widetilde{\C(z^*)}$ is relatively algebraically closed in $\widetilde{\cQ_S}$.
Thus, $\frac{\sp(g)}{g} \in \widetilde{\Kqp}$ and $\widetilde{\Kqp}\langle g\rangle_{\sp}=\widetilde{\Kqp}(g)$.
We claim that any $u_i$ satisfies a nonzero linear $\qpr$-equation with coefficients in $\widetilde{\Kqp}$. If $g \in \widetilde{\cQ_S}$ is   algebraic over $\widetilde{\Kqp}$ then $g \in \widetilde{\Kqp}$, because
$\widetilde{\Kqp}$ is relatively algebraically closed in  $\widetilde{\cQ_S}$. In that case, the claim is obvious. Thus, let us assume that  $g$ is transcendental over $\widetilde{\Kqp}$. If  $m=0$ then  $\sq(g^nl)=g^nl$ and thus $g^n \in \widetilde{\Kqp}$. A contradiction with $g$ transcendental over $\Kqp$. Let us write the equation $\cL_i(y)=0$ as $\sum_{j=0}^\nu \cL_{i,j}(y) g^j=0$ where the $\cL_{i,j}(y)$
are linear $\sp$-operators with coefficients in $\widetilde{\Kqp}$, not all zero. To prove our claim, it is sufficient to 
show that $g$ is transcendental over $\widetilde{\Kqp}\{u_1,\dots,u_n\}_{\sp}$. It is also sufficient to prove that $g^n$ is transcendental over $\widetilde{\Kqp}\{u_1,\dots,u_n\}_{\sp}$.   Assume that there exists
a non zero relation
\begin{equation}\label{eq:liaisong}
\sum_{k=0}^\kappa a_k g^{nk}=0,
\end{equation}
where $\kappa>1$ and $a_0,\dots,a_{\kappa-1}, a_\kappa=1 \in \widetilde{\Kqp}\langle u_1,\dots,u_n\rangle_{\sp}$ and $\kappa$ is minimal.  We recall that $\sq(g^n)=g^n{\qpr}^{md}l/\sq(l)$. Applying $\sq$ to \eqref{eq:liaisong} and subtracting ${\qpr}^{md\kappa}\frac{l^\kappa}{\sq(l^\kappa)} *$ \eqref{eq:liaisong}, we find a smaller liaison of the form
$$
\sum_{k=0}^{\kappa-1} (\sq(a_k/l^{k-\kappa}) -  {\qpr}^{md(\kappa-k)}a_k/l^{k-\kappa})  g^{nk}=0.
$$
Thus, for all $k=0,\dots,\kappa-1$, we have  $\sq(a_k/l^{k-\kappa}) -{\qpr}^{md(\kappa-k)}a_k/l^{k-\kappa} =0$. Let us  state and prove a technical lemma. 
\begin{lem}\label{lem1}
Let us fix $r\in \N^\times$. Then, the  equation $\sq(y)={\qpr}^{mdr}y$  has no non zero solution in $\widetilde{\Kqp}\langle u_1,\dots,u_n\rangle_{\sp}$. 
\end{lem}
\begin{proof}[Proof of Lemma \ref{lem1}]
We have $\widetilde{\Kqp}\langle u_1,\dots,u_n\rangle_{\sp}\subset \widetilde{\C((z^*))}$, the fraction field of $\widetilde{\C}\otimes_{\C} \C((z^*))$. Suppose to the contrary that the equation has a non zero solution in $ \widetilde{\C((z))}$. Once again, replacing $q$ and $\tq$ by some suitable roots, it suffices to  prove Lemma \ref{lem1} in the case where the variable $z$ is non ramified.

By Lemma \ref{lem:descentsolutions}, we can find a non zero solution $f$ in $\C((z))$
Let ${f=\sum_{\ell=\nu}^{\infty} y_{\ell}z^{\ell}}$ with $y_{\nu}\neq 0$ a non zero solution of $\sq(y)={\qpr}^{mdr}y$. Taking the $z^{\nu}$ coefficients of the two sides of $\sq(y)={\qpr}^{mdr}y$, we find $\sq(y_{\nu})q^{\nu}={\qpr}^{mdr} y_{\nu}$. Since $y_\nu \in \C$,
$$
y_\nu q^{\nu}={\qpr}^{mdr}y_\nu.$$

With $q$ and $\qpr$ are multiplicatively independent, one should have ${\nu=mdr=0}$. We recall that $m\neq 0$, so $mdr\neq 0$. Consequently, we find a contradiction and this proves that the  equation $\sq(y)={\qpr}^{mdr}y$  has no non zero solution in $\widetilde{\Kqp}\langle u_1,\dots,u_n\rangle_{\sp}$.\end{proof}

Let us finish the proof of Theorem \ref{theo:sphyperalgdetSLn}. In virtue of Lemma \ref{lem1}, for all $k\in\{0,\dots,\kappa-1\}$, the  equation $\sq(y)={\qpr}^{md(\kappa-k)}y$  has no non zero solution
in  $\widetilde{\Kqp}\{ u_1,\dots,u_n\}_{\sp}$. 
Hence, $g^{n\kappa}=0$. This is a contradiction with the fact that $g$ is transcendental over $\widetilde{\Kqp}$ and proves our claim.\par

Therefore, the $u_i$ satisfy a non zero linear $\sp$-equation over $\widetilde{\Kqp}$. Since $\C$ is algebraically closed and $u_{i}\in \C((z^*))$, a descent argument shows that  the $u_i$ satisfy a non zero linear $\sp$-equation over $\C(z^*)$.

It follows from Lemma~\ref{lem:solcommuneqq1} that the $u_{i}$ belong to $\C(z^*)$. Hence, the first column of $U$ is fixed by the Galois group $G$ and this contradicts the hypothesis \ref{hypo:genassumption}, second point.
Therefore, $\spGal(\cQ_S/\Kqp)$ contains $\bold{G^{ der}}$.\par 

\end{proof}

\subsection{$\sp$-transcendental determinant}
Let us recall that the $\sp$-Galois group of $\sq (y)= \det(A) y$ over  $\Kqp$ is a strict subgroup of the multiplicative group $\bold{\GL_{1,\C}}$ if and only if there exist $b \in  \C(z)^{\times}$, $m\in \Z$, and $c\in \C^{\times}$, such that ${\det(A)=cz^{m}\frac{b(qz)}{b(z)}}$.\par

The goal of the subsection is to prove:
\begin{theo}\label{theo:sphypertrgdetSLn}
Assume that the hypothesis \ref{hypo:genassumption} holds and the $\sp$-Galois group of ${\sq (y)= \det(A) y}$ over  $\Kqp$ equals  $\bold{\GL_{1,\C}}$. Recall that the vector ${Y_0=(u_1,u_2,\dots,u_{n})^t \in (\Kqpform)^n}$ is a non zero vector solution of \eqref{eq:systeminitcriteria}. Then, at least one of the $u_i$ is $\sp$-transcendental over $\Kqp$.  
\end{theo}
  
We start by a technical lemma. 

 \begin{lem}\label{lem4}
Let $L$ be a $\sp$-field and let $L\langle a \rangle_{\sp}$ and $L\langle b_1,\dots,b_n \rangle_{\sp}$ be two $\sp$-field extensions of $L$, both contained in a same $\sp$-field extension of $L$. Assume
 that $a$ is $\sp$-transcendental over $L$ and that any $b_i$ is $\sp$-algebraic over $L$. Then, the 
 field extensions $L\langle a \rangle_{\sp}$ and $L\langle b_1,\dots,b_n \rangle_{\sp}$ are linearly disjoint over $L$.
  \end{lem}

\begin{proof}[Proof of Lemma \ref{lem4}]
To the contrary, suppose that $L\langle a \rangle_{\sp}$ and $L\langle b_1,\dots,b_n \rangle_{\sp}$ are not  linearly disjoint over $L$. Then $a$ is $\sp$-algebraic over $L\langle b_1,\dots,b_n \rangle_{\sp}$. This implies that the $\sp$-transcendence degree of the field $L\langle a, b_1,\dots,b_n \rangle_{\sp}$ over $L\langle b_1,\dots,b_n \rangle_{\sp}$ is zero. Since the $\sp$-transcendence degree of $L\langle b_1,\dots,b_n \rangle_{\sp}$ over $L$ is also zero, by hypothesis,
we find that the $\sp$-transcendence degree of $L\langle a, b_1,\dots,b_n \rangle_{\sp}$ over $L$ is zero by classical properties
of the transcendence degree. This implies that $a$ is $\sp$-algebraic over $L$ and yields a contradiction. 
\end{proof}

\begin{proof}[Proof of Theorem \ref{theo:sphypertrgdetSLn}]

We let $G$ denotes the difference  Galois group of $\sq (Y)=AY$ over the field $\Kqp$, and we let $\spGal(\cQ_S/\Kqp)$ denote its $\sp$-Galois group over the $\ssp$-field $\Kqp$. By assumption, the  $\sp$-Galois group of $\sq (y)= \det(A) y$ over  $\Kqp$ equals $\bold{\GL_{1,\C}}$. \par 
 We claim that at least one of the $u_i$  is $\sp$-transcendental  over $\Kqp$. Suppose to  the  contrary that all  of them are $\sp$-algebraic. In virtue of the results of Section~$\ref{sec53}$, the second case of Proposition~\ref{propo:hypertransgdetcontainsSln} can not hold. Then, 
 there exist  a  regular $\sp$-field extension $\widetilde{\C}$ of $\C$ and   $\widetilde{U} \in \GL_n(\widetilde{\cQ_S} )$ a fundamental matrix of solutions,  a positive integer $d$, $g \in \widetilde{ \cQ_S}^{\times}$, and ${B \in \GL_n( \widetilde{\Kqp})}$,  such that 
$$
\sp^d( \widetilde{U}  )=gB\widetilde{U}. 
$$

But $\widetilde{U} =UC$, for some $C \in \GL_n(\widetilde{\cQ_S}^{\sq})=\GL_n( \widetilde{\C})$. Therefore, 
\begin{equation}\label{eq1}
\sp^d(U)=gBUC\sp^{-d}(C).
\end{equation}
This shows that the  $\widetilde{\C(z)}\langle g \rangle_{\sp}$-vector subspace of $\widetilde{\cQ_S}$ generated by the entries of $U$ and all their successive $\sp$-transforms is of finite dimension. In particular, any $u_{i}$ satisfies a nonzero linear $\sp$-equation $\mathcal{L}_i(y)=0$ with coefficients in $\widetilde{\C(z)}\langle g \rangle_{\sp}$. We can assume that the coefficients of $\mathcal{L}_i(y)$ belong to $\widetilde{\Kqp}\{ g\}_{\sp}$. We write $\mathcal{L}_i(y) =\sum_\alpha L_{i,\alpha}(y)g_\alpha$ where $L_{i,\alpha}(y)$ is a linear $\sp$-operator with coefficients in $\widetilde{\Kqp}$,
and $g_\alpha$ is a monomial in the $\sp^i(g)$'s. 

We recall that the $\sp$-Galois group of $\sq (y)= \det(A) y$ over  $\Kqp$ equals $\bold{\GL_{1,\C}}$. In virtue of Proposition \ref{propo:sprang1}, $\det (U)$ is $\sp$-transcendental over $\Kqp$. With \eqref{eq1}, $g^n=\lambda \frac{\sp^d(\det(U))}{\det(U)}$ for some non zero $\lambda \in \widetilde{\Kqp}$. Thus, $g$ is $\sp$-transcendental over $\widetilde{\Kqp}$.

 By Lemma~\ref{lem4},  the $\sp$-fields $\widetilde{\C(z^*)}\langle g\rangle_{\sp}$ and $\widetilde{\C(z^*)}\langle u_{1},\dots ,u_{n}\rangle_{\sp}$ are linearly disjoint over $\widetilde{\Kqp}$. Since $u_{i}$ satisfies a nonzero linear $\sp$-equation $\mathcal{L}_i(y)=0$ with coefficients in $\widetilde{\C(z)}\langle g \rangle_{\sp}$ it follows  that there exists some non zero $L_{i,\alpha}(y)$ such that $L_{i,\alpha}(u_i)=0$. Therefore, the  element  $u_i$ satisfies a non zero linear $\sp$-equation over $\widetilde{\Kqp}$. Since $\C$ is algebraically closed and $u_{i}\in \Kqpform$, a descent argument shows that  the  element  $u_i$ satisfies a non zero linear $\sp$-equation over $\Kqp$.
It follows from Lemma~\ref{lem:solcommuneqq1} that the  element  $u_{i}$ belongs to $\Kqp$. Hence, the first column of $U$ is fixed by the difference Galois group $G$ and this contradicts the hypothesis~\ref{hypo:genassumption}. 
 \end{proof}

\section{Applications}\label{sec9}

\subsection{User friendly criteria for $\sp$-transcendence} 
The goal of this subsection is to use the results of Section $ \ref{sec8}$ in order to give transcendence criterias. We refer to Section $ \ref{sec8}$ for the notations used in this section. 
\begin{coro}\label{coro2}
Let $A \in \GL_n(\C(z))$ and let $G$ be the difference Galois group of the $q$-difference system $\sq(Y)=AY$ over the $\sq$-field $\C(z)$.  {Assume that one of the following holds 
\begin{itemize}
\item    $n\geq 2$ and $G^{\circ, der}=\SL_{n}(\C)$;
\item  $n\geq 3$ and $G^{\circ, der}=\SO_{n}(\C)$;
\item  $n$ is even and $G^{\circ, der}=\Sp_{n}(\C)$.
\end{itemize}}
 Assume that  ${Y_0=(u_1,u_2,\dots,u_{n})^t \in (\Kqpform)^n}$ is a non zero vector solution of $\sq(Y)=AY$.
Then, at least one of the  $u_i$ is $\sp$-transcendental  over $\Kqp$.

\end{coro}

\begin{proof}
We first make the following simple remark. Given  a $\sp$-field extension $K \subset L$ and $ f \in L$. If $f$ is $\sp^s$-transcendental over $K$ then $f$ is $\sp$-transcendental over $K$.  Indeed if $f$ were $\sp$-algebraic over $K$ then the transcendence degree of $K\langle f \rangle_{\sp}$ over $K$ would be finite. Since $K \subset K\langle f \rangle_{\sp^s} \subset K\langle f \rangle_{\sp}$, the transcendence degree of $K\langle f \rangle_{\sp^s}$ over $K$  would be finite and $f$ would be $\sp^s$-algebraic over $K$. A contradiction.

Then, it suffices to prove that at least one of the $u_i$ is $\sp^s$-transcendental over $K$ for some positive integer $s$. Since $\Kqpform^{\sq}=\C$, Proposition \ref{propo:existenceparamfieldsolutionqdiff} proves that there exist some positive integer  $r,s$  and a $(\sq^{r},\sp^s)$-Picard-Vessiot extension $\cQ_S$ for the system $\sq^{r}(Y)=\sq^{r-1}(A) \hdots \sq(A) A Y=A_{r}Y$ over $\Kqp$ such that 
\begin{itemize}
\item $\cQ_S$ is a field;
\item $\Kqp$ is relatively algebraically closed in $\cQ_S$;
\item $Y_0 \in \cQ_S$ is a solution vector of $\sq^{r}(Y)=A_{r} Y$;
\item the difference Galois group of $\sq^{r}(Y)=A_{r}Y$ over $\Kqp$ equals the connected component of $G$.
\end{itemize}
Replacing $q$ (resp. $\tq$) by $q^{r}$ (resp. by $\tq^s$), one can apply Theorems \ref{theo:sphyperalgdetSLn} and~\ref{theo:sphypertrgdetSLn}  to conclude that at least one of the  $u_i$ is $\sp^s$-transcendental over $\Kqp$ and thereby $\sp$-transcendental by the above remark.
\end{proof}

Similarly, we may prove the following:
\begin{coro}\label{coro4}
Let $G$ be the difference Galois group of the $q$-difference system (\ref{syst generique user-friendly}) over the $\sq$-field $\Kq$.   Assume that one of the following holds 
\begin{itemize}
\item    $n\geq 2$ and $G^{\circ, der}=\SL_{n}(\C)$;
\item  $n\geq 3$ and $G^{\circ, der}=\SO_{n}(\C)$;
\item  $n$ is even and $G^{\circ, der}=\Sp_{n}(\C)$.
\end{itemize}
 Let us assume that (\ref{equa generique user-friendly}) admits a non zero solution $g\in \Kqpform$. Then, $g(z)$  is $\sp$-transcendental  over $\Kqp$.
\end{coro}

\begin{proof}
We apply Corollary \ref{coro2} to the vector $(u_1,\dots,u_n) \in \Kqp^n$ with ${u_i=\sq^{i-1}(g)}$. 
To conclude, we just note that, since $\C(z^*)$ is an inversive $\sq$-field, the element $u_i$ is $\sp$-transcendental over $\C(z^*)$ if and only if $g$ is $\sp$-transcendental over $\C(z^*)$.
\end{proof}
\subsection{Hypergeometric series}

In this section, we follow the notations of  Section~$ \ref{sec42}$. We assume that $0<|q|<1$. Let us fix $n\geq 2$, let us consider ${\underline{a}=(a_{1},\dots,a_{n})\in (q^{\Q})^{n}}$, $\underline{b}=(b_{1},\dots,b_{n})\in (q^{\Q}\setminus q^{-\N})^{n}$, $b_{1}=q$, $\l\in \C^{\times}$.

\begin{coro}
Let us assume that (\ref{eq3}) is irreducible  and not $q$-Kummer induced.  Then  $_{n}\Phi_{n}(\underline{a},\underline{b},\l,q;z)$  is $\sp$-transcendental over $\Kqp$.
\end{coro}

\begin{proof}
 The conclusion is  a direct application of Theorem \ref{theo1} and Corollary~\ref{coro4}.
\end{proof}

We follow the notations of Section  $ \ref{sec43}$. We assume that $0<|q|<1$, $n>s$, $n\geq 2$. Let $\underline{a}=(a_{1},\dots,a_n)\in (q^{\R})^{n}$, $\underline{b}=(b_{1},\dots,b_{s})\in (q^{\R}\setminus q^{-\N})^{s}$, $b_{1}=q$, $\l\in \C^{\times}$, $0<|q|<1$ and consider \eqref{eq3}.

\begin{coro}
For $(i,j)\in \{1,\dots,n\}\times  \{1,\dots,s\}$, let $\a_{i},\b_{j}\in \R$ such that $a_{i}=q^{\a_{i}}$ and $b_{i}=q^{\b_{j}}$.  Assume that for all $(i,j)\in \{1,\dots,n\}\times  \{1,\dots,s\}$, $\a_{i}-\b_{j}\notin \Z$, and that the algebraic group generated by $\mathrm{Diag}(e^{2i\pi \a_{1}},\dots,e^{2i\pi \a_{n}})$ is connected. Then,  $_{n}\Phi_{s}(\underline{a},\underline{b},\l,q;z)$ is $\sp$-transcendental over $\Kqp$.
\end{coro}

\begin{proof}
 The conclusion is  a direct application of Theorem \ref{theo2} and Corollary~\ref{coro4}.
\end{proof}
\appendix
\begin{appendices}

\section{Difference algebraic groups}\label{secA2}
 Let $(\k,\sp)$ be a difference field. 
 We denote by $\Alg_{\k,\sp}$ the category of $\k$-$\sp$-algebras
 and by $\Groups$ the category of groups. We stress out the fact that we do not require $\sp$ to be an automorphism of $\k$, but only an endomorphism of $\k$.
 \begin{defi}\label{defi:sHopfalgebra}
 A $\k$-$\sp$-Hopf algebra $R$ is a $\k$-Hopf-algebra, endowed with a structure of $\k$-$\sp$-algebra,  whose structural maps are $\sp$-morphisms. A $\sp$-Hopf ideal of $R$ is a Hopf ideal, which  is stable under the action of $\sp$.
 \end{defi}
 
We define a $\sp$-algebraic group  over $\k$ as follows.
\begin{defi}\label{defi:sgroupscheme}
A  functor $H$ from the category $\Alg_{\k,\sp}$ to the category of $\Groups$ representable by a $\sp$-finitely generated  $\k$-$\sp$-Hopf algebra $\k\{H\}$  is called a $\sp$-algebraic group. A $\sp$-subgroup  $G$ of $H$ is a   $\sp$-algebraic group over $\k$ such that $G(B) \subset H(B)$ for all $B \in \Alg_{\k,\sp}$. It corresponds to a $\sp$-Hopf ideal $\mathfrak{I}_H$ of $\k\{G\}$ such that $\k\{H\}=\k\{G\}/\mathfrak{I}_H$.
\end{defi}
  
\begin{rem}
We adopt the following convention: if
$G$ is an algebraic group over $\k$, we denote by $\k[G]$ its associated Hopf algebra.

\end{rem}  
  
    The theory of $\sp$-algebraic groups and schemes  was initiated by Wibmer (see for instance \cite{Wibaffinediffgroup}). Many of the terminology for  $\sp$-algebraic schemes is borrowed from the usual terminology of schemes, by adding a straightforward compatibility with the difference operator $\sp$. In order to avoid too many definitions, we chose to refer often  to \cite{DVHaWib1}. However, one has to take care that the $\sp$-geometry is  more subtle, even in the affine case, than the algebraic geometry.

\begin{ex}\label{exa:fundadiffsubgroupGln}
Let  $\k\{X\}_{\sp}$ be  the $\k$-$\sp$-algebra of polynomials in the $n \times n$-matrix $X$ of $\sp$-indeterminates. Localizing $\k\{X\}_\sp$
 with respect to $\det(X)$, we find the $\k$-$\sp$-Hopf algebra $\k\{X,\frac{1}{\det(X)} \}_{\sp}$,that corresponds to the 
 $\sp$-algebraic group  attached to the general linear group  $\bold{\GL_{n,\k}}$. 
\end{ex}
The following proposition shows the connection between  algebraic groups over $\k$ and $\sp$-algebraic groups. 

\begin{propo}[\S A.4 and \S A.5 in \cite{DVHaWib1}] \label{propo:algschelmediffschemezarclosure}
Let $G$ be an algebraic group  over $\k$ represented by the finitely generated $\k$-Hopf algebra $\k[G]$. Let $H$ be a $\sp$-algebraic group  represented by the $\sp$-finitely generated $\k$-$\sp$-Hopf algebra $\k\{H\}$. The following holds.
\begin{itemize}
\item The group functor $\begin{array}{llll}\mathbf{G}:& \Alg_{\k, \sp} &\rightarrow &\Sets\\& B &\mapsto &G(B^\#)\end{array}$, with $B^\#$ the underlying $\k$-algebra of $B$, is representable by a  $\sp$-finitely generated $\k$-$\sp$-Hopf algebra. We call $\mathbf{G}$ the $\sp$-algebraic group attached to $G$. 
\item  We denote by   $H^\#$ the functor $\begin{array}{lll}
\Alg_{\k}& \rightarrow &\Sets\\ B &\mapsto & \Hom_{\Alg_{\k}} ( \k\{H\}^\#, B)
\end{array} $. Then,
$$\Hom(H^\#, G) \simeq \Hom(H,\mathbf{G}).$$

\item Assume that $H$ is a $\sp$-subgroup  of $\mathbf{G}$.  The smallest algebraic group  $\overline{H}$ over $\k$ such that $H^\# \rightarrow G$ factors through $\overline{H} \rightarrow G$ is called the Zariski closure of $H$ in $G$.
\end{itemize}
\end{propo}

\begin{ex}\label{exa:definingideal}
Any $\sp$-subgroup  $H$ of $\bold{\GL_{n,\k}}$
 is entirely determined by a $\sp$-Hopf ideal $\mathfrak{I}_H \subset \k\{X,\frac{1}{det(X)} \}_{\sp}$. The Zariski closure
 of $H$ in $\bold{\GL_{n,\k}}$ is defined by the Hopf ideal $\mathfrak{I}_H \cap \k[X,\frac{1}{det(X)}]$. 
\end{ex}

\begin{defi}\label{defi:baseextension}

Let $G$ be a $\sp$-algebraic group  over $\k$ and let $\widetilde{\k}$ be a $\sp$-field extension of $\k$. The base extension of $G$ to $\widetilde{\k}$ is the functor 
$\begin{array}{lll}\Alg_{\widetilde{\k}, \sp}& \rightarrow &\Sets\\ B &\mapsto &G(B)\end{array}$, where $B$ is viewed as $\k$-$\sp$-algebra. It is represented by the $\widetilde{\k}$-$\sp$-Hopf algebra $\k\{G\}\otimes_{\k}\widetilde{\k}$.
\end{defi}
This allows us to define the $\sp$-analogue of the notion of irreducibility.

\begin{defi}[Definition 4.2 and Lemma A.13 in \cite{DVHaWib1}]\label{defi:absolutelysintegr}
Let $G$ be a $\sp$-algebraic group over $\k$. Let $\widetilde{\k}$ be an algebraically closed, inversive field extension of $\k$.  We say that $G$ is absolutely $\sp$-integral if 
$\widetilde{\k}\{G\}$, the $\sp$-Hopf algebra $\widetilde{\k}\{G\}$ of $G_{\widetilde{\k}}$ is a $\sp$-domain, {\it i.e.}, $\widetilde{\k}\{G\}$ is an integral domain and $\sp$ is injective on $\widetilde{\k}\{G\}$.\end{defi}

\begin{lem}\label{lem:productabsolutelysint}
Let $G$ and $H$ be absolutely $\sp$-integral  $\sp$-algebraic groups over $\k$. Then, the product $G\times H$ is absolutely $\sp$-integral.
\end{lem}
\begin{proof}
Since the product commutes with base extension, we can directly assume that $\k$ is inversive and algebraically closed. Thus $\k\{G\}$ and $\k\{H\}$ are $\sp$-domains. This means for instance that $\k\{H\}$ can be embedded in  a $\sp$-field $L$. Then, 
$\k\{G\}\otimes_{\k}\k\{H\}$ embeds as $\sp$-ring in $\k\{G\}\otimes_{\k} L$. Since $\k$ is inversive and algebraically closed,  \cite[Lemma A.13]{DVHaWib1} shows that $\k\{G\}$ is $\sp$-regular, {\it i.e.}, $\k\{G\}\otimes_{\k} \k'$ is a $\sp$-domain for all 
$\sp$-field extension $\k'$ of $\k$. Thus $\k\{G\}\otimes L$ is a $\sp$-domain and the same holds for $\k\{G\}\otimes\k\{H\} =\k\{G\times H\}$. This ends the proof.
\end{proof}
We would like to classify some $\sp$-subgroups of $\bold{\GL_{n,\k}}$. First, we state a fundamental classification theorem, which is 
a $\sp$-analogue of a result of  Cassidy.

\begin{theo}[Theorem A.25 in \cite{DVHaWib2}] \label{theo:caracspgroupalmostsimple}
Let $\k$ be an algebraically closed, inversive $\sp$-field  of characteristic zero and let $G$ be a $\sp$-integral, $\sp$-algebraic subgroup of $\bold{\GL_{n,\k}}$. Assume that the Zariski closure of $G$ in $\bold{\GL_{n,\k}}$ is an  absolutely almost simple algebraic group, properly containing $G$. Then there exist a $\sp$-field extension $\widetilde{\k}$ of $\k$ and an integer $d\geq 1$ such that $G_{\widetilde{\k}}$ is conjugate to a $\sp^d$-constant subgroup of $\bold{\GL_{n,\widetilde{\k}}}$, i.e., there exists $P \in \GL_n(\widetilde{k})$ such that
$$PGP^{-1}(B) \subset \{g \in \GL_{n,\widetilde{k}}(B) | \sp^d(g)=g \}$$ for all $B \in \Alg_{\widetilde{k},\sp}$.
\end{theo}

We also  have to consider the derived group. In analogy with \cite[\S 10.1]{WaterhouseIntrogroupscheme}, we define the derived
group of a $\sp$-algebraic group  as follows.

\begin{defi}\label{defi:derived group}
Let $G$ be a $\sp$-algebraic group  defined over $\k$ and let $\k\{G \}$ be its $\sp$-Hopf algebra. For any $n \in \N$, we define a  natural transformation $\phi_n$  from  $G^{2n}$ to $G$  as follows. For all $ B \in \Alg_{\k,\sp}$ and $ x_1,\dots,x_n,y_1,\dots,y_n \in G(B)^{2n}$, we set 
$$
   \phi_n(x_1,\dots,x_n,y_1,\dots,y_n)=x_1y_1x_1^{-1}y_1^{-1} \hdots x_ny_nx_n^{-1}y_n^{-1}.$$
Let $\psi_{n,G}: \k\{G\} \rightarrow \otimes^{2n}\k\{G\}$ be the corresponding dual map by Yoneda. Its kernel will be denoted by $\mathfrak{I}_{n,G}$. We will also use the notations $\psi_{n}$ and  $\mathfrak{I}_n$ for $\psi_{n,G}$ and $\mathfrak{I}_{n,G}$ respectively if no confusion is likely to arise.
Let $\mathfrak{I}_{\mathcal{D}(G)}=\cap_{n \in \N} \mathfrak{I}_n$. Then $\mathfrak{I}_{\mathcal{D}(G)}$ is a $\sp$-Hopf ideal of $\k\{G\}$ and we
defined the derived group $\mathcal{D}(G)$ as the $\sp$-algebraic subgroup of $G$ represented by $\k\{G\}/\mathfrak{I}_{\mathcal{D}(G)}$. 
\end{defi}
\begin{proof}
Let $\Delta$ denote the co-multiplication map of $\k\{G\}$. Then, it is clear that $\Delta(\mathfrak{I}_{2n})\subset \mathfrak{I}_n \otimes \mathfrak{I_n}$ since multiplying two products of $n$ commutators yields a product of $2n$ commutators. This shows that 
$\mathfrak{I}_{\mathcal{D}(G)}$ is an Hopf ideal. For all $n \in \N$, the map $\psi_n$ is a $\sp$-morphism so that $\mathfrak{I}_n$ is a $\sp$-ideal. This proves that $\mathfrak{I}_{\mathcal{D}(G)}$ is a $\sp$-ideal. 
\end{proof}

\begin{lem}\label{lem:commutebaseextension}
For any $\sp$-algebraic group  $G$ over $\k$ and any $\sp$-field extension $\widetilde{\k}$ of $\k$, we have 
$\mathcal{D}(G_{\widetilde{\k}})=\mathcal{D}(G)_{\widetilde{\k}}$
\end{lem}
\begin{proof}
The definition of $\mathfrak{I}_{\mathcal{D}(G)}$ commutes with base extension. 
\end{proof}

\begin{propo}\label{propo:closurederivedgroup}
Let $H$ be an algebraic group  over $\k$ and let $G \subset H$ be a Zariski dense $\sp$-algebraic subgroup of $H$. Then, $\mathcal{D}(G)$ is a Zariski dense subgroup of $\mathcal{D}(H)$.
\end{propo}
\begin{proof}
Let $\k\{ \mathbf{H} \}$ be the $\sp$-Hopf algebra of the $\sp$-algebraic group  $\mathbf{H}$ attached to $H$ as in Proposition \ref{propo:algschelmediffschemezarclosure}.  Then, $\k[H]$ is a sub-Hopf algebra of $\k\{\mathbf{H}\}$. This means, in the notation above, that $\psi_n: \k[H] \rightarrow \otimes^{2n}\k[H]$  is the restriction of $\psi_{n,H}: \k\{\mathbf{H}\} \rightarrow \otimes^{2n}\k\{\mathbf{H}\}$. 
Thus, if $I_{\mathcal{D}(H)} \subset \k[H]$ denotes the Hopf ideal of $\mathcal{D}(H)$ in $H$, then $\mathfrak{I}_{\mathcal{D}(\mathbf{H})} \cap \k[H]=I_{\mathcal{D}(H)}$.
Since $G$ is Zariski dense in $H$, the $\sp$-algebraic group $G^n$ is Zariski dense in $H^n$ for any positive integer $n$. For all $n \in \N^{\times}$, we denote by $\mathfrak{I}_{G^n}$ the defining ideal of $G^n$ in $\mathbf{H^n}$. By the above, one has 
$$ \mathfrak{I}_{G^n} \cap k[H^n]=\mathfrak{I}_{G^n} \cap \otimes^n k[H]= \{0\}.$$
We denote by $\pi_n: k\{\mathbf{H^n}\} \rightarrow \k\{G^n\}$ the surjective morphism of $\sp$-Hopf algebras whose kernel is $ \mathfrak{I}_{G^n}$. 
 Since the applications $\psi_{n,G}$ and $\psi_{n,H}$ are constructed using  comultiplication and co-inverse, one finds a commutative diagram  of morphisms of $\sp$-Hopf algebras
$$
  \xymatrix{
 k\{\mathbf{H}\} \ar[r]^{\pi_1} \ar[d]^{\psi_{n,H}} & k\{G\} \ar[d]^{\psi_{n,G}} \\
 \otimes^{2n} k\{\mathbf{H}\} \ar[r]^{\pi_{2n}} & \otimes^{2n} k\{G\}
 }
$$

  Let $\mathfrak{J}_{\mathcal{D}(G)}$ be the defining ideal of $\mathcal{D}(G)$ in $\mathbf{H}$, {\it i.e.}, ${\pi_1}^{-1}(\mathfrak{I}_{\mathcal{D}(G)})$. To prove that $\mathcal{D}(G)$ is Zariski dense in $\mathcal{D}(H)$, we need to show that 
  $$
  \mathfrak{J}_{\mathcal{D}(G)} \cap k[H]= \mathfrak{I}_{\mathcal{D}(\mathbf{H})} \cap \k[H]=I_{\mathcal{D}(H)}.
  $$
Let $x \in    \mathfrak{I}_{\mathcal{D}(\mathbf{H})}$. For any $n \in \N^{\times}$, we have 
$\pi_{2n} \circ \psi_{n,H}(x)=0=\psi_{n,G} \circ \pi_1 (x)$ so that $\pi_1(x) \in \mathfrak{I}_{n,G}$ and $x \in \mathfrak{J}_{\mathcal{D}(G)}$. Thus, 
$$
\mathfrak{I}_{\mathcal{D}(\mathbf{H})} \cap \k[H] \subset  \mathfrak{J}_{\mathcal{D}(G)} \cap k[H].
  $$
 Conversely, let $x \in   \mathfrak{J}_{\mathcal{D}(G)} \cap k[H]$. Then, $\psi_{n,G} \circ \pi_1(x)=0=\pi_{2n} \circ \psi_{n,H}(x)$ so that $\psi_{n,H}(x) \in \mathrm{Ker}(\pi_{2n})= \mathfrak{I}_{G^{2n}}$. Since  $x \in k[H]$, we conclude that $\psi_{n,H}(x) \in k[H^{2n}] \cap  \mathfrak{I}_{G^{2n}} =\{0\}$. Thus, $x \in \mathfrak{I}_{n,H}$ for any $n$ so that 
 $x \in  \mathfrak{I}_{\mathcal{D}(\mathbf{H})}\cap k[H]$. This ends the proof. 
\end{proof}

\begin{lem}\label{lem:derivedgroupabsintegral}
The derived group of an absolutely  $\sp$-integral  $\sp$-algebraic group   $G$ over $\k$ is absolutely $\sp$-integral. 
\end{lem}
\begin{proof}
Since by Lemma \ref{lem:commutebaseextension}, the formation of the derived group commutes with base extension. We can assume that 
$\k$ is algebraically closed and inversive. Since the $\k$-$\sp$-Hopf algebra of $\mathcal{D}(G)$ is $\k\{G\}/\mathfrak{I}_{\mathcal{D}(G)}$, the group $\mathcal{D}(G)$ is absolutely $\sp$-integral if and only if $\mathfrak{I}_{\mathcal{D}(G)}$ is $\sp$-prime, {\it i.e.},  prime and such that $\sp(a) \in \mathfrak{I}_{\mathcal{D}(G)}$ implies $a \in \mathfrak{I}_{\mathcal{D}(G)}$.  By Lemma  
\ref{lem:productabsolutelysint}, we find that for all $n \in \N$, the group $G^{2n}$ is absolutely $\sp$-integral. This means that
$\k\{G^{2n}\}$ is a $\sp$-domain for all $n \in \N$. Since $\mathfrak{I}_n$ is the  kernel of the  $\sp$-morphism 
$\psi_n:\k\{G\} \rightarrow  \k\{G^{2n}\}$, the ideal $\mathfrak{I}_n$ is $\sp$-prime for all $n \in \N$. This implies that $\mathfrak{I}_{\mathcal{D}(G)}$ is $\sp$-prime.  
 \end{proof}

 \begin{defi}\label{defi:spconstantcentralizer}
 Let $(\k,\sp)$ be a $\sp$-field and let $G\subset \bold{\GL_{n,\k}}$ be an algebraic group  defined over $\k$. Let $d \in \N^\times$.  We consider the $\sp$-subgroup $G^{\sp^d}$ of $G$ defined by $G^{\sp^d}(B)= \{g \in G(B) | \sp^d(g)=g \}$ for any $B \in \Alg_{\k,\sp}$.  We say that $G$  has a toric constant centralizer  if, for any $d \in \N^\times$,
 for any $B \in \Alg_{\k,\sp}$, the following holds: if $h \in \GL_{n,\k}(B)$ centralizes $G^{\sp^d}(B)$ then 
 $h =\lambda I_n$ for some $\lambda \in B^\times$.
 \end{defi}

 \begin{lem}\label{lem5}
Let $(\k,\sp)$ be a $\sp$-field and let $G\subset \bold{\GL_{n,\k}}$ be an algebraic group  defined over $\k$.  Assume that $G$ has toric constant centralizer.
  Let  $H$ be a $\sp$-subgroup of $\bold{\GL_{n,\k}}$ such that  $G^{\sp^d}$ is a normal subgroup of $H$, i.e. $G^{\sp^d}(B)$ is a normal subgroup of $H(B)$ for all $B \in \Alg_{\k, \sp}$. Then,   for all $B \in \Alg_{\k,\sp}$ and $g \in H(B)$ there exists 
$\lambda_g \in  B^\times$ such that $\sp^d(g)=\lambda_g g$.
\end{lem}

\begin{proof}[Proof of Lemma \ref{lem5}]

If $g$ normalizes $G^{\sp^d}(B)$, for some $d\in \mathbb{N}^{\times}$, then $\sp^d(g)g^{-1}$ centralizes $G^{\sp^d}(B)$. By assumption, we conclude that $\sp^d(g)g^{-1}$ is a scalar matrix.
\end{proof}

\begin{lem}\label{lem6}
Let $(\k,\sp)$ be a $\sp$-field. The algebraic groups $\bold{\SL_{n,\k}}$ (when $n\geq 2$), $\bold{SO_{n,\k}}$ (when $n\geq 3$) and $\bold{Sp_{n,\k}}$ (when $n$ is even) have toric constant centralizer.
\end{lem}

\begin{proof}
The algebraic groups $\bold{\SL_{n,\k}}$ (when $n\geq 2$), $\bold{SO_{n,\k}}$ (when $n\geq 3$) and $\bold{Sp_{n,\k}}$ (when $n$ is even) are absolutely almost simple algebraic group.  Let $d \in \N^\times$ and let $B \in \Alg_{\k,\sp}$.

Let us consider $\bold{\SL_{n,\k}}$ with $n\geq 2$. Let $M \in \GL_{n,\k}(B)$  that centralizes $\SL_n^{\sp^d}(B)$. For $i \neq j$, the matrices $X_{i,j}= I_n +E_{i,j}$,
where $E_{i,j}$ are matrices with zeros at every entry except $1$ at row $i$ and column $j$, belong to $\SL_{n,\k}^{\sp^d}(B)$ for all $B \in \Alg_{\k,\sp}$. Consequently, for all $i \neq j$, $B \in \Alg_{\k,\sp}$, ${MX_{i,j}=X_{i,j}M}$. This shows that $M =\lambda I_n$ for some $\lambda \in B^\times$. \par
 
Let us consider $\bold{SO_{n,\k}}$ with $n\geq 3$. Let $M \in \GL_{n,\k}(B)$  that centralizes $\SO_n^{\sp^d}(B)$. For all $1\leq i<j\leq n$, $B \in \Alg_{\k,\sp}$, 
$M N_{i,j}=N_{i,j}M$, where $N_{i,j}$ is the diagonal matrix with $1$ entry, except the diagonal entries $i$ and $j$ that are equal to $-1$. It follows that $M$ is diagonal. To conclude that $M =\lambda I_n$ for some $\lambda \in B^\times $, we consider the commutation with  $P_{i}=\mathrm{Diag}\left(\mathrm{I}_{i},\begin{pmatrix}0&1\\ -1&0 \end{pmatrix},\mathrm{I}_{n-i-2}\right)$, $i\leq n-2$. \par 

Let us consider $\bold{Sp_{n,\k}}$ with $n$ even. 
Let  $M \in \GL_{n,\k}(B)$   that centralizes $\Sp_n^{\sp^d}(B)$.
 For all $N\in \SL_{n/2,\k}^{\sp^d}(B)$,  $\mathrm{Diag}(N,(N^{-1})^{t})\in 
\mathrm{Sp}_{n,\k}^{\sp^d}(B)$. Then, for all $N\in \SL_{n/2,\k}^{\sp^d}(B)$, we have ${M \mathrm{Diag}(N,(N^{-1})^{t})=\mathrm{Diag}(N,(N^{-1})^{t})M}$. Let $M=\begin{pmatrix}M_{1,1} &M_{1,2} \\M_{2,1} &M_{2,2} \end{pmatrix}$, $M_{i,j}$ are $n/2$  times $n/2$ matrices.  From the commutation relation we obtain ${M_{1,1} N=N M_{1,1}}$. Using the fact that $\bold{\SL_{n/2,\k}}$ has toric constant centralizer, we conclude that $M_{1,1}= \lambda I_{n/2}$ for some $\lambda \in B^\times$. Similarly, we find that  $M_{2,2}= \mu I_{n/2} $ for some $\mu \in B^\times$. Then,  $M N =NM$ with $N=\begin{pmatrix}\mathrm{I}_{n/2}&\mathrm{I}_{n/2}\\ 0&\mathrm{I}_{n/2} \end{pmatrix}\in \mathrm{Sp}_{n,\k}^{\sp^d}(B)$. We obtain $M_{2,1}=0$. Similarly with $N=\begin{pmatrix}\mathrm{I}_{n/2}&0\\ \mathrm{I}_{n/2}&\mathrm{I}_{n/2} \end{pmatrix}\in \mathrm{Sp}_{n,\k}^{\sp^d}(B)$, we obtain $M_{1,2}=0$. Finally, with the commutation of $M$ with $N=\begin{pmatrix}0&\mathrm{I}_{n/2}\\ -\mathrm{I}_{n/2}&0 \end{pmatrix}\in \bold{\mathrm{Sp}_{n,\k}^{\sp^d}}$, we find $M= \lambda I_n$ for some $\lambda \in B^\times$.
\end{proof}
\section{Convergent power series solution of $q$-difference equation}\label{sec:convergentsolution}

Let $\K =\C(\{z\})$ be the field of fraction of the ring of convergent power series $\C\{z\}$.

Let $A\in \GL_n(\K)$. In \cite{SaulFiltration}, the author attaches to a $q$-difference system $\sq(Y)=AY$, a Newton
polygon $N(A)$. The slopes of the non-vertical half-lines defining the border of $N(A)$ are called the slopes of the Newton polygon and ranked in decreasing order as follows $S(A):=\{ \mu_1> \mu_2 \dots >\mu_r \} \subset \Q$. The Newton polygon and the slopes of the $q$-difference system are invariant under formal gauge transforms, i.e., $S(A)=S(\sq(P)AP^{-1})$ and ${N(A)=N(\sq(P)AP^{-1})}$ for any $P\in \GL_n(\K)$. The slopes induces 
a filtration of  the $q$-difference module associated to the $q$-difference system $\sq(Y)=AY$. One has the following proposition:
\begin{propo}[\cite{RSZ}, \S 3.3]\label{propo:convdecomp}
Let $A\in \GL_n(\K)$ and let $S(A):=\{ \mu_1> \mu_2 \dots >\mu_r \}$ be its set of slopes. Assume that $S(A)\subset \Z$.  Then, there exist ${P \in \GL_n(\K)}$,  $A_1,\dots A_r$ some invertible constant matrices and $U_{i,j}$ some matrices with coefficients in $\K$ such that 
$$
\sq(P)AP^{-1} =\begin{pmatrix}
z^{-{\mu_1} }A_1 & & \dots \dots  &\dots & \dots & \dots & U_{1,r}\\
0 & \ddots & \dots & \dots & \dots & \dots  & \vdots \\
\vdots & \ddots & z^{-\mu_i}A_i & \dots & U_{i,j} &\dots  & \vdots \\
\vdots & \dots & 0& \ddots &  \vdots & \dots & \vdots \\
\vdots & \dots & \dots& \ddots &   z^{-\mu_j}A_j& \dots & \vdots \\
\vdots & \dots & \dots& \dots & 0 &\ddots & \vdots \\

 0& \dots &\dots& \dots &\dots  & 0 & z^{- \mu_r} A_r 
\end{pmatrix}.
$$
\end{propo}

\begin{lem}\label{lem:convergentformal solution}
Let $A \in \GL_n(\K)$.We let $l$ to be the least common multiple of the denominators of the slopes $S(A):=\{ \mu_1> \mu_2 \dots >\mu_r \}$ of $A$.
Then, there exist an integer $r$ and a complex number $c \in \C^*$ such that 
the system ${\sq(Y)=cz^{r/\ell} A Y}$ has a non zero vector solution $Y_0\in \C(\{ z^{1/\ell}\})^{n}$. If ${S(A) \subset \Z}$, we may further assume that $Y_0 \in \K^n \cap \cM er(\C^*)$. Moreover if $A$  is fuchsian, i.e., $S(A)=\{0\}$, one can choose 
$r$ to be $0$. 
\end{lem}
\begin{proof}
Assume first that $S(A) \subset \Z$. We know, by Proposition \ref{propo:convdecomp}, one can find $P \in \GL_n(\K)$ and $A_1,\dots A_r$ some invertible constant matrices such that 
\begin{equation}\label{eq:blockdecomp}
\sq(P)AP^{-1} =\begin{pmatrix}
z^{-{\mu_1} }A_1 & & \dots \dots  &\dots & \dots & \dots & U_{1,r}\\
0 & \ddots & \dots & \dots & \dots & \dots  & \vdots \\
\vdots & \ddots & z^{-\mu_i}A_i & \dots & U_{i,j} &\dots  & \vdots \\
\vdots & \dots & 0& \ddots &  \vdots & \dots & \vdots \\
\vdots & \dots & \dots& \ddots &   z^{-\mu_j}A_j& \dots & \vdots \\
\vdots & \dots & \dots& \dots & 0 &\ddots & \vdots \\

 0& \dots &\dots& \dots &\dots  & 0 & z^{- \mu_r} A_r 
\end{pmatrix}.
\end{equation}
One can also assume, up to multiply $P$ by a constant matrix, that $A_1$ is upper triangular. We let $d \in \C^*$ be the coefficient on the first row and column of $A_1$. An easy computation shows that the vector $Z_0:=\begin{pmatrix}
1 \\
0 \\
\vdots\\
0

\end{pmatrix}$ is a solution of the system $\sq(Z)= \frac{z^{\mu_1}}{d}\sq(P)AP^{-1}Z$. Then, the vector $Y_0:=P^{-1}Z_0 \in \K^n$ is a non zero solution of the system $\sq(Y)=\frac{z^{\mu_1}}{d}A Y$. Moreover, one can show, using the fact that $\sq(Y_0)=\frac{z^{\mu_1}}{d}AY_0$ that the vector $Y_0$ defines a meromorphic function on $\C^*$. This proves the result with $r=\mu_{1}$ and $c=d^{-1}$. If  $S(A)=\{0\}$, then $\mu_{1}=0$ and the result follows in this case too.\par

Let us treat the general case. We let $l$ to be the least common multiple of the denominators of the slopes $S(A):=\{ \mu_1> \mu_2 \dots >\mu_r \}$ of $A$. 
By \cite[Theorem~2.2.1]{RSZ}, the variable change  $z\mapsto z^{1/\ell}$ transforms $\sigma_{q}Y=AY$ into a $q^{1/\ell}$-difference equation with integral slopes $\{ \ell \mu_1> \ell\mu_2 \dots >\ell \mu_r \}$. Therefore, appying the integer slopes case, we find the existence of $r$ and a complex number $c \in \C^*$ such that 
the system $\sq(Y)=cz^{r/\ell} A Y$ has a non zero vector solution $Y_0 \in \C(\{ z^{1/\ell}\})^{n}$.
\end{proof}

\end{appendices}
\bibliographystyle{alpha}

\bibliography{biblio}

\def\udot#1{\ifmmode\oalign{$#1$\crcr\hidewidth.\hidewidth
  }\else\oalign{#1\crcr\hidewidth.\hidewidth}\fi} \def\cprime{$'$}
  \def\polhk#1{\setbox0=\hbox{#1}{\ooalign{\hidewidth
  \lower1.5ex\hbox{`}\hidewidth\crcr\unhbox0}}} \def\cprime{$'$}
\begin{thebibliography}{DVHW14b}

\bibitem[BB92]{BeBo}
Jean-Paul B{\'e}zivin and Abdelbaki Boutabaa.
\newblock Sur les \'equations fonctionelles {$p$}-adiques aux
  {$q$}-diff\'erences.
\newblock {\em Collect. Math.}, 43(2):125--140, 1992.

\bibitem[Bou03]{Bourbaki}
Nicolas Bourbaki.
\newblock {\em Algebra {II}. {C}hapters 4--7}.
\newblock Elements of Mathematics (Berlin). Springer-Verlag, Berlin, 2003.
\newblock Translated from the 1981 French edition by P. M. Cohn and J. Howie,
  Reprint of the 1990 English edition [Springer, Berlin; MR1080964
  (91h:00003)].

\bibitem[Cas72]{C72}
Phyllis~Joan Cassidy.
\newblock Differential algebraic groups.
\newblock {\em Amer. J. Math.}, 94:891--954, 1972.

\bibitem[Cas89]{CassidyCSSDAG}
Phyllis~Joan Cassidy.
\newblock The classification of the semisimple differential algebraic groups
  and the linear semisimple differential algebraic {L}ie algebras.
\newblock {\em J. Algebra}, 121(1):169--238, 1989.

\bibitem[CHS08]{CHS}
Zo{\'e} Chatzidakis, Charlotte Hardouin, and Michael~F. Singer.
\newblock On the definitions of difference {G}alois groups.
\newblock In {\em Model theory with applications to algebra and analysis.
  {V}ol. 1}, volume 349 of {\em London Math. Soc. Lecture Note Ser.}, pages
  73--109. Cambridge Univ. Press, Cambridge, 2008.

\bibitem[Coh65]{Cohn:difference}
Richard~M. Cohn.
\newblock {\em Difference algebra}.
\newblock Interscience Publishers John Wiley \& Sons, New York-London-Sydeny,
  1965.

\bibitem[CS07]{CS}
Phyllis~Joan Cassidy and Michael~F. Singer.
\newblock Galois theory of parameterized differential equations and linear
  differential algebraic groups.
\newblock In {\em Differential equations and quantum groups}, volume~9 of {\em
  IRMA Lect. Math. Theor. Phys.}, pages 113--155. Eur. Math. Soc., Z\"urich,
  2007.

\bibitem[DHR18]{DHR}
Thomas Dreyfus, Charlotte Hardouin, and Julien Roques.
\newblock Hypertranscendance of solutions of mahler equations.
\newblock {\em Journal of the European Mathematical Society (JEMS)},
  20(9):2209--2238, 2018.

\bibitem[DVHW14a]{DVHaWib2}
Lucia Di~Vizio, Charlotte Hardouin, and Michael Wibmer.
\newblock Difference algebraic relations among solutions of linear differential
  equations.
\newblock 2014.

\bibitem[DVHW14b]{DVHaWib1}
Lucia Di~Vizio, Charlotte Hardouin, and Michael Wibmer.
\newblock Difference {G}alois theory of linear differential equations.
\newblock {\em Adv. Math.}, 260:1--58, 2014.

\bibitem[Har08]{hardouin2008hypertranscendance}
Charlotte Hardouin.
\newblock Hypertranscendance des systemes aux diff{\'e}rences diagonaux.
\newblock {\em Compositio Mathematica}, 144(3):565--581, 2008.

\bibitem[HS08]{HS}
Charlotte Hardouin and Michael~F. Singer.
\newblock Differential {G}alois theory of linear difference equations.
\newblock {\em Math. Ann.}, 342(2):333--377, 2008.

\bibitem[Kol74]{kolchin1974constrained}
Ellis~R Kolchin.
\newblock Constrained extensions of differential fields.
\newblock {\em Advances in Mathematics}, 12(2):141--170, 1974.

\bibitem[Lev06]{Levin}
Alexander~B. Levin.
\newblock Difference algebra.
\newblock In {\em Handbook of algebra. {V}ol. 4}, volume~4 of {\em Handb.
  Algebr.}, pages 241--334. Elsevier/North-Holland, Amsterdam, 2006.

\bibitem[OW15]{OvWib}
Alexey Ovchinnikov and Michael Wibmer.
\newblock {$\sigma$}-{G}alois theory of linear difference equations.
\newblock {\em Int. Math. Res. Not. IMRN}, (12):3962--4018, 2015.

\bibitem[Ram92]{R92}
Jean-Pierre Ramis.
\newblock About the growth of entire functions solutions of linear algebraic
  {$q$}-difference equations.
\newblock {\em Ann. Fac. Sci. Toulouse Math. (6)}, 1(1):53--94, 1992.

\bibitem[Roq08]{Ro08}
Julien Roques.
\newblock Galois groups of the basic hypergeometric equations.
\newblock {\em Pacific J. Math.}, 235(2):303--322, 2008.

\bibitem[Roq11]{Ro11}
Julien Roques.
\newblock Generalized basic hypergeometric equations.
\newblock {\em Invent. Math.}, 184(3):499--528, 2011.

\bibitem[{Roq}12]{Ro12}
Julien {Roques}.
\newblock {On classical irregular $q$-difference equations.}
\newblock {\em {Compos. Math.}}, 148(5):1624--1644, 2012.

\bibitem[Roq18]{Ro15}
Julien Roques.
\newblock On the algebraic relations between mahler functions.
\newblock {\em Transactions of the American Mathematical Society},
  370(1):321--355, 2018.

\bibitem[RSZ13]{RSZ}
Jean-Pierre Ramis, Jacques Sauloy, and Changgui Zhang.
\newblock Local analytic classification of {$q$}-difference equations.
\newblock {\em Ast\'erisque}, (355):vi+151, 2013.

\bibitem[Sau04]{SaulFiltration}
Jacques Sauloy.
\newblock La filtration canonique par les pentes d'un module aux
  {$q$}-diff\'erences et le gradu\'e associ\'e.
\newblock {\em Ann. Inst. Fourier (Grenoble)}, 54(1):181--210, 2004.

\bibitem[SS16]{schafke2016consistent}
Reinhard Sch{\"a}fke and Michael~F Singer.
\newblock Consistent systems of linear differential and difference equations.
\newblock {\em Journal of the European Mathematical Society (JEMS)}, 2016.

\bibitem[vdPS97]{VdPS97}
Marius van~der Put and Michael~F. Singer.
\newblock {\em Galois theory of difference equations}, volume 1666 of {\em
  Lecture Notes in Mathematics}.
\newblock Springer-Verlag, Berlin, 1997.

\bibitem[vdPS03]{VdPS}
Marius van~der Put and Michael~F. Singer.
\newblock {\em Galois theory of linear differential equations}, volume 328 of
  {\em Grundlehren der Mathematischen Wissenschaften [Fundamental Principles of
  Mathematical Sciences]}.
\newblock Springer-Verlag, Berlin, 2003.

\bibitem[Wat79]{WaterhouseIntrogroupscheme}
William~C. Waterhouse.
\newblock {\em Introduction to affine group schemes}, volume~66 of {\em
  Graduate Texts in Mathematics}.
\newblock Springer-Verlag, New York-Berlin, 1979.

\bibitem[Wib15]{Wibaffinediffgroup}
Michael Wibmer.
\newblock Affine difference algebraic groups.
\newblock 2015.

\end{thebibliography}
\end{document}